\documentclass{amsart}
\usepackage{fullpage}
\usepackage{amsfonts,amsmath,amssymb,parskip,tikz,amsthm,hyperref,enumerate,subcaption,multirow,booktabs,subcaption}
\usepackage[capitalise,nameinlink]{cleveref}
\usetikzlibrary{shapes.geometric} 
\usetikzlibrary{calc,positioning}
\usepackage[dvipsnames]{xcolor}

\newcommand{\quaad}{\quad\quad}
\newcommand{\Sr}{S_{\vec{r}}}
\newcommand{\Z}{\mathbb{Z}}
\newcommand{\C}{\mathrm{C}}
\newcommand{\UC}{\mathrm{UC}}
\newcommand{\D}{\mathrm{D}}
\newcommand{\UD}{\mathrm{UD}}
\newcommand{\St}{\mathrm{S}}
\newcommand{\from}{\!:}

\newcommand{\vertex}[3][1]{
\node [circle,  fill, inner sep = 0pt,outer sep = 0pt, minimum size=#1mm] (#3) at #2 {};
}

\newcommand{\fg}[1]{
    \pgfdeclarelayer{foreground}
    \pgfsetlayers{main,foreground}
    \begin{pgfonlayer}{foreground}
    #1
    \end{pgfonlayer}
    }

\newcommand{\smallvertex}[2]{
\node [circle,  fill, inner sep = 0pt,outer sep = 0pt, minimum size=1mm] (#2) at #1 {};
}

\newcommand{\Ygraph}{
    \fg{\vertex{(0,0)}{B}
    \vertex{(0,-1)}{A}
    \vertex{(-.707,.707)}{C}
    \vertex{(.707,.707)}{D}
    \draw (A) -- (B) -- (C) -- (B) -- (D);
    }
}

\newcommand{\Xgraph}{
    \fg{\vertex{(0,-1.25)}{A}
    \vertex{(0,0)}{B}
    \vertex{(1.25,0)}{C}
    \vertex{(0,1.25)}{D}
    \vertex{(-1.25,0)}{E}
    \draw (A) -- (D);
    \draw (C) -- (E);
    }
}

\definecolor{cbred}{HTML}{DE3220}
\newcommand{\RedRobot}[2]{
\node [rectangle, cbred, fill, inner sep = 0pt,outer sep = 1pt, minimum size=2.3mm] (#2) at #1 {};
}

\definecolor{cbgreen}{HTML}{13C113}
\newcommand{\GreenRobot}[2]{
\node [diamond, cbgreen, fill, inner sep = 0pt,outer sep = 1pt,minimum size=3mm] (#2) at #1 {};
}
\definecolor{cbblue}{HTML}{005AB5}
\newcommand{\BlueRobot}[2]{
\node [isosceles triangle, shape border rotate=90,cbblue, fill, inner sep = 0pt,outer sep = 1pt,minimum size=2.5mm] (#2) at #1 {};
}

\definecolor{cbyellow}{HTML}{FFC100}
\newcommand{\YellowRobot}[2]{
\node [star, cbyellow, fill, inner sep = 0pt,outer sep = 1pt,minimum size=3.2mm] (#2) at #1 {};
}

\newcommand{\PurpleRobot}[2]{
\node [circle, Purple, fill, inner sep = 0pt,outer sep = 1pt,minimum size=2.7mm] (#2) at #1 {};
}

\newcommand{\GrayRobot}[2]{
\node [circle, Gray, fill, inner sep = 0pt,outer sep = 1pt, minimum size=2.5mm] (#2) at #1 {};
}

\newtheorem{thm}{Theorem}[section]
\newtheorem{prop}[thm]{Proposition}
\newtheorem{lemma}[thm]{Lemma}
\newtheorem{cor}[thm]{Corollary}
\newtheorem*{thm1}{\cref{thm:SrGConnected}}
\newtheorem*{thm2}{\cref{thm:numCellsA}}
\newtheorem*{thm3}{\cref{thm:numCellsB}}
\theoremstyle{definition}
\newtheorem{defn}[thm]{Definition}
\newtheorem{ex}[thm]{Example}
\theoremstyle{remark}
\newtheorem*{rmk}{Remark}

\title{Grouped Stirling complexes}
\author{Alessia Revelli \and Steven Scheirer}
\address{Department of Mathematics and Computer Science, Susquehanna University, Selinsgrove, PA 17870}
\email{revelli@susqu.edu}
\email{scheirer@susqu.edu}


\begin{document}

\begin{abstract}
        Given a graph $G$, a configuration space of $G$ can be thought of as the set of all possible configurations of ``robots" which can move throughout $G$, subject to some constraints. We introduce a type of configuration space which we call \textit{Grouped Stirling complexes}, denoted by $\Sr(G)$, in which we place robots in groups subject to two constraints. First, there must be at least one robot on each vertex of $G$, and second, any two robots from the same group must be ``separated by at least one full open edge" of $G$. The space $\Sr(G)$ has a \textit{closed cell structure,} which means it can be built out of \textit{cells} of various dimensions. Our main results show $\Sr(G)$ is path-connected, provided there are at least three groups, and determine the number of cells of $\Sr(G)$ in certain cases.
\end{abstract}
\maketitle

\section{Introduction}\label{sec:Intro}

Given a set $X$ and a positive integer $r$, a \textit{configuration space of $r$ points in $X$} can be thought of as the set of all possible ways to place $r$ points in $X,$ subject to certain constraints. In this paper, the set $X$ will be a graph $G$ and we will think of the $r$ points as ``robots" which can move throughout the graph $G.$ One can view this as modeling the situation in which the graph represents a system of tracks in a factory, and the robots can move along those tracks to perform tasks throughout the factory.

There are different versions of configuration spaces depending on factors such as whether the robots are distinguishable, whether two robots can occupy the same location on $G,$ and whether there are any points on $G$ which must be occupied by at least one robot. In this section, we will provide an intuitive overview of some of these variations; the relevant precise definitions will be provided in \cref{sec:Basics}.

In the ``standard" configuration space, denoted by $\C^r(G),$ the robots are all distinguishable, and no two robots can occupy the same location on $G.$ To distinguish the robots, we can give each robot a unique label (say an integer between 1 and $r$), or a unique color. For this reason, $\C^r(G)$ is often referred to as a \textit{labeled} configuration space. For example, the left-hand side of \cref{fig:NormalConfigSpaces} shows a graph $Y$, and the middle shows a configuration in $\C^2(Y)$. Here and in the remainder of the figures, we represent the robots as various shapes on the graph. Since the robots in a configuration in $\C^r(G)$ are distinguishable, we use different colors and shapes for each robot. In the configuration in the middle of \cref{fig:NormalConfigSpaces} we have two robots: one red/square robot on the vertex $a$, and one green/diamond robot on the edge $(a,b)$.

\begin{figure}[h!]
    \centering
    \begin{tikzpicture}
        \Ygraph
        \node[left] at (A) {$a$};
        \node[below left] at (B) {$b$};
        \node[left] at (C) {$c$};
        \node[right] at (D) {$d$};
    \end{tikzpicture}
    \hspace{1cm}
    \begin{tikzpicture}
        \Ygraph
        \RedRobot{(A)}{}
        \GreenRobot{($(B)!0.5!(A)$)}{}
        \Ygraph
    \end{tikzpicture}
    \hspace{1cm}
    \begin{tikzpicture}
        \Ygraph
        \GrayRobot{($(B)!.7!(D)$)}{}
        \GrayRobot{($(B)!0.5!(C)$)}{}
        \Ygraph
    \end{tikzpicture}
    \caption{A graph $Y$ (left); configurations in $\C^2(Y)$ (middle) and $\UC^2(Y)$ (right)}
    \label{fig:NormalConfigSpaces}
\end{figure}
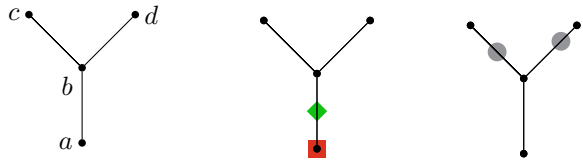

On the other hand, in the \textit{unlabeled} configuration space, denoted by $\UC^r(G)$, it is still required that no two robots occupy the same location on $G$, but the robots are now indistinguishable. That is, we view the robots as all having the same color. For example, the right-hand side of \cref{fig:NormalConfigSpaces} shows a configuration in $\UC^2(Y).$ Here, we use gray circles for both robots to indicate that they are indistinguishable.

There are also \textit{discrete} labeled and unlabeled configuration spaces, introduced in \cite{AbramsThesis} and denoted here by $\D^r(G)$ and $\UD^r(G),$ respectively. These are subspaces of $\C^r(G)$ and $\UC^r(G)$, respectively, in which any two robots must be ``separated by at least a full open edge."  By an ``open edge" or the ``interior of an edge," we mean all the points on that edge which fall strictly between the endpoints of that edge. More specifically, the spaces $\D^r(G)$ and $\UD^r(G)$ have the following constraint:

\begin{itemize}
    \item[($\ast$)] If one robot is on a vertex $v$ of $G$, then no other robot can fall on $v$ or the interior of any edge which has $v$ as an endpoint, and if one robot is on the interior of an edge $e,$ then no other robot can be on the interior of any edge which shares an endpoint with $e.$
\end{itemize} 

For example, the configuration in the middle of \cref{fig:NormalConfigSpaces} is \textit{not} a configuration in $\D^2(Y)$, since there is a robot on the vertex $a$, and another robot on the interior of the edge $(a,b)$. Likewise the configuration on the right of \cref{fig:NormalConfigSpaces} is \textit{not} a configuration in $\UD^2(Y)$ since there is a robot on the interior of the edge $(b,c)$ and another robot on the interior of the edge $(b,d),$ which shares an endpoint with $(b,c).$ However, the configurations in \cref{fig:DiscreteConfigSpaces} \textit{are}  configurations in $\D^2(Y)$ and $\UD^2(Y)$. 

\begin{figure}[h!]
    \centering
    \begin{tikzpicture}
        \fg{\Ygraph}
        \RedRobot{($(A)!.5!(B)$)}{}
        \GreenRobot{(C)}{}
    \end{tikzpicture}
    \hspace{1cm}
    \begin{tikzpicture}
        \fg{\Ygraph}
        \GrayRobot{(B)}{}
        \GrayRobot{(D)}{}
    \end{tikzpicture}
    \caption{Configurations in $\D^2(Y)$ (left) and $\UD^2(Y)$ (right)}
    \label{fig:DiscreteConfigSpaces}
\end{figure}
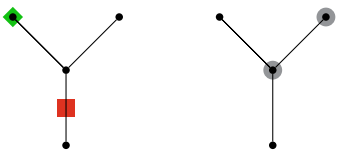

The benefit of working with the discrete configuration spaces $\D^r(G)$ and $\UD^r(G)$ over $\C^r(G)$ and $\UC^r(G)$ is that the discrete spaces have a \textit{closed cubical cell structure}. This means that these spaces can be built out of 0-cells (points), 1-cells (line segments), 2-cells (solid squares), 3-cells (solid cubes), and so on. 

For example, consider the unordered discrete space $\UD^2(Y),$ where $Y$ is again the graph in \cref{fig:NormalConfigSpaces}. One configuration in this space is the configuration in which one robot is on the vertex $a$ and the other robot is on the vertex $b.$ This is represented with the 0-cell (i.e. point) on the top-right of the hexagon in the middle of \cref{fig:UD2Y}. To ``build up" the entire space $\UD^2(Y),$ consider how we can move robot(s) from this configuration. Constraint $(\ast)$ prohibits us from moving the robot on vertex $b$ towards vertex $a$, but we are able to move this robot towards either vertex $c$ or $d.$ First consider moving this robot towards $c$. In doing so, we are moving one robot along the edge $(b,c),$ while the other robot remains fixed on vertex $a.$ This is represented with the 1-cell (line segment) at the top of the hexagon. Any point on this line segment corresponds to a configuration in which one robot is on the vertex $a$ and the other is somewhere on the edge $(b,c)$.

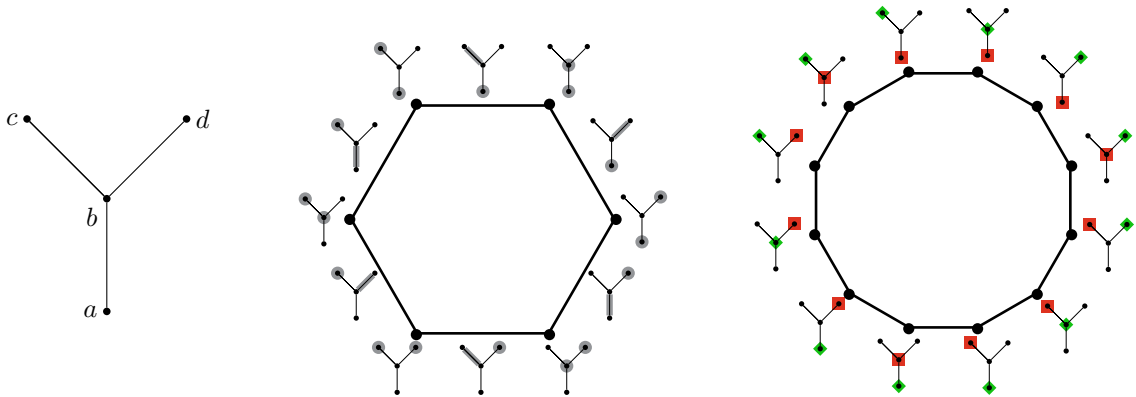
\begin{figure}[h!]
    \centering
    \newsavebox\ConfigA
    \savebox{\ConfigA}{
    \begin{tikzpicture}[scale=0.5]
        \Ygraph
        \GrayRobot{(A)}{}
        \GrayRobot{(B)}{}
    \end{tikzpicture}
    }
    \newsavebox\ConfigB
    \savebox{\ConfigB}{
    \begin{tikzpicture}[scale=0.5]
        \Ygraph;
        \GrayRobot{(A)}{}
        \GrayRobot{(C)}{}
    \end{tikzpicture}
    }
    \newsavebox\ConfigC
    \savebox{\ConfigC}{
    \begin{tikzpicture}[scale=0.5]
        \Ygraph
        \GrayRobot{(B)}{}
        \GrayRobot{(C)}{}
    \end{tikzpicture}
    }
    \newsavebox\ConfigD
    \savebox{\ConfigD}{
    \begin{tikzpicture}[scale=0.5]
        \Ygraph;
        \GrayRobot{(C)}{}
        \GrayRobot{(D)}{}
    \end{tikzpicture}
    }
    \newsavebox\ConfigE
    \savebox{\ConfigE}{
    \begin{tikzpicture}[scale=0.5]
        \Ygraph
        \GrayRobot{(B)}{}
        \GrayRobot{(D)}{}
    \end{tikzpicture}
    }
    \newsavebox\ConfigF
    \savebox{\ConfigF}{
    \begin{tikzpicture}[scale=0.5]
        \Ygraph
        \GrayRobot{(A)}{}
        \GrayRobot{(D)}{}
    \end{tikzpicture}
    }
    \newsavebox\ConfigG
    \savebox{\ConfigG}{
    \begin{tikzpicture}[scale=0.5]
        \Ygraph
        \GrayRobot{(A)}{}
        \draw[Gray, line width = 1.2mm] (B) -- (C);
    \end{tikzpicture}
    }
    \newsavebox\ConfigH
    \savebox{\ConfigH}{
    \begin{tikzpicture}[scale=0.5]
        \Ygraph
        \GrayRobot{(C)}{}
        \draw[Gray, line width = 1.2mm] (A) -- (B);
    \end{tikzpicture}
    }
    \newsavebox\ConfigI
    \savebox{\ConfigI}{
    \begin{tikzpicture}[scale=0.5]
        \Ygraph
        \GrayRobot{(C)}{}
        \draw[Gray, line width = 1.2mm] (B) -- (D);
    \end{tikzpicture}
    }
    \newsavebox\ConfigJ
    \savebox{\ConfigJ}{
    \begin{tikzpicture}[scale=0.5]
        \Ygraph
        \GrayRobot{(D)}{}
        \draw[Gray, line width = 1.2mm] (B) -- (C);
    \end{tikzpicture}
    }
    \newsavebox\ConfigK
    \savebox{\ConfigK}{
    \begin{tikzpicture}[scale=0.5]
        \Ygraph
        \GrayRobot{(D)}{}
        \draw[Gray, line width = 1.2mm] (A) -- (B);
    \end{tikzpicture}
    }
    \newsavebox\ConfigL
    \savebox{\ConfigL}{
    \begin{tikzpicture}[scale=0.5]
        \Ygraph
        \GrayRobot{(A)}{}
        \draw[Gray, line width = 1.2mm] (D) -- (B);
    \end{tikzpicture}
    }
    \raisebox{1cm}{\begin{tikzpicture}[scale=1.5]
        \Ygraph
        \node[left] at (A) {$a$};
        \node[below left] at (B) {$b$};
        \node[left] at (C) {$c$};
        \node[right] at (D) {$d$};
    \end{tikzpicture}}
    \quaad
    \scalebox{.7}{
    \begin{tikzpicture}[scale=.85]
        \node[regular polygon, regular polygon sides = 6, draw, minimum size = 5cm, line width = 0.5mm] (B) at (0,0) {};
        \node[regular polygon, regular polygon sides = 6, minimum size = 6.5cm] (A) at (0,0) {};
        \node[regular polygon, regular polygon sides = 6, minimum size = 6.5cm] (C) at (0,0) {};
        \node[] at (A.corner 1) {\usebox\ConfigA};
        \node[] at (A.corner 2) {\usebox\ConfigB};
        \node[xshift=2mm] at (A.corner 3) {\usebox\ConfigC};
        \node[] at (A.corner 4) {\usebox\ConfigD};
        \node[] at (A.corner 5) {\usebox\ConfigE};
        \node[xshift = -2mm] at (A.corner 6) {\usebox\ConfigF};
        \node[] at ($(C.corner 1) !.5! (C.corner 2)$) {\usebox\ConfigG};
        \node[] at ($(C.corner 2) !.5! (C.corner 3)$) {\usebox\ConfigH};
        \node[] at ($(C.corner 3) !.5! (C.corner 4)$) {\usebox\ConfigI};
        \node[] at ($(C.corner 4) !.5! (C.corner 5)$) {\usebox\ConfigJ};
        \node[] at ($(C.corner 5) !.5! (C.corner 6)$) {\usebox\ConfigK};
        \node[] at ($(C.corner 1) !.5! (C.corner 6)$) {\usebox\ConfigL};
        \foreach \i in{1,...,6}{
        \node[fill,circle, inner sep = .75mm] at (B.corner \i) {};
        }
    \end{tikzpicture}
    }
    \savebox{\ConfigA}{
    \begin{tikzpicture}[scale=0.5]
        \Ygraph
        \RedRobot{(A)}{}
        \GreenRobot{(B)}{}
    \end{tikzpicture}
    }
    \savebox{\ConfigB}{
    \begin{tikzpicture}[scale=0.5]
        \Ygraph
        \RedRobot{(A)}{}
        \GreenRobot{(C)}{}
    \end{tikzpicture}
    }
    \savebox{\ConfigC}{
    \begin{tikzpicture}[scale=0.5]
        \Ygraph
        \RedRobot{(B)}{}
        \GreenRobot{(C)}{}
    \end{tikzpicture}
    }
    \savebox{\ConfigD}{
    \begin{tikzpicture}[scale=0.5]
        \Ygraph
        \RedRobot{(D)}{}
        \GreenRobot{(C)}{}
    \end{tikzpicture}
    }
    \savebox{\ConfigE}{
    \begin{tikzpicture}[scale=0.5]
        \Ygraph
        \RedRobot{(D)}{}
        \GreenRobot{(B)}{}
    \end{tikzpicture}
    }
    \savebox{\ConfigF}{
    \begin{tikzpicture}[scale=0.5]
        \Ygraph
        \RedRobot{(D)}{}
        \GreenRobot{(A)}{}
    \end{tikzpicture}
    }
    \savebox{\ConfigG}{
    \begin{tikzpicture}[scale=0.5]
        \Ygraph
        \RedRobot{(B)}{}
        \GreenRobot{(A)}{}
    \end{tikzpicture}
    }
    \savebox{\ConfigH}{
    \begin{tikzpicture}[scale=0.5]
        \Ygraph
        \RedRobot{(C)}{}
        \GreenRobot{(A)}{}
    \end{tikzpicture}
    }
    \savebox{\ConfigI}{
    \begin{tikzpicture}[scale=0.5]
        \Ygraph
        \RedRobot{(C)}{}
        \GreenRobot{(B)}{}
    \end{tikzpicture}
    }
    \savebox{\ConfigJ}{
    \begin{tikzpicture}[scale=0.5]
        \Ygraph
        \RedRobot{(C)}{}
        \GreenRobot{(D)}{}
    \end{tikzpicture}
    }
    \savebox{\ConfigK}{
    \begin{tikzpicture}[scale=0.5]
        \Ygraph
        \RedRobot{(B)}{}
        \GreenRobot{(D)}{}
    \end{tikzpicture}
    }
    \savebox{\ConfigL}{
    \begin{tikzpicture}[scale=0.5]
        \Ygraph
        \RedRobot{(A)}{}
        \GreenRobot{(D)}{}
    \end{tikzpicture}
    }
    \hspace*{-1cm}
    \scalebox{.7}{
    \begin{tikzpicture}[scale=0.85]
        \node[regular polygon, regular polygon sides = 12, draw, minimum size = 5cm, line width = 0.5mm] (B) at (0,0) {};
        \node[regular polygon, regular polygon sides = 12, minimum size = 6.5cm] (A) at (0,0) {};
        \node[] at (A.corner 1) {\usebox\ConfigA};
        \node[] at (A.corner 2) {\usebox\ConfigB};
        \node[] at (A.corner 3) {\usebox\ConfigC};
        \node[] at (A.corner 4) {\usebox\ConfigD};
        \node[] at (A.corner 5) {\usebox\ConfigE};
        \node[yshift = -1mm] at (A.corner 6) {\usebox\ConfigF};
        \node[] at (A.corner 7) {\usebox\ConfigG};
        \node[] at (A.corner 8) {\usebox\ConfigH};
        \node[yshift = -1mm] at (A.corner 9) {\usebox\ConfigI};
        \node[] at (A.corner 10) {\usebox\ConfigJ};
        \node[] at (A.corner 11) {\usebox\ConfigK};
        \node[] at (A.corner 12) {\usebox\ConfigL};
        \foreach \i in{1,...,12}{
        \node[fill,circle, inner sep = .75mm] at (B.corner \i) {};
        }
    \end{tikzpicture}
    }
    \caption{The graph $Y$ (left); the configuration spaces $\UD^2(Y)$ (middle) and $\D^2(Y)$ (right)}
    \label{fig:UD2Y}
\end{figure}

The left endpoint of this line segment corresponds to the configuration in which one robot is on $a$ and the other is on $c.$ From here, we can move the robot on $a$ along the edge $(a,b),$ which gives another 1-cell of $\UD^2(Y)$ (the top-left side of the hexagon). Continuing in this fashion, we can see that the entire configuration space $\UD^2(Y)$ is a hexagon: each 0-cell in the hexagon represents a configuration in which both robots are on vertices of $Y$, and each 1-cell represents the configurations in which one robot is on a vertex and the other is somewhere on an edge.

Using a similar approach, we can show that the ordered discrete space $\D^2(Y)$ is the 12-gon shown in \cref{fig:UD2Y} (right). For the sake of avoiding clutter, only the 0-cells are labeled.

A few remarks: first, notice that for each cell of $\UD^2(Y)$ there are $2!=2$ cells in $\D^2(Y),$ corresponding to the 2! ways to permute the two colors. More generally, for each cell of $\UD^r(G)$ there are $r!$ cells in $\D^r(G).$

Next, notice that both spaces in \cref{fig:UD2Y} are \textit{connected}. That is, each space only has one ``piece," known more formally as a \textit{connected component}. This means that if we start with any configuration, we can move to any other configuration without ``breaking the rules" (i.e. in moving from the first configuration to the second, we never break constraint $(\ast)$). 

Finally, notice that for the graph $Y$, constraint $(\ast)$ prohibits two robots from occupying the interiors of two different edges, since any two edges in $Y$ have a common endpoint. This is reflected in the fact that the configuration spaces $\UD^2(Y)$ and $\D^2(Y)$ do not have any $i$-cells for $i\ge 2.$ However, if the edge $(c,d)$ is added to $Y$ to form a new graph $Y'$, then in $\UD^2(Y'),$ it is possible that one robot is on the edge $(a,b)$ and the other is on the edge $(c,d),$ so that there is a 2-cell of $\UD^2(Y')$ (in addition to two more 1-cells). The entire configuration space $\UD^2(Y')$ is shown in \cref{fig:UD2Yprime}; the 2-cell is shaded gray. 

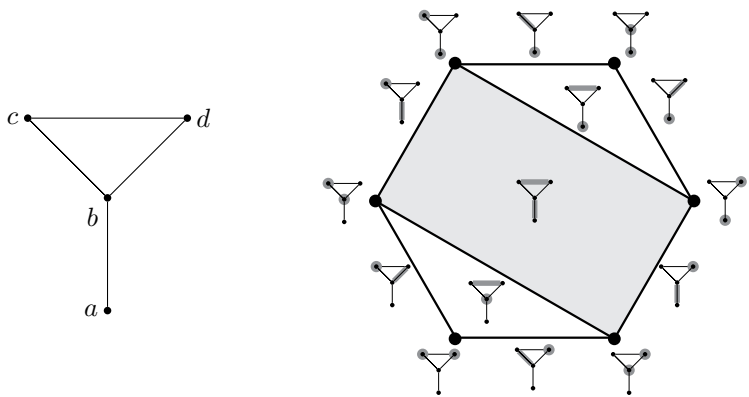
\begin{figure}[h!]
    \centering
    \savebox{\ConfigA}{
    \begin{tikzpicture}[scale=0.5]
        \Ygraph
        \draw (C) -- (D);
        \GrayRobot{(A)}{}
        \GrayRobot{(B)}{}
    \end{tikzpicture}
    }
    \savebox{\ConfigB}{
    \begin{tikzpicture}[scale=0.5]
        \Ygraph
        \draw (C) -- (D);
        \GrayRobot{(A)}{}
        \GrayRobot{(C)}{}
    \end{tikzpicture}
    }
    \savebox{\ConfigC}{
    \begin{tikzpicture}[scale=0.5]
        \Ygraph
        \draw (C) -- (D);
        \GrayRobot{(B)}{}
        \GrayRobot{(C)}{}
    \end{tikzpicture}
    }
    \savebox{\ConfigD}{
    \begin{tikzpicture}[scale=0.5]
        \Ygraph
        \draw (C) -- (D);
        \GrayRobot{(C)}{}
        \GrayRobot{(D)}{}
    \end{tikzpicture}
    }
    \savebox{\ConfigE}{
    \begin{tikzpicture}[scale=0.5]
        \Ygraph
        \draw (C) -- (D);
        \GrayRobot{(B)}{}
        \GrayRobot{(D)}{}
    \end{tikzpicture}
    }
    \savebox{\ConfigF}{
    \begin{tikzpicture}[scale=0.5]
        \Ygraph
        \draw (C) -- (D);
        \GrayRobot{(A)}{}
        \GrayRobot{(D)}{}
    \end{tikzpicture}
    }
    \savebox{\ConfigG}{
    \begin{tikzpicture}[scale=0.5]
        \Ygraph
        \draw (C) -- (D);
        \GrayRobot{(A)}{}
        \draw[Gray, line width = 1.2mm] (B) -- (C);
    \end{tikzpicture}
    }
    \savebox{\ConfigH}{
    \begin{tikzpicture}[scale=0.5]
        \Ygraph
        \draw (C) -- (D);
        \GrayRobot{(C)}{}
        \draw[Gray, line width = 1.2mm] (A) -- (B);
    \end{tikzpicture}
    }
    \savebox{\ConfigI}{
    \begin{tikzpicture}[scale=0.5]
        \Ygraph
        \draw (C) -- (D);
        \GrayRobot{(C)}{}
        \draw[Gray, line width = 1.2mm] (B) -- (D);
    \end{tikzpicture}
    }
    \savebox{\ConfigJ}{
    \begin{tikzpicture}[scale=0.5]
        \Ygraph
        \draw (C) -- (D);
        \GrayRobot{(D)}{}
        \draw[Gray, line width = 1.2mm] (B) -- (C);
    \end{tikzpicture}
    }
    \savebox{\ConfigK}{
    \begin{tikzpicture}[scale=0.5]
        \Ygraph
        \draw (C) -- (D);
        \GrayRobot{(D)}{}
        \draw[Gray, line width = 1.2mm] (A) -- (B);
    \end{tikzpicture}
    }
    \savebox{\ConfigL}{
    \begin{tikzpicture}[scale=0.5]
        \Ygraph
        \draw (C) -- (D);
        \GrayRobot{(A)}{}
        \draw[Gray, line width = 1.2mm] (D) -- (B);
    \end{tikzpicture}
    }
    \newsavebox\ConfigM
    \savebox{\ConfigM}{
    \begin{tikzpicture}[scale=0.5]
        \Ygraph
        \draw (C) -- (D);
        \GrayRobot{(A)}{}
        \draw[Gray, line width = 1.2mm] (C) -- (D);
    \end{tikzpicture}
    }
    \newsavebox\ConfigN
    \savebox{\ConfigN}{
    \begin{tikzpicture}[scale=0.5]
        \Ygraph
        \draw (C) -- (D);
        \GrayRobot{(B)}{}
        \draw[Gray, line width = 1.2mm] (C) -- (D);
    \end{tikzpicture}
    }
    \newsavebox\ConfigO
    \savebox{\ConfigO}{
    \begin{tikzpicture}[scale=0.5]
        \Ygraph
        \draw (C) -- (D);
        \draw[Gray, line width = 1.2mm] (A) -- (B);
        \draw[Gray, line width = 1.2mm] (C) -- (D);
    \end{tikzpicture}
    }
    \raisebox{1cm}{
    \begin{tikzpicture}[scale=1.5]
        \Ygraph
        \draw(C)--(D);
        \node[left] at (A) {$a$};
        \node[below left] at (B) {$b$};
        \node[left] at (C) {$c$};
        \node[right] at (D) {$d$};
    \end{tikzpicture}}
    \hspace{1cm}
    \scalebox{.6}{\begin{tikzpicture}[scale=.85]
        \node[regular polygon, regular polygon sides = 6, draw, minimum size = 7cm, line width = 0.5mm] (B) at (0,0) {};
        \node[regular polygon, regular polygon sides = 6, minimum size = 8.5cm] (A) at (0,0) {};
        \node[regular polygon, regular polygon sides = 6, minimum size = 8.5cm] (C) at (0,0) {};
        \draw [line width = .5mm, fill=Gray!20] (B.corner 2) -- (B.corner 6) -- (B.corner 5) -- (B.corner 3) -- cycle;
        \node[] at (A.corner 1) {\usebox\ConfigA};
        \node[] at (A.corner 2) {\usebox\ConfigB};
        \node[] at (A.corner 3) {\usebox\ConfigC};
        \node[yshift = -1mm] at (A.corner 4) {\usebox\ConfigD};
        \node[yshift = -1mm] at (A.corner 5) {\usebox\ConfigE};
        \node[] at (A.corner 6) {\usebox\ConfigF};
        \node[] at ($(C.corner 1) !.5! (C.corner 2)$) {\usebox\ConfigG};
        \node[] at ($(C.corner 2) !.4! (C.corner 3)$) {\usebox\ConfigH};
        \node[] at ($(C.corner 3) !.5! (C.corner 4)$) {\usebox\ConfigI};
        \node[] at ($(C.corner 4) !.5! (C.corner 5)$) {\usebox\ConfigJ};
        \node[] at ($(C.corner 5) !.5! (C.corner 6)$) {\usebox\ConfigK};
        \node[] at ($(C.corner 1) !.4! (C.corner 6)$) {\usebox\ConfigL};
        \node[yshift = 2mm] at ($(C.corner 2) !.5! (C.corner 6)$) {\usebox\ConfigM};
        \node[yshift=-4mm] at ($(C.corner 3) !.5! (C.corner 5)$) {\usebox\ConfigN};
        \node[] at (C) {\usebox\ConfigO};
        \foreach \i in {1,...,6}{
        \node[fill,circle, inner sep = 1mm] at (B.corner \i) {};
        }
    \end{tikzpicture}
    }
    \caption{A graph $Y'$ (left); the configuration space $\UD^2(Y')$ (right)}
    \label{fig:UD2Yprime}
\end{figure}

The configurations described above have been the subjects of several articles. For an introduction to graph configuration spaces, see \cite{AbramsGhrist}. For an introduction to the discrete spaces, see \cite{AbramsThesis}, where it is also shown that under certain assumptions, the discrete spaces are in some sense ``equivalent"\footnote{More specifically, the discrete spaces are \textit{homotopy equivalent} to their non-discrete counterparts. The notion of homotopy equivalence will be briefly discussed in \cref{sec:CellCounts}.} to their non-discrete counterparts. For an in-depth study of the unordered discrete configuration spaces, see \cite{FarleySabalka, FarleyHomology,FarleyCohomology,FarleySabalkaCohomology}. 

More recently, Kozlov has introduced \textit{Stirling Complexes}\footnote{Stirling complexes are named after the so-called \textit{Stirling numbers}. See an introductory combinatorics text such as \cite{Bona_2016} for more information on Stirling numbers.} in \cite{Kozlov_2023}, denoted here by $\St^r(G)$ (where $G$ is again a graph). Stirling complexes have been further studied in \cite{Kozlov_2022b,Hoekstra-Mendoza_2024}. In the Stirling complex, there are $r$ labeled robots, but there are two key differences between $\C^r(G)$ and $\St^r(G)$: in $\St^r(G)$, multiple robots are permitted to occupy the same location on $G,$ and there must be at least one robot occupying each vertex of $G.$ For example, in \cref{fig:StirlingConfigs}, the configurations on the left and middle are configurations in $\St^5(Y).$ Here and in the remaining figures, if two robots occupy the same location, the figure will depict them side-by-side, but in reality, they occupy the same point. For example, in the configuration on the left, both the green/diamond and yellow/star robots are both located on the central vertex.

\begin{figure}[h!]
    \centering
    \begin{tikzpicture}
        \Ygraph
        \RedRobot{(A)}{}
        \GreenRobot{(B)}{}
        \YellowRobot{($(B)+(-.25,0)$)}{}
        \BlueRobot{(C)}{}
        \PurpleRobot{(D)}{}
    \end{tikzpicture}
    \hspace{1cm}
    \begin{tikzpicture}
        \Ygraph
        \RedRobot{(A)}{}
        \GreenRobot{(B)}{}
        \YellowRobot{($(B)!.5!(C)$)}{}
        \BlueRobot{(C)}{}
        \PurpleRobot{(D)}{}
    \end{tikzpicture}
    \hspace{1cm}
    \begin{tikzpicture}
        \Ygraph
        \RedRobot{(A)}{}
        \RedRobot{(B)}{}
        \GreenRobot{($(B)!.5!(C)$)}{}
        \RedRobot{(C)}{}
        \GreenRobot{(D)}{}
    \end{tikzpicture}
    \caption{Configurations in $\St^5(Y)$ (left and middle) and in $\Sr(Y)$ with $\vec{r}=(3,2)$ (right)}
    \label{fig:StirlingConfigs}
\end{figure}
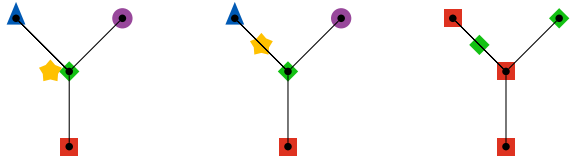

The space $\St^r(G)$ also has a closed cubical cell structure. As an example, for the graph $Y$, one can show that $\St^5(Y)$ consists of 240 0-cells and 360 1-cells. The configuration on the left of \cref{fig:StirlingConfigs} is one of the 0-cells, and the configuration in the middle falls on the interior of a 1-cell.

The main goal of this paper is to work with a generalization of $\St^r(G),$ which we call a \textit{grouped Stirling complex}, in which the robots are placed into groups (indicated by color), and robots within each group must adhere to condition $(\ast)$. For example, the right-hand side of \cref{fig:StirlingConfigs} shows a configuration in which there are three robots in the red/square group, and two in the green/diamond group. We will use a \textit{color vector} $\vec{r}=(3,2)$ (see \cref{defn:GroupedStirling}) to indicate that we have two colors, with three robots of the first color (red) and two robots of the second color (green).

The grouped Stirling complex, which we denote by $\Sr(G)$, also has a closed cell structure. For example, \cref{fig:GroupedStirlingY} shows the configuration space $\Sr(Y)$ for $\vec{r}=(3,2).$ In this case, we have 12 0-cells and 9 1-cells. Notice that unlike the previous examples, $\Sr(Y)$ (with $\vec{r}=(3,2)$) has three connected components. If we start with a configuration in any one of these three connected components, we can move to any other configuration in that component, but it is not possible to move to a configuration in any of the other two connected components. In other words, the space $\Sr(Y)$ is \textit{not} connected in this case. 

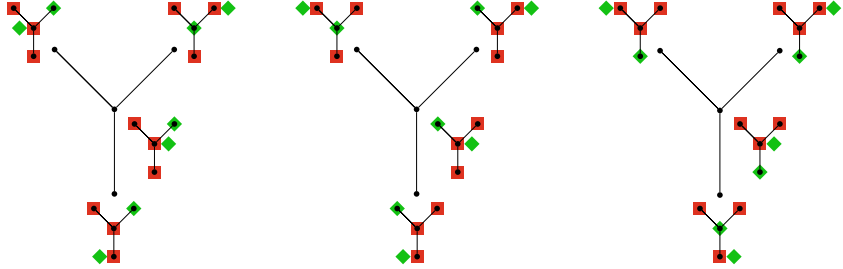
\begin{figure}[h!]
    \centering
    \scalebox{.75}{
    \savebox{\ConfigB}{
    \begin{tikzpicture}[scale=0.5]
        \Ygraph
        \RedRobot{(A)}{}
        \RedRobot{(B)}{}
        \RedRobot{(C)}{}
        \GreenRobot{($(B)+(.5,0)$)}{}
        \GreenRobot{(D)}{}
    \end{tikzpicture}
    }
    \savebox{\ConfigD}{
    \begin{tikzpicture}[scale=0.5]
        \Ygraph
        \RedRobot{(A)}{}
        \RedRobot{(D)}{}
        \RedRobot{(C)}{}
        \GreenRobot{(B)}{}
        \GreenRobot{($(D)+(.5,0)$)}{}
    \end{tikzpicture}
    }
    \savebox{\ConfigC}{
    \begin{tikzpicture}[scale=0.5]
        \Ygraph
        \RedRobot{(A)}{}
        \RedRobot{(B)}{}
        \RedRobot{(C)}{}
        \GreenRobot{($(B)+(-0.5,0)$)}{}
        \GreenRobot{(D)}{}
    \end{tikzpicture}
    }
    \savebox{\ConfigA}{
    \begin{tikzpicture}[scale=0.5]
        \Ygraph
        \RedRobot{(A)}{}
        \RedRobot{(B)}{}
        \RedRobot{(C)}{}
        \GreenRobot{($(A)+(-0.5,0)$)}{}
        \GreenRobot{(D)}{}
    \end{tikzpicture}
    }
    \begin{tikzpicture}[scale=1.5]
        \Ygraph
        \node[below] at (A) {\usebox\ConfigA};
        \node[below right] at (B) {\usebox\ConfigB};
        \node[above left, outer sep = -1em] at (C) {\usebox\ConfigC};
        \node[above right, outer sep = -1em] at (D) {\usebox\ConfigD};
    \end{tikzpicture}
    \savebox{\ConfigB}{
    \begin{tikzpicture}[scale=0.5]
        \Ygraph
        \RedRobot{(A)}{}
        \RedRobot{(B)}{}
        \RedRobot{(D)}{}
        \GreenRobot{($(B)+(.5,0)$)}{}
        \GreenRobot{(C)}{}
    \end{tikzpicture}
    }
    \savebox{\ConfigD}{
    \begin{tikzpicture}[scale=0.5]
        \Ygraph
        \RedRobot{(A)}{}
        \RedRobot{(D)}{}
        \RedRobot{(B)}{}
        \GreenRobot{(C)}{}
        \GreenRobot{($(D)+(.5,0)$)}{};
    \end{tikzpicture}
    }
    \savebox{\ConfigC}{
    \begin{tikzpicture}[scale=0.5]
        \Ygraph
        \RedRobot{(A)}{}
        \RedRobot{(C)}{}
        \RedRobot{(D)}{}
        \GreenRobot{($(C)+(-0.5,0)$)}{}
        \GreenRobot{(B)}{}
    \end{tikzpicture}
    }
    \savebox{\ConfigA}{
    \begin{tikzpicture}[scale=0.5]
        \Ygraph
        \RedRobot{(A)}{}
        \RedRobot{(B)}{}
        \RedRobot{(D)}{}
        \GreenRobot{($(A)+(-0.5,0)$)}{}
        \GreenRobot{(C)}{}
    \end{tikzpicture}
    }
    \begin{tikzpicture}[scale=1.5]
        \Ygraph
        \node[below] at (A) {\usebox\ConfigA};
        \node[below right] at (B) {\usebox\ConfigB};
        \node[above left, outer sep = -1em] at (C) {\usebox\ConfigC};
        \node[above right, outer sep = -1em] at (D) {\usebox\ConfigD};
    \end{tikzpicture}
    \savebox{\ConfigB}{
    \begin{tikzpicture}[scale=0.5]
        \Ygraph
        \RedRobot{(C)}{}
        \RedRobot{(B)}{}
        \RedRobot{(D)}{}
        \GreenRobot{($(B)+(.5,0)$)}{}
        \GreenRobot{(A)}{}
    \end{tikzpicture}
    }
    \savebox{\ConfigD}{
    \begin{tikzpicture}[scale=0.5]
        \Ygraph
        \RedRobot{(C)}{}
        \RedRobot{(B)}{}
        \RedRobot{(D)}{}
        \GreenRobot{($(D)+(.5,0)$)}{}
        \GreenRobot{(A)}{}
    \end{tikzpicture}
    }
    \savebox{\ConfigC}{
    \begin{tikzpicture}[scale=0.5]
        \Ygraph
        \RedRobot{(C)}{}
        \RedRobot{(B)}{}
        \RedRobot{(D)}{}
        \GreenRobot{($(C)+(-.5,0)$)}{}
        \GreenRobot{(A)}{}
    \end{tikzpicture}
    }
    \savebox{\ConfigA}{
    \begin{tikzpicture}[scale=0.5]
        \Ygraph
        \RedRobot{(C)}{}
        \RedRobot{(A)}{}
        \RedRobot{(D)}{}
        \GreenRobot{($(A)+(.5,0)$)}{}
        \GreenRobot{(B)}{}
    \end{tikzpicture}
    }
    \begin{tikzpicture}[scale=1.5]
        \Ygraph;
        \node[below] at (A) {\usebox\ConfigA};
        \node[below right] at (B) {\usebox\ConfigB};
        \node[above left, outer sep = -1em] at (C) {\usebox\ConfigC};
        \node[above right, outer sep = -1em] at (D) {\usebox\ConfigD};
    \end{tikzpicture}
    }
    \caption{The grouped Stirling complex $\Sr(Y)$ with $\vec{r}=(3,2)$}
    \label{fig:GroupedStirlingY}
\end{figure}

The rest of the paper is organized as follows. In \cref{sec:Basics}, we provide the precise definitions of the grouped Stirling complex $\Sr(G)$ and discuss some basic properties. In \cref{sec:Connectivity}, we discuss the connectivity of $\Sr(G)$ and show that $\Sr(G)$ is connected (more specifically, path-connected) whenever there are at least three colors, provided $\vec r$ satisfies some non-triviality conditions:

\begin{thm1}
    If $G$ is a connected graph and $\vec r$ is a non-trivial color vector with at least three colors, then $\Sr(G)$ is path-connected.
\end{thm1} In \cref{sec:CellCounts}, we determine the number of cells of each dimension in $\Sr(G)$ for specific families of color vectors:

\begin{thm2}
       Let $G$ be a simple graph with $n$ vertices and $m$ edges, and let $\vec{r}=(2,1,\dots,1)\in\Z^n$. In $\Sr(G)$ we have 
    \[
    \#(\text{i-cells)}=\begin{cases}
        \frac{n!(n^2+n-2)}{4}&\text{if }i=0\\
        \frac{m(n-1)!(n^2+n-4)}{2}&\text{if }i=1\\
        0&\text{if }i\ge 2\\
    \end{cases}.
    \]
\end{thm2}

\begin{thm3}
 Let $G$ be a simple graph with $n$ vertices and $m$ edges, and let $\vec{r}=(n-1,n-1,\dots,n-1)\in\Z^r$ for some integer $r\ge 2$. In $\Sr(G)$ we have 
    \[
    \#(\text{i-cells)}=\begin{cases}\displaystyle
        \binom{r}{i}\left(m^in^{r-i}-\displaystyle\sum_{v\in V(G)}\deg(v)^i\right)&\text{if }0\le i<r\\
        m^r+m-\displaystyle\sum_{v\in V(G)}\deg(v)^r&\text{if }i=r\\
        0&\text{if }i> r\\
    \end{cases}.
    \]
\end{thm3}
\section{Definitions and Basic Properties of $\Sr(G)$}\label{sec:Basics}

We start with the definition of a grouped Stirling complex. We view a graph $G$ as a topological space, so that the elements of $G$ are individual points on the graph. We use $V(G)$ and $E(G)$ to denote the set of vertices and closed edges of $G$, respectively, and for the remainder of the paper, we denote $|V(G)|$ by $n$ and denote $|E(G)|$ by $m.$ By a ``closed edge" $e,$ we mean the set of all points on the edge $e,$ including the endpoints. By the ``interior" of $e,$ we mean all points on $e$, excluding the endpoints.

\begin{defn}\label[defn]{defn:GroupedStirling}
    A \textit{color vector} is a vector of the form $\vec{r}=(\ell_1,\ell_2,\dots,\ell_r)$ with $\ell_i\in\Z^+.$ We refer to the integers $1,2,\dots,r$ as \textit{colors}, and for each color $i$, $\ell_i$ is the number of $i$-colored robots. Given a graph $G$ and a color vector $\vec r$, the \textit{grouped Stirling complex} $\Sr(G)$ consists of all $r$-tuples of the form 
    \[
    A=(A_1,A_2,\dots,A_r)
    \]
    such that:
    \begin{enumerate}[(a)]
        \item $A_i\subset G$ and $|A_i|=\ell_i$ for $i=1,\dots,r$. We refer to $A_i$ as the \textit{locations of the $i$-colored robots,} and $A$ as a \textit{configuration} in $\Sr(G).$
        \item For each vertex $v\in V(G)$, there is some $i$ such that $v\in A_i$. In other words, for each vertex $v,$ there is some color $i$ such that there is an $i$-colored robot on $v.$
        \item If $A_i$ contains a point on the interior of some $e\in E(G)$, then $A_i$ cannot contain either endpoint of $e$, or any other points on the interior of any edge which shares an endpoint with $e.$ In other words, any two $i$-colored robots must be separated by at least a full open edge of $G$.
    \end{enumerate}

\end{defn}

\begin{ex} \cref{fig:groupedStirlingConfig} shows the path graph $P_3$ and four configurations in $\Sr(P_3)$ with $\vec r = (2,2,1)$. Here and in the remaining figures, we consider red/square to be color 1, green/diamond to be color 2, and blue/triangle to be color 3. The configuration $A$ is given by $A=(\{x,y\},\{x,z\},\{x\}),$ since the 1-colored robots are on the vertices $x$ and $y$ (i.e. $A_1=\{x,y\}$), the 2-colored robots are on $x$ and $z$ (i.e. $A_2=\{x,z\}$), and the 3-colored robot is on $x$ (i.e. $A_3=\{x\}$). Similarly, $A'=(\{x,y\},\{x,z\},\{a\}),$ where $a$ is a point close to $x$ on the edge $e$,  $A''=(\{x,y\},\{b,z\},\{x\}),$ where $b$ is a point close to $y$ on $e$, and $A'''=(\{x,y\},\{b,z\},\{a\}).$

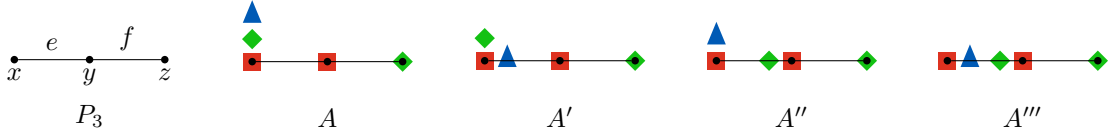
\begin{figure}[h!]
    \centering
    \begin{tikzpicture}[scale=1]
        \fg{\vertex{(0,0)}{x}
        \vertex{(1,0)}{y}
        \vertex{(2,0)}{z}
        \draw (x)--(z);}
        \node[below] at (x) {$x$};
        \node[below] at (y) {$y$};
        \node[below] at (z) {$z$};
        \node[above] at ($.5*(x)+.5*(y)$) {$e$};
        \node[above] at ($.5*(z)+.5*(y)$) {$f$};
        
        \node at (1,-.75) {$P_3$};
    \end{tikzpicture}
    \quaad
    \begin{tikzpicture}[scale=1]
        \fg{\vertex{(0,0)}{x}
        \vertex{(1,0)}{y}
        \vertex{(2,0)}{z}
        \draw (x)--(z);}
        \RedRobot{(x)}{}
        \RedRobot{(y)}{}
        \GreenRobot{($(x)+(0,.3)$)}{}
        \GreenRobot{(z)}{}
        \BlueRobot{($(x)+(0,.6)$)}{}
        \node at (1,-.75) {$A$};
    \end{tikzpicture}
    \quaad
    \begin{tikzpicture}[scale=1]
        \fg{\vertex{(0,0)}{x}
        \vertex{(1,0)}{y}
        \vertex{(2,0)}{z}
        \draw (x)--(z);}
        \RedRobot{(x)}{}
        \RedRobot{(y)}{}
        \GreenRobot{($(x)+(0,.3)$)}{}
        \GreenRobot{(z)}{}
        \BlueRobot{($.7*(x)+.3*(y)$)}{}
        \node at (1,-.75) {$A'$};
    \end{tikzpicture}
    \quaad
    \begin{tikzpicture}[scale=1]
        \fg{\vertex{(0,0)}{x}
        \vertex{(1,0)}{y}
        \vertex{(2,0)}{z}
        \draw (x)--(z);}
        \RedRobot{(x)}{}
        \RedRobot{(y)}{}
        \GreenRobot{($.3*(x)+0.7*(y)$)}{}
        \GreenRobot{(z)}{}
        \BlueRobot{($(x)+(0,.3)$)}{}
        \node at (1,-.75) {$A''$};
    \end{tikzpicture}
    \quaad
    \begin{tikzpicture}[scale=1]
        \fg{\vertex{(0,0)}{x}
        \vertex{(1,0)}{y}
        \vertex{(2,0)}{z}
        \draw (x)--(z);}
        \RedRobot{(x)}{}
        \RedRobot{(y)}{}
        \GreenRobot{($.3*(x)+0.7*(y)$)}{}
        \GreenRobot{(z)}{}
        \BlueRobot{($.7*(x)+.3*(y)$)}{}
        \node at (1,-.75) {$A'''$};
    \end{tikzpicture}
    \caption{The graph $P_3$ and four configurations in $\Sr(P_3)$ with $\vec{r} =(2,2,1)$}
    \label{fig:groupedStirlingConfig}
\end{figure}
\end{ex}
\begin{rmk}
    If the color vector $\vec{r}$ is of the form $\vec{r}=(1,1,\dots,1)$ (i.e. there is only one robot of each color), then the grouped Stirling complex $\Sr(G)$ is homeomorphic to the Stirling complex $\St^r(G)$.
\end{rmk}

As a first observation, we determine necessary and sufficient conditions for the space $\Sr(G)$ to be non-empty.

\begin{prop}
    If $G$ is any graph and $\vec{r}=(\ell_1,\dots,\ell_r)$ is any color vector, then $\Sr(G)$ is non-empty if and only if:
    \begin{enumerate}
        \item $\ell_1+\cdots+\ell_r\ge n$, and
        \item $\ell_i\le n$ for each $i=1,\dots,r$.
    \end{enumerate}
\end{prop}

\begin{proof}
    If $\vec{r}$ satisfies conditions (1) and (2), we can define a configuration $A=(A_1,A_2,\dots,A_r)\in\Sr(G)$ by first placing each vertex of $G$ in some $A_i$ while keeping $|A_i|\le \ell_i$ (this is possible due to condition (1)). Then, for each $i$, if $|A_i|<\ell_i,$ condition (2) allows us to arbitrarily place $\ell_i-|A_i|$ additional vertices in $A_i$. The resulting $r$-tuple $A$ is a configuration in $\Sr(G).$

    Conversely, if $A=(A_1,\dots,A_r)$ is a configuration in $\Sr(G),$ then since each $v\in V(G)$ must be in some $A_i,$ we must have $\ell_1+\cdots+\ell_r\ge n$, and property (c) of \cref{defn:GroupedStirling} forces $\ell_i\le n$ for each $i=1,\dots,r$.
\end{proof}

We make two additional observations. First, if $\ell_1+\cdots+\ell_r=n$, then in any configuration, there must be exactly one robot at each vertex, so that no robots can move. In this case, $\Sr(G)$ is a discrete space.

Second, if $\ell_i=n$ for some $i$, then the $i$-colored robots must be located at each vertex of $G$, so that condition (b) of \cref{defn:GroupedStirling} is automatically satisfied by color $i$ for every $v$. This means that the robots of the remaining colors can move freely without regard to condition (b), so that in this case, $\Sr(G)$ is homeomorphic to a product of unordered discrete spaces of the form 
\[
\UD^{\ell_1}(G)\times\cdots\times\UD^{\ell_r}(G).
\]

So, we will only be interested in the case in which $\ell_1+\cdots+\ell_r>n$ and $\ell_i<n$ for each $i.$

\begin{defn}\label[defn]{NonTrivialColorVector}
    A \textit{non-trivial} color vector is a color vector $\vec{r} = (\ell_1,\dots,\ell_r)$ such that:
    
    \begin{enumerate}
        \item $\ell_1+\cdots+\ell_r> n$, and
        \item $\ell_i< n$ for each $i=1,\dots,r$.
    \end{enumerate}
\end{defn}

As noted in \cref{sec:Intro}, the grouped Stirling complex $\Sr(G)$ has a \textit{closed cell structure.} We will now describe the \textit{cells} of $\Sr(G)$.

\begin{defn}\label[defn]{defn:cell}
    A \textit{cell} of $\Sr(G)$ (with $\vec r=(\ell_1,\dots,\ell_r)$) is an $r$-tuple of the form 
    \[
    c=(c_1,c_2,\dots,c_r)
    \]
    such that
    \begin{enumerate}[(a)]
        \item $c_i\subset V(G)\cup E(G)$ and $|c_i|=\ell_i$ for $i=1,\dots,r.$ We will view $c_i$ as the set of the vertices and edges of $G$ on which there are $i$-colored robots.
        \item For each vertex $v\in V(G),$ there is some $i$ with $v\in c_i.$
        \item For $i=1,\dots,r,$ we have $\sigma\cap\sigma'=\emptyset$ for all distinct vertices and/or closed edges $\sigma,\sigma'\in c_i$.
    \end{enumerate}
    The \textit{dimension} of a cell $c$ is the number of (not necessarily distinct) edges that appear somewhere in $c.$ If the dimension of $c$ is $k,$ we say $c$ is an \textit{$k$-cell}.
\end{defn}

We will view a cell $c=(c_1,c_2,\dots,c_r)$ as the union of all configurations in $\Sr(G)$ in which each $i$-colored robot is located on one of the vertices or edges of $c_i.$

\begin{ex}
    
The square in \cref{fig:2Cell} shows several cells of $\Sr(P_3)$ with $\vec r= (2,2,1)$.

The bottom left corner of the square is the 0-cell $c=(\{x,y\},\{x,z\},\{x\})$. This cell consists of the single configuration in which the two 1-colored (i.e. red/square) robots are on $x$ and $y$, the two 2-colored (i.e. green/diamond) robots are on $x$ and $z,$ and the one 3-colored (i.e. blue/triangle) robot is on $x.$ In other words, $c$ consists of the single configuration $A$ shown in \cref{fig:groupedStirlingConfig}.  Likewise, the other three corners of the square are other 0-cells.

\def\scalefactor{.8}
\savebox{\ConfigA}{
    \scalebox{\scalefactor}{\begin{tikzpicture}[scale=0.7]
    \fg{\vertex{(0,0)}{x}
    \vertex{(1,0)}{y}
    \vertex{(2,0)}{z}
    \draw (x)--(z);}
    \RedRobot{(0,0)}{}
    \RedRobot{(1,0)}{}
    \GreenRobot{(2,0)}{}
    \GreenRobot{(0,.4)}{}
    \BlueRobot{(0,.8)}{}
    \node at (1,-.75) {$c$};
    \end{tikzpicture}
}}
\savebox{\ConfigB}{
    \scalebox{\scalefactor}{\begin{tikzpicture}[scale=0.7]
    \fg{\vertex{(0,0)}{x}
    \vertex{(1,0)}{y}
    \vertex{(2,0)}{z}
    \draw (x)--(z);}
    \RedRobot{(0,0)}{}
    \RedRobot{(1,0)}{}
    \GreenRobot{(2,0)}{}
    \GreenRobot{(0,.4)}{}
    \BlueRobot{(1,.4)}{}
    \end{tikzpicture}
}}
\savebox{\ConfigC}{
    \scalebox{\scalefactor}{\begin{tikzpicture}[scale=0.7]
    \fg{\vertex{(0,0)}{x}
    \vertex{(1,0)}{y}
    \vertex{(2,0)}{z}
    \draw (x)--(z);}
    \RedRobot{(0,0)}{}
    \RedRobot{(1,0)}{}
    \GreenRobot{(2,0)}{}
    \GreenRobot{(1,.4)}{}
    \BlueRobot{(0,.4)}{}
    \end{tikzpicture}
}}
\savebox{\ConfigD}{
    \scalebox{\scalefactor}{\begin{tikzpicture}[scale=0.7]
   \fg{\vertex{(0,0)}{x}
    \vertex{(1,0)}{y}
    \vertex{(2,0)}{z}
    \draw (x)--(z);}
    \RedRobot{(0,0)}{}
    \RedRobot{(1,0)}{}
    \GreenRobot{(2,0)}{}
    \GreenRobot{(1,.4)}{}
    \BlueRobot{(1,.8)}{}
    \end{tikzpicture}
}}
\newsavebox{\OneCellA}
\savebox{\OneCellA}{
    \scalebox{\scalefactor}{\begin{tikzpicture}[scale=0.7]
    \fg{\vertex{(0,0)}{x}
    \vertex{(1,0)}{y}
    \vertex{(2,0)}{z}
    \draw (x)--(z);}
    \draw[line width = 5, Blue] (x)--(y);
    \RedRobot{(0,0)}{}
    \RedRobot{(1,0)}{}
    \GreenRobot{(2,0)}{}
    \GreenRobot{(0,.4)}{}
    \node at (1,-.75) {$c'$};
    \end{tikzpicture}
}}
\newsavebox{\OneCellB}
\savebox{\OneCellB}{
    \scalebox{\scalefactor}{\begin{tikzpicture}[scale=0.7]
    \fg{\vertex{(0,0)}{x}
    \vertex{(1,0)}{y}
    \vertex{(2,0)}{z}
    \draw (x)--(z);}
    \draw[line width = 5, Green] (x)--(y);
    \RedRobot{(0,0)}{}
    \RedRobot{(1,0)}{}
    \GreenRobot{(2,0)}{}
    \BlueRobot{(0,.4)}{}
    \node at (1,-.75) {$c''$};
    \end{tikzpicture}
}}
\newsavebox{\OneCellC}
\savebox{\OneCellC}{
    \scalebox{\scalefactor}{\begin{tikzpicture}[scale=0.7]
    \fg{\vertex{(0,0)}{x}
    \vertex{(1,0)}{y}
    \vertex{(2,0)}{z}
    \draw (x)--(z);}
    \draw[line width = 5, Blue] (x)--(y);
    \RedRobot{(0,0)}{}
    \RedRobot{(1,0)}{}
    \GreenRobot{(2,0)}{}
    \GreenRobot{(1,.4)}{}
    \end{tikzpicture}
}}
\newsavebox{\OneCellD}
\savebox{\OneCellD}{
    \scalebox{\scalefactor}{\begin{tikzpicture}[scale=0.7]
    \fg{\vertex{(0,0)}{x}
    \vertex{(1,0)}{y}
    \vertex{(2,0)}{z}
    \draw (x)--(z);}
    \draw[line width = 5, Green] (x)--(y);
    \RedRobot{(0,0)}{}
    \RedRobot{(1,0)}{}
    \GreenRobot{(2,0)}{}
    \BlueRobot{(1,.4)}{}
    \end{tikzpicture}
}}
\newsavebox{\TwoCell}
\savebox{\TwoCell}{
    \scalebox{1}{\begin{tikzpicture}[scale=0.7]
    \fg{\vertex{(0,0)}{x}
    \vertex{(1,0)}{y}
    \vertex{(2,0)}{z}
    \draw (x)--(z);}
    \draw[line width = 3, Green] (0,-.1)--(1,-.1);
    \draw[line width = 3, Blue] (0,.1)--(1,.1);
    \RedRobot{(0,0)}{}
    \RedRobot{(1,0)}{}
    \GreenRobot{(2,0)}{}
    \node at (1,-.75) {$c'''$};
    \end{tikzpicture}
}}
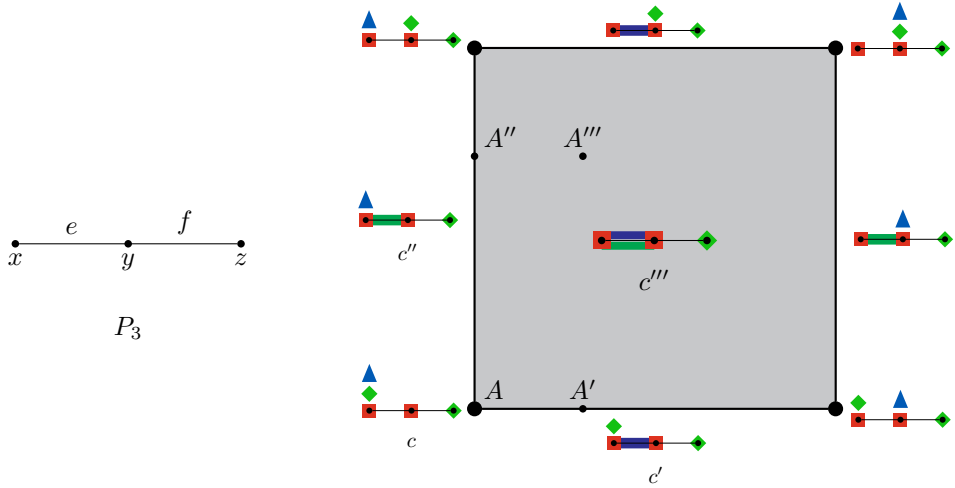
\begin{figure}[h!]
    \centering
        \raisebox{2cm}{\begin{tikzpicture}[scale=1.5]
        \vertex{(0,0)}{x}
        \vertex{(1,0)}{y}
        \vertex{(2,0)}{z}
        \node[below] at (x) {$x$};
        \node[below] at (y) {$y$};
        \node[below] at (z) {$z$};
        \node[above] at ($.5*(x)+.5*(y)$) {$e$};
        \node[above] at ($.5*(z)+.5*(y)$) {$f$};
        \draw (x)--(z);
        \node at (1,-.75) {$P_3$};
    \end{tikzpicture}}\hspace{1cm}
    \begin{tikzpicture}[scale=1.2]
        \draw[fill=Gray!50] (0,0) rectangle (4,4);
        \vertex[2]{(0,0)}{A}
        \vertex[2]{(4,0)}{B}
        \vertex[2]{(4,4)}{D}
        \vertex[2]{(0,4)}{C}
        \node[xshift=-.85cm] at (A) {\usebox\ConfigA};
        \node[xshift=.85cm] at (B) {\usebox\ConfigB};
        \node[below] at ($0.5*(A)+0.5*(B)$) {\usebox\OneCellA};
        \draw[thick] (A)--(B);
        \node[xshift=-.85cm,yshift=.25cm] at (C) {\usebox\ConfigC};
        \draw[thick] (A)--(C);
        \node[left] at ($0.5*(A)+0.5*(C)$) {\usebox\OneCellB};
        \node[xshift=.85cm,yshift=.25cm] at (D) {\usebox\ConfigD};
        \draw[thick] (C)--(D);
        \node[above] at ($0.5*(C)+0.5*(D)$) {\usebox\OneCellC};
        \draw[thick] (B)--(D);
        \node[right] at ($0.5*(B)+0.5*(D)$) {\usebox\OneCellD};
        \node at (2,1.6) {\usebox\TwoCell};
        \node[above right] at (A) {$A$};
        \vertex{($.7*(A)+0.3*(B)$)}{Aprime};
        \node[above]  at (Aprime) {$A'$};
        \vertex{($.3*(A)+0.7*(C)$)}{Ap};
        \node[above right]  at (Ap) {$A''$};
        \vertex{($0.3*(B)+0.7*(C)$)}{Appp};
        \node[above]  at (Appp) {$A'''$};
    \end{tikzpicture}
    \caption{Four 0-cells, four 1-cells, and one 2-cell of $\Sr(P_3)$ with $\vec r=(2,2,1)$}
    \label{fig:2Cell}
\end{figure}

Now, consider the bottom edge of the square. This is the 1-cell $c'=(\{x,y\},\{x,z\},\{e\}),$ which consists of all configurations in which the two 1-colored robots are on $x$ and $y,$ the two 2-colored robots are on $x$ and $z,$ and the one 3-colored robot is somewhere on the edge $e.$ Points on the bottom edge of the square close to the bottom left corner are configurations in which the 3-colored robot on $e$ is close to the vertex $x$; points on the bottom edge close to the bottom right corner are configurations in which the 3-colored robot on $e$ is close to $y.$ For example, the configuration $A'$ in \cref{fig:groupedStirlingConfig} is a point on the bottom edge close to the bottom left corner.  Similarly, the left edge of the square is the 1-cell $c''=(\{x,y\},\{z,e\},\{x\});$ the configuration $A''$ from \cref{fig:groupedStirlingConfig} falls on this 1-cell.

Finally, consider the solid square. This is the 2-cell $c'''=(\{x,y\},\{z,e\},\{e\}),$ which consists of all configurations in which the two 1-colored robots are on $x$ and $y,$ one 2-colored robot is on $z$ and the other is somewhere on $e$, and the one 3-colored robot is also somewhere on the edge $e.$ The location of a point in the square relative to the four corners determines the location of the 2- and 3-colored robots on $e.$ For example, the configuration $A'''$ in \cref{fig:groupedStirlingConfig} is near the top left corner since the 2-colored robot on $e$ is close to $y$ and the 3-colored robot on $e$ is close to $x.$ 
\end{ex}

The square in \cref{fig:2Cell} is one piece of the entire Stirling complex $\Sr(P_3).$ One can show that there are a total of 21 0-cells, 32 1-cells, and 10 2-cells in $\Sr(P_3)$, and $\Sr(P_3)$ is the union of all of these cells. In general, for any graph $G$ and any color vector $\vec r,$ $\Sr(G)$ is the union of all of its cells.

\begin{prop}\label[prop]{prop:UnionOfCells}
    For any graph $G$ and any color vector $\vec r = (\ell_1,\dots,\ell_r),$ the grouped Stirling complex $\Sr(G)$ is the union of the cells defined in \cref{defn:cell}.
\end{prop}

\begin{proof}
    First consider a configuration $A=(A_1,\dots,A_r),$ as defined in \cref{defn:GroupedStirling}. We will show that this configuration falls in some cell $c.$ For each color $i=1,\dots,r,$ the set of locations of the $i$-colored robots is of the form $A_i=\{x^1_i,\dots,x_i^{\ell_i}\},$ where each $x_i^j$ is some point in $G$. For example, for the configuration $A'''=(\{x,y\},\{b,z\},\{a\})$ in \cref{fig:groupedStirlingConfig}, the set of locations of the 1-colored robots is $A_1=\{x,y\},$ so we have $x_1^1=x$ and $x_1^2=y.$ Similarly, for the 2-colored robots, we have $x_2^1=b$ and $x_2^2=z,$ and for the 3-colored robot, we have $x_3^1=a.$ 
    
    Now, for each color $i$, let $c_i=\{\sigma_i^1,\dots,\sigma_i^{\ell_i}\}$, where $\sigma_i^j=x_i^j$ if $x_i^j$ is a vertex and $\sigma_i^j=e$ if $x_i^j$ falls on the interior of some edge $e.$ For example, for $A'''=(\{x,y\},\{b,z\},\{a\})$, since both 1-colored robots fall on vertices, we have $\sigma_1^1=x_1^1=x$ and $\sigma_1^2=x_1^2=y,$ so $c_1=\{x,y\}.$ For the 2-colored robots, since $x_2^1=b,$ and $b$ falls on the interior of the edge $e,$ we have $\sigma_2^1=e,$ but since $x_2^2=z$ is a vertex, we have $\sigma_2^2=z,$ so $c_2=\{e,z\}.$ Likewise, for the 3-colored robot, we have $x_3^1=a,$ which falls on the interior of $e,$ so $c_3=\{e\}.$

    Consider the $r$-tuple $c=(c_1,\dots,c_r).$ We will show this is a valid cell according to \cref{defn:cell}. First, by construction of $c,$ each element in $c_i$ is either a vertex or an edge of $G$, and $|c_i|=\ell_i$. Second, since $A$ is a configuration in $\Sr(G),$ for each vertex $v\in G,$ there is some $i$ with $v\in A_i,$ so $v\in c_i.$ Finally, for each fixed $i$, if $\sigma$ and $ \sigma'$ are distinct elements of $c_i,$ we must show $\sigma\cap\sigma'=\emptyset.$ If both $\sigma$ and $\sigma'$ are vertices, we trivially have $\sigma\cap\sigma'=\emptyset$. If at least one (say $\sigma$) is an edge $e$, this means there is a point $x\in A_i$ where $x$ falls on the interior of $e.$ Therefore, condition (c) of \cref{defn:GroupedStirling} says there are no other points in $A_i$ on either endpoint of $e$, or on the interior of any edge which shares an endpoint with $e.$ Thus, $\sigma'$ is not an endpoint of $e$ or an edge which shares an endpoint with $e,$ so $\sigma\cap \sigma'=\emptyset.$ Thus, $c$ is a cell which contains the configuration $A$.

    A similar argument shows that for any cell $c$, any point in $c$ is a configuration in $\Sr(G)$. Therefore, $A$ is a configuration in $\Sr(G)$ if and only if $A$ is in some cell $c,$ so $\Sr(G)$ is the union of all of its cells.
\end{proof}

Next, we introduce notation that will be used in \cref{sec:Connectivity,sec:CellCounts}.
\begin{defn}
    For a cell $c$ and a vertex or edge $\sigma$, let $c(\sigma)=\{i\mid \sigma\in c_i\}$. A vertex $v\in V(G)$ is \textit{available} in $c$ if $|c(v)|\ge 2.$
\end{defn} 

In other words, $c(\sigma)$ is the list of the colors on the vertex or edge $\sigma$ in $c,$ and a vertex $v$ is available if there are at least two robots on $v$. For example, for the cell $c'=(\{x,y\},\{x,z\},\{e\})$ in \cref{fig:2Cell}, we have 
\[
c'(x)=\{1,2\},\ c'(y)=\{1\},\ c'(z)=\{2\},\ c'(e)=\{3\},\ c'(f)=\emptyset.
\]

The only vertex which is available in $c'$ is $x.$ 

We end this section with the following upper bound on the dimension of a cell.

\begin{prop}\label[prop]{prop:MaxDimension}
    Given a color vector $\vec r = (\ell_1,\ell_2,\dots,\ell_r)$ and any cell $c$ of $\Sr(G)$, we have 
    \[
    \dim(c)\le \ell_1+\cdots+\ell_r-n.
    \]
\end{prop}

\begin{proof}
    Recall $\dim(c)$ is the number of edges that appear somewhere in $c,$ and $\ell_1+\cdots+\ell_r$ is the total number of vertices and edges in $c$. Since $G$ has $n$ vertices, and each vertex must appear somewhere in $c$, this means we can have at most $\ell_1+\cdots+\ell_r-n$ edges in $c.$ Therefore, $\dim(c)\le\ell_1+\cdots+\ell_r-n.$
\end{proof}

\section{Connectedness of $\Sr(G)$}\label{sec:Connectivity}

In this section, we describe how to move between configurations in $\Sr(G).$ More specifically, we will describe \textit{paths} in $\Sr(G).$ Given configurations $A$ and $B$, a path in $\Sr(G)$ from $A$ to $B$ is a continuous function $\gamma\from[0,1]\to \Sr(G)$ with $\gamma(0)=A$ and $\gamma(1)=B.$ Here, we think of the interval $[0,1]$ as a ``time" interval, so that for each $t\in[0,1],$ $\gamma(t)$ gives the positions of the robots at time $t$, with the robots starting in configuration $A$ at $t=0$ and ending in configuration $B$ at $t=1.$

\begin{ex}\label[ex]{ex:path} \cref{fig:path} gives configurations $A$ (left) and $B$ (right). Notice the only difference between $A$ and $B$ is that the green/diamond robot is on $x$ in $A$ and on $y$ in $B$. We can define a path in $\Sr(G)$ from $A$ to $B$ by letting $\gamma(t)$ be the configuration in which the green/diamond robot's location on the edge $e$ is proportional to $t$, and the red/square and blue/triangle robots stay fixed on $x$ and $y$. It is clear that this function is continuous (small perturbations in $t$ result in small perturbations in the location of the green/diamond robot).

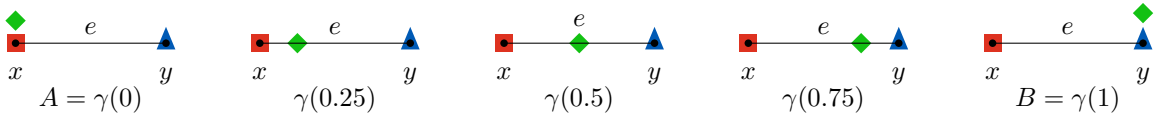
\begin{figure}[h!]
    \centering
    \begin{tikzpicture}
        \fg{
        \vertex{(0,0)}{x}
        \vertex{(2,0)}{y}
        \node[below,yshift=-2mm] at (x) {$x$};
        \node[below,yshift=-2mm] at (y) {$y$};
        \draw (0,0) to node[above,pos=0.5]{$e$} (2,0);
        \node at (1,-.75) {$A=\gamma(0)$};
        }
        \RedRobot{(x)}{}
        \GreenRobot{($(x)+(0,.3)$)}{}
        \BlueRobot{(y)}{}
    \end{tikzpicture}
    \quaad
    \begin{tikzpicture}
        \fg{
        \vertex{(0,0)}{x}
        \vertex{(2,0)}{y}
        \node[below,yshift=-2mm] at (x) {$x$};
        \node[below,yshift=-2mm] at (y) {$y$};
        \draw (0,0) to node[above,pos=0.5]{$e$} (2,0);
        \node at (1,-.75) {$\gamma(0.25)$};
        }
        \RedRobot{(x)}{}
        \GreenRobot{($.75*(x)+.25*(y)$)}{}
        \BlueRobot{(y)}{}
    \end{tikzpicture}
    \quaad
    \begin{tikzpicture}
        \fg{
        \vertex{(0,0)}{x}
        \vertex{(2,0)}{y}
        \node[below,yshift=-2mm] at (x) {$x$};
        \node[below,yshift=-2mm] at (y) {$y$};
        \draw (0,0) to node[above,pos=0.5,yshift=1mm]{$e$} (2,0);
        \node at (1,-.75) {$\gamma(0.5)$};
        }
        \RedRobot{(x)}{}
        \GreenRobot{($.5*(x)+.5*(y)$)}{}
        \BlueRobot{(y)}{}
    \end{tikzpicture}
    \quaad
    \begin{tikzpicture}
        \fg{
        \vertex{(0,0)}{x}
        \vertex{(2,0)}{y}
        \node[below,yshift=-2mm] at (x) {$x$};
        \node[below,yshift=-2mm] at (y) {$y$};
        \draw (0,0) to node[above,pos=0.5]{$e$} (2,0);
        \node at (1,-.75) {$\gamma(0.75)$};
        }
        \RedRobot{(x)}{}
        \GreenRobot{($.25*(x)+.75*(y)$)}{}
        \BlueRobot{(y)}{}
    \end{tikzpicture}
    \quaad
    \begin{tikzpicture}
        \fg{
        \vertex{(0,0)}{x}
        \vertex{(2,0)}{y}
        \node[below,yshift=-2mm] at (x) {$x$};
        \node[below,yshift=-2mm] at (y) {$y$};
        \draw (0,0) to node[above,pos=0.5]{$e$} (2,0);
        \node at (1,-.75) {$B=\gamma(1)$};
        }
        \RedRobot{(x)}{}
        \GreenRobot{($(y)+(0,.4)$)}{}
        \BlueRobot{(y)}{}
    \end{tikzpicture}
    \caption{A path $\gamma$ from $A$ to $B$}
    \label{fig:path}
\end{figure}
\end{ex}

In general, we can describe paths between various configurations by moving one robot along an edge at a time as in \cref{ex:path} (although not every path in $\Sr(G)$ is of this form).

We say that $\Sr(G)$ is \textit{path-connected}\footnote{In topology, a space being ``path-connected" is a stronger statement than being ``connected." For more information regarding the relationship between these two notions, see any standard topology book, such as \cite{Munkres_2000}.} if there is a path between any two configurations in $\Sr(G)$. In \cref{fig:GroupedStirlingY}, we saw that for the graph $Y$ and the color vector $\vec r = (3,2),$ $\Sr(Y)$ is \textit{not} path-connected. We encourage the reader to convince themselves that for $\vec r =(3,3),$ $\Sr(Y)$ \textit{is} path-connected. This shows that with only two colors, $\Sr(G)$ may or may not be path-connected. However, we ultimately show in \cref{thm:SrGConnected} that $\Sr(G)$ is connected provided $\vec{r}$ is a non-trivial color vector with at least three colors.

\textit{For the remainder of this section, we assume $\vec r$ is a non-trivial color vector with at least three colors.} In particular, this means that in any 0-cell, there must be at least one available vertex.

\begin{ex}\label[ex]{ex:movingrobots}
Consider the two 0-cells $c$ and $c'$ in \cref{fig:SwapAdjacentEx}. To obtain a path in $\Sr(G)$ from $c$ to $c',$ we must \textit{swap} the green/diamond robot on $x$ with the red/square robot on $y$ (see \cref{defn:SwapColors}).

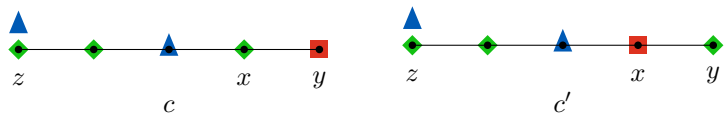
\begin{figure}[h!]
    \centering
    \begin{tikzpicture}
        \fg{\foreach\x in {0,...,4}{\vertex{(\x,0)}{}}
        \draw (0,0) -- (4,0);}
        \GreenRobot{(0,0)}{G1}
        \BlueRobot{(0,.3)}{Y1}
        \GreenRobot{(1,0)}{G2}
        \BlueRobot{(2,0)}{Y2}
        \GreenRobot{(3,0)}{G3}
        \RedRobot{(4,0)}{R1}
        \node[yshift=-4mm] at (0,0) {$z$};
        \node[yshift=-4mm] at (3,0) {$x$};
        \node[yshift=-4mm] at (4,0) {$y$};
        \node at (2,-.75) {$c$};
    \end{tikzpicture}
    \quaad
    \begin{tikzpicture}
        \fg{\foreach\x in {0,...,4}{\vertex{(\x,0)}{}}
        \draw (0,0) -- (4,0);}
        \GreenRobot{(0,0)}{G1}
        \BlueRobot{(0,.3)}{Y1}
        \GreenRobot{(1,0)}{G2}
        \BlueRobot{(2,0)}{Y2}
        \RedRobot{(3,0)}{R1}
        \GreenRobot{(4,0)}{G3}
        \node[yshift=-4mm] at (0,0) {$z$};
        \node[yshift=-4mm] at (3,0) {$x$};
        \node[yshift=-4mm] at (4,0) {$y$};
        \node at (2,-.75) {$c'$};
    \end{tikzpicture}
    \caption{Two 0-cells}
    \label{fig:SwapAdjacentEx}
\end{figure}

Since there must be a robot on every vertex, we must ``borrow" the blue/triangle robot from $z$ so that we can swap the robots on $x$ and $y$. To do this, we can first move the blue/triangle robot from $z$ onto the second vertex from the left, then move the green/diamond robot from the second to the third vertex, then finally move the blue/triangle robot from the third vertex onto $x.$  See \cref{fig:Leapfrog}. In \cref{fig:Leapfrog} and in the remaining figures, numbers above arcs indicate the order in which the robots are moved. We refer to this process as \textit{leapfrogging} (see \cref{defn:Leapfrog}).

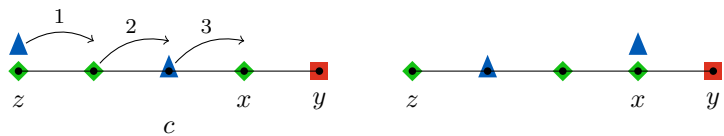
\begin{figure}[h!]
    \centering
    \begin{tikzpicture}
        \fg{\foreach\x in {0,...,4}{\vertex{(\x,0)}{}}
        \draw (0,0) -- (4,0);}
        \GreenRobot{(0,0)}{G1}
        \BlueRobot{(0,.3)}{Y1}
        \GreenRobot{(1,0)}{G2}
        \BlueRobot{(2,0)}{Y2}
        \GreenRobot{(3,0)}{G3}
        \RedRobot{(4,0)}{R1}
        \node[yshift=-4mm] at (0,0) {$z$};
        \node[yshift=-4mm] at (3,0) {$x$};
        \node[yshift=-4mm] at (4,0) {$y$};
        \draw [bend left, ->] (Y1) to node[pos=0.5,above] {\footnotesize1}  ($(G2)+(0,0.4)$);
        \draw [bend left, ->] (G2) to node[pos=0.5,above] {\footnotesize2}  ($(Y2)+(0,0.4)$);
        \draw [bend left, ->] (Y2) to node[pos=0.5,above] {\footnotesize3}  ($(G3)+(0,0.4)$);
        \node at (2,-0.75) {$c$};
    \end{tikzpicture}
    \quaad
    \begin{tikzpicture}
        \fg{\foreach\x in {0,...,4}{\vertex{(\x,0)}{}}
        \draw (0,0) -- (4,0);}
        \GreenRobot{(0,0)}{G1}
        \BlueRobot{(1,0)}{Y1}
        \GreenRobot{(2,0)}{G2}
        \BlueRobot{(3,0.3)}{Y2}
        \GreenRobot{(3,0)}{G3}
        \RedRobot{(4,0)}{R1}
        \node[yshift=-4mm] at (0,0) {$z$};
        \node[yshift=-4mm] at (3,0) {$x$};
        \node[yshift=-4mm] at (4,0) {$y$};
        \node at (2,-0.75) {$\phantom c$};
    \end{tikzpicture}
\caption{Leapfrogging}
\label{fig:Leapfrog}
\end{figure}

After this, we can swap the green/diamond and red/square robots on $x$ and $y$ by moving the green/diamond robot onto $y$ then moving the red/square robot onto $x$. See \cref{fig:SwapRG}.

\begin{figure}[h!]
    \centering
    \begin{tikzpicture}
        \fg{\foreach\x in {0,...,4}{\vertex{(\x,0)}{}}
        \draw (0,0) -- (4,0);}
        \GreenRobot{(0,0)}{G1}
        \BlueRobot{(1,0)}{Y1}
        \GreenRobot{(2,0)}{G2}
        \BlueRobot{(3,0.3)}{Y2}
        \GreenRobot{(3,0)}{G3}
        \RedRobot{(4,0)}{R1}
        \node[yshift=-4mm] at (0,0) {$z$};
        \node[yshift=-4mm] at (3,0) {$x$};
        \node[yshift=-4mm] at (4,0) {$y$};
        \draw [bend left, ->] (G3) to node[pos=0.5,above] {\footnotesize1}  ($(R1)+(0,0.4)$);
        \draw [bend left, ->] (R1) to node[pos=0.5,below] {\footnotesize2}  (G3);
    \end{tikzpicture}
    \quaad
    \begin{tikzpicture}
        \fg{\foreach\x in {0,...,4}{\vertex{(\x,0)}{}}
        \draw (0,0) -- (4,0);}
        \GreenRobot{(0,0)}{G1}
        \BlueRobot{(1,0)}{Y1}
        \GreenRobot{(2,0)}{G2}
        \BlueRobot{(3,0.3)}{Y2}
        \GreenRobot{(4,0)}{G3}
        \RedRobot{(3,0)}{R1}
        \node[yshift=-4mm] at (0,0) {$z$};
        \node[yshift=-4mm] at (3,0) {$x$};
        \node[yshift=-4mm] at (4,0) {$y$};
    \end{tikzpicture}
\caption{Swapping green/diamond and red/square robots}
\label{fig:SwapRG}
\end{figure}
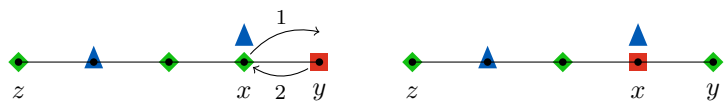

Finally, we can undo the leapfrog as in \cref{fig:ReverseLeapfrog} to arrive at the desired configuration on the right side of \cref{fig:SwapAdjacentEx}.

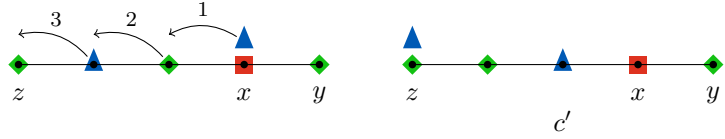
\begin{figure}[h!]
    \centering
    \begin{tikzpicture}
        \fg{\foreach\x in {0,...,4}{\vertex{(\x,0)}{}}
        \draw (0,0) -- (4,0);}
        \GreenRobot{(0,0)}{G1}
        \BlueRobot{(1,0)}{Y1}
        \GreenRobot{(2,0)}{G2}
        \BlueRobot{(3,0.3)}{Y2}
        \GreenRobot{(4,0)}{G3}
        \RedRobot{(3,0)}{R1}
        \node[yshift=-4mm] at (0,0) {$z$};
        \node[yshift=-4mm] at (3,0) {$x$};
        \node[yshift=-4mm] at (4,0) {$y$};
        \draw [bend right, ->] (Y2) to node[pos=0.5,above] {\footnotesize1}  ($(G2)+(0,0.4)$);
        \draw [bend right, ->] (G2) to node[pos=0.5,above] {\footnotesize2}  ($(Y1)+(0,0.4)$);
        \draw [bend right, ->] (Y1) to node[pos=0.5,above] {\footnotesize3}  ($(G1)+(0,0.4)$);
        \node at (2,-0.75) {$\phantom{c'}$};
    \end{tikzpicture}
    \quaad
    \begin{tikzpicture}
        \fg{\foreach\x in {0,...,4}{\vertex{(\x,0)}{}}
        \draw (0,0) -- (4,0);}
        \GreenRobot{(0,0)}{G1}
        \BlueRobot{(0,0.3)}{Y1}
        \GreenRobot{(1,0)}{G2}
        \BlueRobot{(2,0)}{Y2}
        \GreenRobot{(4,0)}{G3}
        \RedRobot{(3,0)}{R1}
        \node[yshift=-4mm] at (0,0) {$z$};
        \node[yshift=-4mm] at (3,0) {$x$};
        \node[yshift=-4mm] at (4,0) {$y$};
        \node at (2,-0.75) {$c'$};
    \end{tikzpicture}
\caption{Undoing the leapfrog}
\label{fig:ReverseLeapfrog}
\end{figure}
\end{ex}

\cref{ex:movingrobots} illustrates the first major step in proving \cref{thm:SrGConnected}: we will show that we can swap any two robots with colors $i\ne j$ on vertices $x\ne y$ (see \cref{cor:SwapColors}). If $x$ and $y$ are both unavailable (as in \cref{ex:movingrobots}), we need to first borrow some $k$-colored robot with $k\notin\{i,j\}$. If there is a $k$-colored robot on some available vertex $z,$ we can leapfrog a robot from $z$ to add a $k$-colored robot to $x$, as in \cref{ex:movingrobots}. \cref{lemma:leapfrog2} shows that this leapfrogging process is possible. If all available vertices contain only $i$- and $j$-colored robots, we will first swap a $k$-colored robot on some unavailable vertex $w$ with an $i$-colored robot on some available vertex $z$ (see \cref{lemma:SwapThird}), then leapfrog a robot from $z$ to add a $k$-colored robot to $x$. 

\textbf{Convention for figures:} \textit{For the remainder of the figures in this section, we will represent color $i$ as green/diamond, color $j$ as red/square, and color $k$ as blue/triangle.}

We start by making the notion of leapfrogging precise. Here and in the remainder of the paper, by a \textit{path in $G$} (not to be confused with a path in $\Sr(G)$), we mean a sequence of vertices in $G$, where there is an edge of $G$ between any two consecutive vertices in the sequence. 

\begin{defn}\label[defn]{defn:Leapfrog}
    Let $x$ and $z$ be vertices of $G$, let $P$ be a path in $G$ from $z$ to $x,$ let $c$ be a 0-cell of $\Sr(G),$ and let $k$ be any color. We say that we can \textit{leapfrog a robot from $z$ along $P$ to add a $k$-colored robot to $x$} if there is a path in $\Sr(G)$ from $c$ to a 0-cell $c'$ such that:
    \begin{itemize}
    \item $c'(x)=c(x)\cup\{k\}$,
    \item $|c'(z)|=|c(z)|-1$,
    \item $|c'(v)|=|c(v)|$ for all vertices $v\in P\setminus\{x,z\}$, and
    \item $c'(v)=c(v)$ for all vertices $v\notin P.$
\end{itemize}
\end{defn}

In other words, we can add a $k$-colored robot to $x$ by removing a robot from $z,$ without changing the number of robots on the other vertices of $P$, and without affecting any vertices outside of $P$. For example, in Figure \ref{fig:Leapfrog}, we leapfrogged a robot from $z$ along the path from $z$ to $x$ to add a blue/triangle robot to $x.$

\begin{rmk}
    Consider a case in which there is an available vertex $z$, a path $P$ in $G$ from $z$ to some vertex $x,$ and a color $k$ such that $k\in c(z)$, but $k\notin c(v)$ for any $v$ on $P\setminus\{z\}$. It is clear that in this case, we can leapfrog a robot from $z$ along $P$ to add a $k$-colored robot to $x.$ See \cref{fig:SimpleLeapfrog}.
\end{rmk}

\begin{figure}[h!]
    \centering
    \begin{tikzpicture}
        \fg{\foreach\x in {0,...,3}{\vertex{(\x,0)}{}}
        \draw (0,0) -- (3,0);}
        \node[below, yshift=-2mm] at (0,0) {$z$};
        \node[below, yshift=-2mm] at (3,0) {$x$};
        \GreenRobot{(0,0)}{G1}
        \BlueRobot{(0,0.3)}{Y1}
        \GreenRobot{(1,0)}{G2}
        \GreenRobot{(2,0)}{G2}
        \RedRobot{(3,0.)}{R1}
        \draw[bend left,->] (Y1) to node[pos=0.5,above]{\footnotesize1} (1,.3);
        \draw[bend left,->] (1.1,.3) to node[pos=0.5,above]{\footnotesize2} (2,.3);
        \draw[bend left,->] (2.1,.3) to node[pos=0.5,above]{\footnotesize3} (3,.3);
        \node at (1.5,-.5) {$c \phantom{c'}$};
    \end{tikzpicture}
    \quaad
    \begin{tikzpicture}
        \fg{\foreach\x in {0,...,3}{\vertex{(\x,0)}{}}
        \draw (0,0) -- (3,0);}
        \node[below, yshift=-2mm] at (0,0) {$z$};
        \node[below, yshift=-2mm] at (3,0) {$x$};
        \GreenRobot{(0,0)}{G1}
        \BlueRobot{(3,0.3)}{Y1}
        \GreenRobot{(1,0)}{G2}
        \GreenRobot{(2,0)}{G2}
        \RedRobot{(3,0.)}{R1}
        \node at (1.5,-.5) {$c'$};
    \end{tikzpicture}
    \caption{Leapfrogging a robot from $z$ to add a blue/triangle robot to $x$}
    \label{fig:SimpleLeapfrog}
\end{figure}
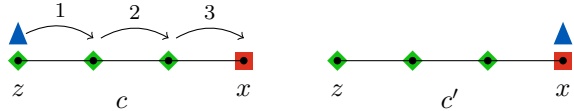

\cref{lemma:leapfrog2} shows that we can replace the condition that $k\notin c(v)$ for all $v$ on $P\setminus\{z\}$ with the weaker condition that $k\notin c(x).$

\begin{lemma}\label[lemma]{lemma:leapfrog2}
Let $c$ be a 0-cell of $\Sr(G)$ such that there are vertices $x$ and $z$ and a color $k,$ where $z$ is available, $k\in c(z),$ and $k\notin c(x).$ For any path $P$ in $G$ from $z$ to $x,$ we can leapfrog a robot from $z$ along $P$ to add a $k$-colored robot to $x.$ 
\end{lemma}

\begin{proof}
    To describe the path in $\Sr(G)$ from $c$ to the cell $c'$ described in \cref{defn:Leapfrog}, we will use induction on the number, $s$, of vertices on $P$. For the base case, $s=2$ (so that $P$ only contains $z$ and $x$, i.e. $z$ and $x$ are adjacent), we can move the $k$-colored robot from $z$ onto $x$ to obtain a cell $c'$ as claimed.

    For the inductive step, assume the claim is true for paths with at most $s$ vertices, and consider the case in which $P$ has $s+1$ vertices. Let $\tilde z$ be the unique vertex on $P$ such that $k\in c(\tilde z)$ and the path $P_1$ from $\tilde z$ to $x$ along $P$ contains no other $k$-colored robots. We consider two cases:

    \textbf{Case 1: $\boldsymbol{\tilde z}$ is available} (including the case in which $\tilde z=z).$ Since there are no $k$-colored robots between $\tilde z$ and $x,$ we can leapfrog a robot from $\tilde z$ along $P_1$ to add a $k$-colored robot to $x,$ as noted in the remark following \cref{defn:Leapfrog}. If $\tilde z=z,$ this gives the desired cell $c'$. Otherwise, the path $P_2$ from $z$ to $\tilde z$ has at most $s$ vertices, so we can use the inductive hypothesis to leapfrog a robot from $z$ along $P_2$ to add a $k$-colored robot to $\tilde z$ to obtain the desired cell $c'.$ See \cref{fig:leapfrog2case1}.

    \begin{figure}[h!]
        \centering
        \begin{tikzpicture}
            \fg{
            \vertex{(0,0)}{}
            \vertex{(2,0)}{}
            \vertex{(4,0)}{}
            \draw[dashed](0,0) to node[pos=0.5,below]{$P_2$} (2,0);
            \draw[dashed](2,0) to node[pos=0.5,below]{$P_1$} (4,0);
            }
            \node[below, yshift=-2mm] at (0,0) {$z$};
            \node[below, yshift=-2mm] at (2,0) {$\tilde z$};
            \node[below, yshift=-2mm] at (4,0) {$x$};
            \GreenRobot{(0,0)}{g1}
            \BlueRobot{(0,0.3)}{y1}
            \BlueRobot{(2,0.3)}{y2}
            \GreenRobot{(4,0)}{g2}
            \RedRobot{(2,0)}{r1}
            \draw[bend left, ->] (y2) to node[pos=0.5,above]{\footnotesize 1} ($(g2)+(0,.3)$);
            \draw[bend left, ->] (y1) to node[pos=0.5,above]{\footnotesize 2} (y2);
            \node at (2,-1) {$c$};
        \end{tikzpicture}
        \quad
        \begin{tikzpicture}
            \fg{
            \vertex{(0,0)}{}
            \vertex{(2,0)}{}
            \vertex{(4,0)}{}
            \draw[dashed](0,0) to node[pos=0.5,below]{$P_2$} (2,0);
            \draw[dashed](2,0) to node[pos=0.5,below]{$P_1$} (4,0);
            }
            \node[below, yshift=-2mm] at (0,0) {$z$};
            \node[below, yshift=-2mm] at (2,0) {$\tilde z$};
            \node[below, yshift=-2mm] at (4,0) {$x$};
            \GreenRobot{(0,0)}{g1}
            \BlueRobot{(4,0.3)}{y1}
            \BlueRobot{(2,0.3)}{y2}
            \GreenRobot{(4,0)}{g2}
            \RedRobot{(2,0)}{r1}
            \node at (2,-1) {$c'$};
        \end{tikzpicture}
        \caption{Case 1 of \cref{lemma:leapfrog2}}
        \label{fig:leapfrog2case1}
    \end{figure}
    
    \textbf{Case 2:  $\boldsymbol{\tilde z}$ is not available} (so $\tilde z\ne z$). Again let $P_2$ be the path from $z$ to $\tilde z$. Since $z$ is available, there is some $i\ne k$ with $i\in c(z).$  Since $\tilde z$ is not available and $k\in c(\tilde z)$, we have $i\notin c(\tilde z),$ so we can use the inductive hypothesis to leapfrog a robot from $z$ along $P_2$ to add an $i$-colored robot to $\tilde z,$ then leapfrog a robot from $\tilde z$ along $P_1$ to add a $k$-colored robot to $x.$ This gives the desired cell $c'.$ See \cref{fig:leapfrog2case2}.
    
    \begin{figure}[h!]
        \centering
        \begin{tikzpicture}
            \fg{
            \vertex{(0,0)}{}
            \vertex{(2,0)}{}
            \vertex{(4,0)}{}
            \draw[dashed](0,0) to node[pos=0.5,below]{$P_2$} (2,0);
            \draw[dashed](2,0) to node[pos=0.5,below]{$P_1$} (4,0);
            }
            \node[below, yshift=-2mm] at (0,0) {$z$};
            \node[below, yshift=-2mm] at (2,0) {$\tilde z$};
            \node[below, yshift=-2mm] at (4,0) {$x$};
            \GreenRobot{(0,0)}{g1}
            \BlueRobot{(0,0.3)}{y1}
            \BlueRobot{(2,0)}{y2}
            \GreenRobot{(4,0)}{g2}
            \draw[bend left, ->] (y2) to node[pos=0.5,above]{\footnotesize 2} ($(g2)+(0,.3)$);
            \draw[bend left, ->] (g1) to node[pos=0.5,above]{\footnotesize 1} ($(y2)+(0,.3)$);
            \node at (2,-1) {$c$};
        \end{tikzpicture}
        \quad
        \begin{tikzpicture}
            \fg{
            \vertex{(0,0)}{}
            \vertex{(2,0)}{}
            \vertex{(4,0)}{}
            \draw[dashed](0,0) to node[pos=0.5,below]{$P_2$} (2,0);
            \draw[dashed](2,0) to node[pos=0.5,below]{$P_1$} (4,0);
            }
            \node[below, yshift=-2mm] at (0,0) {$z$};
            \node[below, yshift=-2mm] at (2,0) {$\tilde z$};
            \node[below, yshift=-2mm] at (4,0) {$x$};
            \GreenRobot{(2,0)}{g1}
            \BlueRobot{(0,0)}{y1}
            \BlueRobot{(4,0.3)}{y2}
            \GreenRobot{(4,0)}{g2}
            \node at (2,-1) {$c'$};
        \end{tikzpicture}
        \caption{Case 2 of \cref{lemma:leapfrog2}}
        \label{fig:leapfrog2case2}
    \end{figure}
\end{proof}

\begin{rmk}
    Since the path in $\Sr(G)$ described in the proof of \cref{lemma:leapfrog2} doesn't depend on the color of any of the robots on $x$ (other than the fact that there are no $k$-colored robots on $x$), it is clear that if even we change the color of any of the non-$k$-colored robots on $x$ in $c',$ we can reverse the path in $\Sr(G)$ to return all of the robots on $P$ to their original positions in $c$, other than the one whose color was altered on $x$. We refer to this as \textit{undoing the leapfrog} (as in \cref{fig:ReverseLeapfrog}).
\end{rmk}

Next, we make the notion of swapping robots precise.

\begin{defn}\label[defn]{defn:SwapColors}
    Let $x$ and $y$ be any vertices of $G$, and let $c$ be a 0-cell of $\Sr(G)$ such that there are colors $i,j$ with $i\in c(x)$ and $j\in c(y)$. We say that we can \textit{swap the $i$-colored robot on $x$ with the $j$-colored robot on $y$} if there is a path in $\Sr(G)$ from $c$ to a 0-cell $c'$ such that:
    \begin{itemize}
        \item $c'(x)=(c(x)\setminus\{i\})\cup\{j\}$,
        \item $c'(y)=(c(y)\setminus\{j\})\cup\{i\}$, and
        \item $c'(v)=c(v)$ for all $v\in V(G)\setminus\{x,y\}$.
    \end{itemize}
\end{defn}

As mentioned above, in \cref{cor:SwapColors} we show that we can swap any two robots of different colors $i\ne j$ on vertices $x\ne y$. There are several cases in the proof of \cref{cor:SwapColors}, the most difficult of which is when $x$ and $y$ are both unavailable, and all available vertices only contain $i$- and $j$-colors robots. In this case, we will first swap a $k$-colored robot (for some $k\notin \{i,j\}$) on some non-available vertex $w$ with an $i$-colored robot on some available vertex $z.$ \cref{lemma:SwapThird} shows that such a swap is always possible.

\begin{lemma}\label[lemma]{lemma:SwapThird}
    Let $c$ be a 0-cell of $\Sr(G)$ such that there are vertices $z$ and $w,$ a path $P$ in $G$ from $z$ to $w$, and colors $i\ne k$ such that $c(w)=\{k\},$ $z$ is available and $i\in c(z)$, and $k\notin c(v)$ for all $v\in P\setminus\{w\}$. Then, we can swap the $i$-colored robot on $z$ with the $k$-colored robot on $w.$
\end{lemma}

\begin{proof}
    To describe a path from $c$ to the desired cell $c',$ we will use induction on the number, $s$, of vertices on $P$. For the base case, $s=2$, $P$ only contains $z$ and $w$. Since $w$ only contains a $k$-colored robot, and $i\ne k$, we can move the $i$-colored robot from $z$ onto $w$. Then, since $z$ does not have a $k$-colored robot, we can move the $k$-colored robot from $w$ to $z$ to obtain $c'$. See \cref{fig:SwapThirdBaseCase}.

    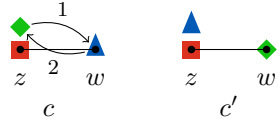
\begin{figure}[h!]
        \centering
        \begin{tikzpicture}
            \fg{
            \vertex{(0,0)}{}
            \vertex{(1,0)}{}
            \draw (0,0)--(1,0); 
            }
            \RedRobot{(0,0)}{r1}
            \GreenRobot{(0,0.3)}{g1}
            \BlueRobot{(1,0)}{y1}
            \draw[bend left, ->](g1) to node[pos=0.5, above]{\footnotesize1} (y1);
            \draw[bend left, ->](y1) to node[pos=0.5, below]{\footnotesize2}(g1);
            \node[below, yshift=-2mm] at (0,0){$z$};
            \node[below, yshift=-2mm] at (1,0){$w$};
            \node at (0.5, -.75){$c\phantom{c'}$};
        \end{tikzpicture}
        \quaad 
        \begin{tikzpicture}
            \fg{
            \vertex{(0,0)}{}
            \vertex{(1,0)}{}
            \draw (0,0)--(1,0); 
            } 
            \RedRobot{(0,0)}{r1}
            \BlueRobot{(0,0.3)}{y1}
            \GreenRobot{(1,0)}{g1}
            \node[below, yshift=-2mm] at (0,0){$z$};
            \node[below, yshift=-2mm] at (1,0){$w$};
            \node at (0.5, -.75){$c'$};
        \end{tikzpicture}
        \caption{The base case for \cref{lemma:SwapThird}}
        \label{fig:SwapThirdBaseCase}
    \end{figure}
    
    For the inductive step, assume the claim is true for paths of length $s$ and consider a case where the length of the path is $s + 1$. Let $\tilde z$ be the vertex on $P$ that is adjacent to $z$, and let $\tilde P$ be the path from $\tilde z$ to $w,$ so that $\tilde P$ has length $s$. We consider three different cases.

    \textbf{Case 1: $\boldsymbol{i\notin c(\tilde z)}$.} In this case, we can move the $i$-colored robot from $z$ onto $\tilde z$. This gives a cell $\tilde c$ for which the path $\tilde{P}$ satisfies the inductive hypothesis. Thus, by induction, we can swap the $i$-colored robot on $\tilde z$ with the $k$-colored robot on $w$. Then, we can move the $k$-colored robot from $\tilde z$ onto $z$ to obtain $c'$. See \cref{fig:SwapThird1}.
    
    \begin{figure}[h!]
        \centering
        \begin{tikzpicture}
            \fg{
            \vertex{(0,0)}{}
            \vertex{(1,0)}{}
            \vertex{(3,0)}{}
            \draw (0,0) -- (1,0);
            \draw[dashed] (1,0) to node[above]{$\tilde P$} (3,0);
            }
            \RedRobot{(0,0)}{r1}
            \GreenRobot{(0,.3)}{g1}
            \RedRobot{(1,0)}{r2}
            \BlueRobot{(3,0)}{y1}
            \node[below, yshift=-2mm] at (0,0) {$z$};
            \node[below, yshift=-2mm] at (1,0) {$\tilde z$};
            \node[below, yshift=-2mm] at (3,0) {$w$};
            \draw[bend left, ->] (g1) to ($(r2)+(0,0.3)$);
            \node at (1.5,-.75) {$c$};
        \end{tikzpicture}
        \quad
        \begin{tikzpicture}
           \fg{
            \vertex{(0,0)}{}
            \vertex{(1,0)}{}
            \vertex{(3,0)}{}
            \draw (0,0) -- (1,0);
            \draw[dashed] (1,0) to node[above]{$\tilde P$} (3,0);
            }
            \RedRobot{(0,0)}{r1}
            \GreenRobot{(1,.3)}{g1}
            \RedRobot{(1,0)}{r2}
            \BlueRobot{(3,0)}{y1}
            \node[below, yshift=-2mm] at (0,0) {$z$};
            \node[below, yshift=-2mm] at (1,0) {$\tilde z$};
            \node[below, yshift=-2mm] at (3,0) {$w$};
            \draw[out=45, ->] (g1) to  (y1);
            \draw[bend left, ->] (y1) to (g1);
            \node at (1.5,-.75) {$\tilde c$};
        \end{tikzpicture}
        \quad
        \begin{tikzpicture}
            \fg{
            \vertex{(0,0)}{}
            \vertex{(1,0)}{}
            \vertex{(3,0)}{}
            \draw (0,0) -- (1,0);
            \draw[dashed] (1,0) to node[above]{$\tilde P$} (3,0);
            }
            \RedRobot{(0,0)}{r1}
            \GreenRobot{(3,0)}{g1}
            \RedRobot{(1,0)}{r2}
            \BlueRobot{(1,0.3)}{y1}
            \node[below, yshift=-2mm] at (0,0) {$z$};
            \node[below, yshift=-2mm] at (1,0) {$\tilde z$};
            \node[below, yshift=-2mm] at (3,0) {$w$};
            \draw[bend right, ->] (y1) to ($(r1)+(0,.3)$);
            \node at (1.5,-.75) {$\phantom c$};
        \end{tikzpicture}
        \quad
        \begin{tikzpicture}
            \fg{
            \vertex{(0,0)}{}
            \vertex{(1,0)}{}
            \vertex{(3,0)}{}
            \draw (0,0) -- (1,0);
            \draw[dashed] (1,0) to node[above]{$\tilde P$} (3,0);
            }
            \RedRobot{(0,0)}{r1}
            \GreenRobot{(3,0)}{g1}
            \RedRobot{(1,0)}{r2}
            \BlueRobot{(0,0.3)}{y1}
            \node[below, yshift=-2mm] at (0,0) {$z$};
            \node[below, yshift=-2mm] at (1,0) {$\tilde z$};
            \node[below, yshift=-2mm] at (3,0) {$w$};
            \node at (1.5,-.75) {$c'$};
        \end{tikzpicture}
        \caption{Case 1 of \cref{lemma:SwapThird}}
        \label{fig:SwapThird1}
    \end{figure}

    \textbf{Case 2: $\boldsymbol{i\in c(\tilde z)}$ and $\boldsymbol{\tilde z}$ is not available.} We can move any $j$-colored robot from $z$, with $j\ne i$, onto $\tilde z$. By the inductive hypothesis, we can swap the $i$-colored robot on $\tilde z$ with the $k$-colored robot on $w$. Then, we can move the $k$-colored robot from $\tilde z$ onto $z,$ then move the $i$-colored robot from $z$ onto $\tilde z$, then finally move the $j$-colored robot from $\tilde z$ to $z$ to obtain $c'$. See \cref{fig:SwapThird2}.

        \begin{figure}[h!]
        \centering
        \begin{tikzpicture}
            \fg{
            \vertex{(0,0)}{}
            \vertex{(1,0)}{}
            \vertex{(3,0)}{}
            \draw (0,0) -- (1,0);
            \draw[dashed] (1,0) to node[above]{$\tilde P$} (3,0);
            }
            \RedRobot{(0,0)}{r1}
            \GreenRobot{(0,.3)}{g1}
            \GreenRobot{(1,0)}{g2}
            \BlueRobot{(3,0)}{y1}
            \node[below, yshift=-2mm] at (0,0) {$z$};
            \node[below, yshift=-2mm] at (1,0) {$\tilde z$};
            \node[below, yshift=-2mm] at (3,0) {$w$};
            \draw[bend left, ->] (r1) to ($(g2)+(0,0.3)$);
            \node at (1.5,-.75) {$c$};
        \end{tikzpicture}
        \quad
        \begin{tikzpicture}
            \fg{
            \vertex{(0,0)}{}
            \vertex{(1,0)}{}
            \vertex{(3,0)}{}
            \draw (0,0) -- (1,0);
            \draw[dashed] (1,0) to node[above]{$\tilde P$} (3,0);
            }
            \GreenRobot{(0,0)}{g1}
            \GreenRobot{(1,0)}{g2}
            \RedRobot{(1,.3)}{r2}
            \BlueRobot{(3,0)}{y1}
            \node[below, yshift=-2mm] at (0,0) {$z$};
            \node[below, yshift=-2mm] at (1,0) {$\tilde z$};
            \node[below, yshift=-2mm] at (3,0) {$w$};
            \draw[out=60, in=120,->] (g2) to (y1);
            \draw[bend left, ->] (y1) to (g2);
            \node at (1.5,-.75) {$\phantom c$};
        \end{tikzpicture}
        \quad
        \begin{tikzpicture}
            \fg{
            \vertex{(0,0)}{}
            \vertex{(1,0)}{}
            \vertex{(3,0)}{}
            \draw (0,0) -- (1,0);
            \draw[dashed] (1,0) to node[above]{$\tilde P$} (3,0);
            }
            \GreenRobot{(0,0)}{g1}
            \GreenRobot{(3,0)}{g2}
            \RedRobot{(1,0)}{r2}
            \BlueRobot{(1,0.3)}{y1}
            \node[below, yshift=-2mm] at (0,0) {$z$};
            \node[below, yshift=-2mm] at (1,0) {$\tilde z$};
            \node[below, yshift=-2mm] at (3,0) {$w$};
            \draw[out=115, in=90, ->] (y1) to node[pos = .5, above]{\footnotesize1} ($(r1)+(0,.3)$);
            \draw[bend right, ->] (g1) to node[pos=.5, below]{\footnotesize2} (r2);
            \draw[bend right, ->] (r2) to node[pos=.5, above]{\footnotesize3} (g1);
            \node at (1.5,-.75) {$\phantom c$};
        \end{tikzpicture}
        \quad
        \begin{tikzpicture}
            \fg{
            \vertex{(0,0)}{}
            \vertex{(1,0)}{}
            \vertex{(3,0)}{}
            \draw (0,0) -- (1,0);
            \draw[dashed] (1,0) to node[above]{$\tilde P$} (3,0);
            }
            \RedRobot{(0,0)}{r1}
            \GreenRobot{(3,0)}{g1}
            \GreenRobot{(1,0)}{g2}
            \BlueRobot{(0,0.3)}{y1}
            \node[below, yshift=-2mm] at (0,0) {$z$};
            \node[below, yshift=-2mm] at (1,0) {$\tilde z$};
            \node[below, yshift=-2mm] at (3,0) {$w$};
            \node at (1.5,-.75) {$c'$};
        \end{tikzpicture}
        \caption{Case 2 of \cref{lemma:SwapThird}}
        \label{fig:SwapThird2}
    \end{figure}
    
    \textbf{Case 3: $\boldsymbol{i\in c(\tilde z)}$ and $\boldsymbol{\tilde z}$ is available}. By the inductive hypothesis, we can swap the $i$-colored robot on $\tilde z$ with the $k$-colored robot on $w$. Then, we can move the $k$-colored robot from $\tilde z$ onto $z$ and move the $i$-colored robot from $z$ onto $\tilde z$ to obtain $c'$. See \cref{fig:SwapThird3}.

 \begin{figure}[h!]
        \centering
        \begin{tikzpicture}
           \fg{
            \vertex{(0,0)}{}
            \vertex{(1,0)}{}
            \vertex{(3,0)}{}
            \draw (0,0) -- (1,0);
            \draw[dashed] (1,0) to node[above]{$\tilde P$} (3,0);
            }
            \GreenRobot{(0,0.3)}{g1}
            \GreenRobot{(1,.3)}{g1}
            \RedRobot{(1,0)}{r2}
            \RedRobot{(0,0)}{r3}
            \BlueRobot{(3,0)}{y1}
            \node[below, yshift=-2mm] at (0,0) {$z$};
            \node[below, yshift=-2mm] at (1,0) {$\tilde z$};
            \node[below, yshift=-2mm] at (3,0) {$w$};
            \draw[out=45, ->] (g1) to (y1);
            \draw[bend left, ->] (y1) to (g1);
            \node at (1.5,-.75) {$ c$};
        \end{tikzpicture}
        \quaad
        \begin{tikzpicture}
            \fg{
            \vertex{(0,0)}{}
            \vertex{(1,0)}{}
            \vertex{(3,0)}{}
            \draw (0,0) -- (1,0);
            \draw[dashed] (1,0) to node[above]{$\tilde P$} (3,0);
            }
            \GreenRobot{(0,0.3)}{g1}
            \GreenRobot{(3,0)}{g2}
            \RedRobot{(1,0)}{r2}
            \RedRobot{(0,0)}{r1}
            \BlueRobot{(1,0.3)}{y1}
            \node[below, yshift=-2mm] at (0,0) {$z$};
            \node[below, yshift=-2mm] at (1,0) {$\tilde z$};
            \node[below, yshift=-2mm] at (3,0) {$w$};
            \draw[bend right, ->] (y1) to node[pos=.5, above]{\footnotesize1} ($(g1)+(0,.3)$);
            \draw[bend right, ->](g1) to node[pos=.5, above, yshift=-.5mm]{\footnotesize2} (y1);
            \node at (1.5,-.75) {$\phantom c$};
        \end{tikzpicture}
        \quaad
        \begin{tikzpicture}
            \fg{
            \vertex{(0,0)}{}
            \vertex{(1,0)}{}
            \vertex{(3,0)}{}
            \draw (0,0) -- (1,0);
            \draw[dashed] (1,0) to node[above]{$\tilde P$} (3,0);
            }
            \GreenRobot{(3,0)}{g1}
            \GreenRobot{(1,0.3)}{g2}
            \RedRobot{(1,0)}{r2}
            \RedRobot{(0,0)}{r1}
            \BlueRobot{(0,0.3)}{y1}
            \node[below, yshift=-2mm] at (0,0) {$z$};
            \node[below, yshift=-2mm] at (1,0) {$\tilde z$};
            \node[below, yshift=-2mm] at (3,0) {$w$};
            \node at (1.5,-.75) {$c'$};
        \end{tikzpicture}
        \caption{Case 3 of \cref{lemma:SwapThird}}
        \label{fig:SwapThird3}
    \end{figure}
\end{proof}

\begin{cor}\label[cor]{cor:SwapColors}
    Let $G$ be a connected graph, and let $c$ be a 0-cell of $\Sr(G)$ such that there are vertices $x\ne y$ and colors $i\ne j$ where $i\in c(x),$ $j\notin c(x)$, $j\in c(y),$ and $i\notin c(y)$. Then, we can swap the $i$-colored robot on $x$ with the $j$-colored robot on $y.$
\end{cor}

\begin{proof}
Since $G$ is connected, there is a path $P$ in $G$ from $x$ to $y$. To describe a path from $c$ to the desired cell $c',$ we will use induction on the number, $s$, of vertices of $P$. For the base case, $s=2$, $x$ and $y$ are adjacent. We will consider the following cases:

\textbf{Case 1: $\boldsymbol{x}$ or $\boldsymbol{y}$ is available}. Without loss of generality, assume $x$ is available. We can move the $i$-colored robot from $x$ onto $y$, then move the $j$-colored robot from $y$ onto $x$ to obtain $c'$. See \cref{fig:SwapColorsBase1}.

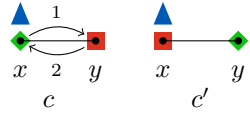
\begin{figure}[h!]
    \centering
        \begin{tikzpicture}
        \fg{
        \vertex{(0,0)}{}
        \vertex{(1,0)}{}
        \draw (0,0) to (1,0);
        }
        \GreenRobot{(0,0)}{g1}
        \RedRobot{(1,0)}{r1}
        \BlueRobot{(0,0.3)}{y1}
        \draw[bend left, ->] (g1) to node[above] {$\scriptstyle 1$}(r1);
        \draw[bend left, ->] (r1) to node[below] {$\scriptstyle2$} (g1);
        \node[below,yshift=-2mm] at (0,0) {$x$};
        \node[below,yshift=-2mm] at (1,0) {$y$};
        \node at (0.5,-.75) {$c\phantom{c'}$};
    \end{tikzpicture}
      \quad
        \begin{tikzpicture}
        \fg{
        \vertex{(0,0)}{}
        \vertex{(1,0)}{}
        \draw (0,0) to (1,0);
        }
        \GreenRobot{(1,0)}{g1}
        \RedRobot{(0,0)}{r1}
        \BlueRobot{(0,0.3)}{y1}
        \node[below,yshift=-2mm] at (0,0) {$x$};
        \node[below,yshift=-2mm] at (1,0) {$y$};
        \node at (0.5,-.75) {$c'$};
    \end{tikzpicture}
    \caption{Case 1 of the base case for \cref{cor:SwapColors}}
    \label{fig:SwapColorsBase1}
\end{figure}

\textbf{Case 2: Neither $\boldsymbol{x}$ nor $\boldsymbol{y}$ is available.}

\quad\textbf{Case 2(a): there is some color $\boldsymbol{k\notin\{i,j\}}$ with $\boldsymbol{k\in c(z)}$ for some available vertex $\boldsymbol{z}$}.  We can use  \cref{lemma:leapfrog2} to leapfrog a robot from $z$ to add a $k$-colored robot to $x$. We can then use Case 1 to swap the $i$-colored robot on $x$ with the $j$-colored robot on $y$. Finally, we can undo the leapfrog to obtain $c'$. See \cref{fig:SwapColorsBase2a}.

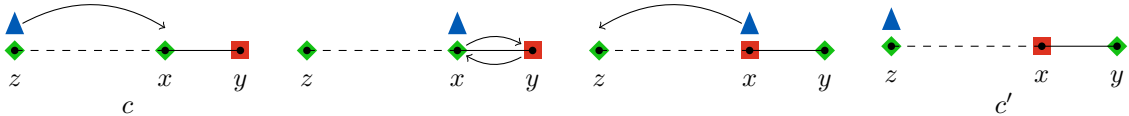
\begin{figure}[h!]
    \centering
    \begin{tikzpicture}
        \fg{
        \vertex{(-2,0)}{}
        \vertex{(0,0)}{}
        \vertex{(1,0)}{}
        \draw[dashed] (-2,0) to (0,0);
        \draw (0,0) to (1,0);
        }
        \GreenRobot{(0,0)}{g1}
        \RedRobot{(1,0)}{r1}
        \GreenRobot{(-2,0)}{g2}
        \BlueRobot{(-2,0.3)}{y1}
        \draw[bend left, ->] (y1) to ($(g1)+(0,0.3)$);
        \node[below,yshift=-2mm] at (-2,0) {$z$};
        \node[below,yshift=-2mm] at (0,0) {$x$};
        \node[below,yshift=-2mm] at (1,0) {$y$};
        \node at (-0.5,-.75) {$c$};
    \end{tikzpicture}
    \quad
    \begin{tikzpicture}
        \fg{
        \vertex{(-2,0)}{}
        \vertex{(0,0)}{}
        \vertex{(1,0)}{}
        \draw[dashed] (-2,0) to (0,0);
        \draw (0,0) to (1,0);}
        \GreenRobot{(0,0)}{g1}
        \RedRobot{(1,0)}{r1}
        \GreenRobot{(-2,0)}{g2}
        \BlueRobot{(0,0.3)}{y1}
        \draw[bend left, ->] (g1) to (r1);
        \draw[bend left, ->] (r1) to (g1);
        \node[below,yshift=-2mm] at (-2,0) {$z$};
        \node[below,yshift=-2mm] at (0,0) {$x$};
        \node[below,yshift=-2mm] at (1,0) {$y$};
        \node at (-0.5,-.75) {$\phantom c$};
    \end{tikzpicture}
    \quad
    \begin{tikzpicture}
        \fg{
        \vertex{(-2,0)}{}
        \vertex{(0,0)}{}
        \vertex{(1,0)}{}
        \draw[dashed] (-2,0) to  (0,0);
        \draw (0,0) to (1,0);}
        \GreenRobot{(1,0)}{g1}
        \RedRobot{(0,0)}{r1}
        \GreenRobot{(-2,0)}{g2}
        \BlueRobot{(0,0.3)}{y1}
        \draw[bend right, ->] (y1) to ($(g2)+(0,0.3)$);
        \node[below,yshift=-2mm] at (-2,0) {$z$};
        \node[below,yshift=-2mm] at (0,0) {$x$};
        \node[below,yshift=-2mm] at (1,0) {$y$};
        \node at (-0.5,-.75) {$\phantom c$};
    \end{tikzpicture}
    \quad
    \begin{tikzpicture}
        \fg{
        \vertex{(-2,0)}{}
        \vertex{(0,0)}{}
        \vertex{(1,0)}{}
        \draw[dashed] (-2,0) to  (0,0);
        \draw (0,0) to (1,0);}
        \GreenRobot{(1,0)}{g1}
        \RedRobot{(0,0)}{r1}
        \GreenRobot{(-2,0)}{g2}
        \BlueRobot{(-2,0.3)}{y1}
        \node[below,yshift=-2mm] at (-2,0) {$z$};
        \node[below,yshift=-2mm] at (0,0) {$x$};
        \node[below,yshift=-2mm] at (1,0) {$y$};
        \node at (-0.5,-.75) {$c'$};
    \end{tikzpicture}
    \caption{Case 2(a) of the base case for \cref{cor:SwapColors}}
    \label{fig:SwapColorsBase2a}
\end{figure}

\quad\textbf{Case 2(b): every available vertex contains only $\boldsymbol{i}$- and $\boldsymbol{j}$-colored robots}. Let $z$ be any available vertex (such a vertex exists because $\vec r$ is non-trivial), and let $w$ be a vertex with a $k$-colored robot, for some $k\notin\{i,j\}$, such that the path from $z$ to $w$ contains no $k$-colored robots other than the one on $w.$ We can use \cref{lemma:SwapThird} to swap the $k$-colored robot on $w$ with the $i$-colored robot on $z$. We can then use Case 2(a) to swap the $i$-colored robot on $x$ with the $j$-colored robot on $y$. Then we can again use \cref{lemma:SwapThird} to swap the $k$-colored robot on $z$ with the $i$-colored robot on $w.$ See \cref{fig:SwapColorsBase2b}.

\begin{figure}[h!]
    \centering
    \begin{tikzpicture}
    \fg{
        \node[below, yshift=-2mm] at (0,0) {$x$};
        \node[below, yshift=-2mm] at (1,0) {$y$};
        \node[below, yshift=-2mm] at (-1,0) {$z$};
        \node[below, yshift=-2mm] at (-2,0) {$w$};
        \vertex{(0,0)}{}
        \vertex{(1,0)}{}
        \vertex{(-1,0)}{}
        \vertex{(-2,0)}{}
        \draw (0,0) to (1,0);
        \draw[dashed] (-2,0) to (0,0);}
        \GreenRobot{(0,0)}{g1}
        \RedRobot{(1,0)}{r1}
        \RedRobot{(-1,0)}{r2}
        \GreenRobot{(-1,0.3)}{g2}
        \BlueRobot{(-2,0)}{y1}
        \draw[bend right,->] (y1) to (g2);
        \draw[bend right,->] (g2) to (y1);
        \node at (-0.5,-.75) {$c$};
    \end{tikzpicture}
    \quad
    \begin{tikzpicture}
        \fg{
        \node[below, yshift=-2mm] at (0,0) {$x$};
        \node[below, yshift=-2mm] at (1,0) {$y$};
        \node[below, yshift=-2mm] at (-1,0) {$z$};
        \node[below, yshift=-2mm] at (-2,0) {$w$};
        \vertex{(0,0)}{}
        \vertex{(1,0)}{}
        \vertex{(-1,0)}{}
        \vertex{(-2,0)}{}
        \draw (0,0) to (1,0);
        \draw[dashed] (-2,0) to (0,0);}
        \GreenRobot{(0,0)}{g1}
        \RedRobot{(1,0)}{r1}
        \RedRobot{(-1,0)}{r2}
        \GreenRobot{(-2,0)}{g2}
        \BlueRobot{(-1,0.3)}{y1}
        \draw[bend right,->] (g1) to (r1);
        \draw[bend right,->] (r1) to (g1);
        \node at (-0.5,-.75) {\phantom{$c'$}};
    \end{tikzpicture}
    \quad
    \begin{tikzpicture}
        \fg{
        \node[below, yshift=-2mm] at (0,0) {$x$};
        \node[below, yshift=-2mm] at (1,0) {$y$};
        \node[below, yshift=-2mm] at (-1,0) {$z$};
        \node[below, yshift=-2mm] at (-2,0) {$w$};
        \vertex{(0,0)}{}
        \vertex{(1,0)}{}
        \vertex{(-1,0)}{}
        \vertex{(-2,0)}{}
        \draw (0,0) to (1,0);
        \draw[dashed] (-2,0) to (0,0);}
        \GreenRobot{(1,0)}{g1}
        \RedRobot{(0,0)}{r1}
        \RedRobot{(-1,0)}{r2}
        \GreenRobot{(-2,0)}{g2}
        \BlueRobot{(-1,0.3)}{y1}
        \draw[bend right,->] (y1) to (g2);
        \draw[bend right,->] (g2) to (y1);
        \node at (-0.5,-.75) {\phantom{$c'$}};
    \end{tikzpicture}
    \quad
    \begin{tikzpicture}
        \fg{
        \node[below, yshift=-2mm] at (0,0) {$x$};
        \node[below, yshift=-2mm] at (1,0) {$y$};
        \node[below, yshift=-2mm] at (-1,0) {$z$};
        \node[below, yshift=-2mm] at (-2,0) {$w$};
        \vertex{(0,0)}{}
        \vertex{(1,0)}{}
        \vertex{(-1,0)}{}
        \vertex{(-2,0)}{}
        \draw (0,0) to (1,0);
        \draw[dashed] (-2,0) to (0,0);}
        \GreenRobot{(1,0)}{g1}
        \RedRobot{(0,0)}{r1}
        \RedRobot{(-1,0)}{r2}
        \GreenRobot{(-1,0.3)}{g2}
        \BlueRobot{(-2,0)}{y1}
        \node at (-0.5,-.75) {$c'$};
    \end{tikzpicture}
    \caption{Case 2(b) of the base case for \cref{cor:SwapColors}}
    \label{fig:SwapColorsBase2b}
\end{figure}
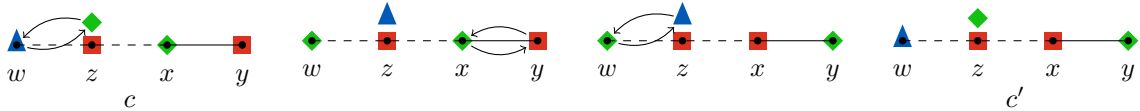

For the inductive step, assume the claim is true for paths of length $s$ and consider a case where the length of the path is $s+1$. Let $\tilde x$ be the vertex adjacent to $x$ on $P$. We will consider the following cases:

\textbf{Case 1: $\boldsymbol{i\in c(\tilde x),\ j\notin c(\tilde x)}$}. Using the inductive hypothesis, we can swap the $i$-colored robot on $\tilde x$ with the $j$-colored robot on $y$. Then, using the base case, we can swap the $i$-colored robot on $x$ with the $j$-colored robot on $\tilde x$ to obtain $c'$.  See \cref{fig:SwapColorsInductive1}.

\begin{figure}[h!]
    \centering
    \begin{tikzpicture}
        \fg{\node[below, yshift=-2mm] at (0,0) {$x$};
        \node[below, yshift=-2mm] at (1,0) {$\tilde x$};
        \node[below, yshift=-2mm] at (3,0) {$y$};
        \vertex{(0,0)}{}
        \vertex{(1,0)}{}
        \vertex{(3,0)}{}
        \draw (0,0) to (1,0);
        \draw[dashed] (1,0) to (3,0);}
        \GreenRobot{(0,0)}{g1}
        \GreenRobot{(1,0)}{g2}
        \RedRobot{(3,0)}{r1}
        \draw[bend left , ->](g2) to (r1);
        \draw[bend left, ->] (r1) to (g2);
        \node at (1.5,-.7){$c$};
    \end{tikzpicture}
    \quad
    \begin{tikzpicture}
     \fg{\node[below, yshift=-2mm] at (0,0) {$x$};
        \node[below, yshift=-2mm] at (1,0) {$\tilde x$};
        \node[below, yshift=-2mm] at (3,0) {$y$};
        \vertex{(0,0)}{}
        \vertex{(1,0)}{}
        \vertex{(3,0)}{}
        \draw (0,0) to (1,0);
        \draw[dashed] (1,0) to (3,0);}
        \GreenRobot{(0,0)}{g1}
        \GreenRobot{(3,0)}{g2}
        \RedRobot{(1,0)}{r1}
        \draw[bend left, ->] (g1) to (r1);
        \draw[bend left, ->](r1) to (g1); 
         \node at (1.5,-.7){$\phantom c$};
    \end{tikzpicture}
    \quad
        \begin{tikzpicture}
     \fg{\node[below, yshift=-2mm] at (0,0) {$x$};
        \node[below, yshift=-2mm] at (1,0) {$\tilde x$};
        \node[below, yshift=-2mm] at (3,0) {$y$};
        \vertex{(0,0)}{}
        \vertex{(1,0)}{}
        \vertex{(3,0)}{}
        \draw (0,0) to (1,0);
        \draw[dashed] (1,0) to (3,0);}
        \GreenRobot{(1,0)}{g1}
        \GreenRobot{(3,0)}{g2}
        \RedRobot{(0,0)}{r1}
         \node at (1.5,-.7){$c'$};
    \end{tikzpicture}
    \caption{Case 1 of the inductive step of \cref{cor:SwapColors}}
    \label{fig:SwapColorsInductive1}
\end{figure}

\textbf{Case 2: $\boldsymbol{j\in c(\tilde x),\ i\notin c(\tilde x)}$}. Using the base case, we can swap the $i$-colored robot on $x$ with the $j$-colored robot on $\tilde x$, then use the inductive hypothesis to swap the $i$-colored robot on $\tilde x$ with the $j$-colored robot on $y$ to obtain $c'$.  See \cref{fig:SwapColorsInductive2}.

\begin{figure}[h!]
    \centering
    \begin{tikzpicture}
        \fg{\node[below, yshift=-2mm] at (0,0) {$x$};
        \node[below, yshift=-2mm] at (1,0) {$\tilde x$};
        \node[below, yshift=-2mm] at (3,0) {$y$};
        \vertex{(0,0)}{}
        \vertex{(1,0)}{}
        \vertex{(3,0)}{}
        \draw (0,0) to (1,0);
        \draw[dashed] (1,0) to (3,0);}
        \GreenRobot{(0,0)}{g1}
        \RedRobot{(3,0)}{r2}
        \RedRobot{(1,0)}{r1}
        \draw[bend left , ->](g1) to (r1);
        \draw[bend left, ->] (r1) to (g1);
        \node at (1.5,-.7){$c$};
    \end{tikzpicture}
    \quad
    \begin{tikzpicture}
     \fg{\node[below, yshift=-2mm] at (0,0) {$x$};
        \node[below, yshift=-2mm] at (1,0) {$\tilde x$};
        \node[below, yshift=-2mm] at (3,0) {$y$};
        \vertex{(0,0)}{}
        \vertex{(1,0)}{}
        \vertex{(3,0)}{}
        \draw (0,0) to (1,0);
        \draw[dashed] (1,0) to (3,0);}
        \GreenRobot{(1,0)}{g1}
        \RedRobot{(3,0)}{r2}
        \RedRobot{(0,0)}{r1}
        \draw[bend left, ->] (g1) to (r2);
        \draw[bend left, ->](r2) to (g1); 
         \node at (1.5,-.7){$\phantom c$};
    \end{tikzpicture}
    \quad
        \begin{tikzpicture}
     \fg{\node[below, yshift=-2mm] at (0,0) {$x$};
        \node[below, yshift=-2mm] at (1,0) {$\tilde x$};
        \node[below, yshift=-2mm] at (3,0) {$y$};
        \vertex{(0,0)}{}
        \vertex{(1,0)}{}
        \vertex{(3,0)}{}
        \draw (0,0) to (1,0);
        \draw[dashed] (1,0) to (3,0);}
        \GreenRobot{(3,0)}{g1}
        \RedRobot{(0,0)}{r2}
        \RedRobot{(1,0)}{r1};
         \node at (1.5,-.7){$c'$};
    \end{tikzpicture}
    \caption{Case 2 of the inductive step of \cref{cor:SwapColors}}
    \label{fig:SwapColorsInductive2}
\end{figure}

\textbf{Case 3: $\boldsymbol{i,j\in c(\tilde x)}$}. We can move the $j$-colored robot from $\tilde x$ onto $x$, then use the inductive hypothesis to swap the $i$-colored robot on $\tilde x$ with the $j$-colored robot on $y$, then finally move the $i$-colored robot from $x$ onto $\tilde x$ to obtain $c'$.   See \cref{fig:SwapColorsInductive3}.

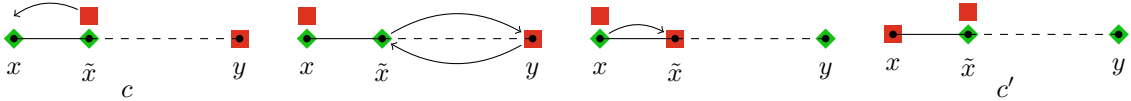
\begin{figure}[h!]
    \centering
    \begin{tikzpicture}
        \fg{\node[below, yshift=-2mm] at (0,0) {$x$};
        \node[below, yshift=-2mm] at (1,0) {$\tilde x$};
        \node[below, yshift=-2mm] at (3,0) {$y$};
        \vertex{(0,0)}{}
        \vertex{(1,0)}{}
        \vertex{(3,0)}{}
        \draw (0,0) to (1,0);
        \draw[dashed] (1,0) to (3,0);}
        \GreenRobot{(0,0)}{g2}
        \RedRobot{(3,0)}{r2}
        \RedRobot{(1,0.3)}{r1}
        \GreenRobot{(1,0)}{g1}
        \draw[bend right, ->](r1) to ($(g2)+(0,0.3)$);
        \node at (1.5,-.7){$c$};
    \end{tikzpicture}
    \quad
    \begin{tikzpicture}
     \fg{\node[below, yshift=-2mm] at (0,0) {$x$};
        \node[below, yshift=-2mm] at (1,0) {$\tilde x$};
        \node[below, yshift=-2mm] at (3,0) {$y$};
        \vertex{(0,0)}{}
        \vertex{(1,0)}{}
        \vertex{(3,0)}{}
        \draw (0,0) to (1,0);
        \draw[dashed] (1,0) to (3,0);}
        \GreenRobot{(1,0)}{g2}
        \RedRobot{(3,0)}{r2}
        \RedRobot{(0,0.3)}{r1}
        \GreenRobot{(0,0)}{g1}
        \draw[bend left, ->] (g2) to (r2);
        \draw[bend left, ->](r2) to (g2); 
         \node at (1.5,-.7){$\phantom c$};
    \end{tikzpicture}
    \quad
        \begin{tikzpicture}
        \fg{\node[below, yshift=-2mm] at (0,0) {$x$};
        \node[below, yshift=-2mm] at (1,0) {$\tilde x$};
        \node[below, yshift=-2mm] at (3,0) {$y$};
        \vertex{(0,0)}{}
        \vertex{(1,0)}{}
        \vertex{(3,0)}{}
        \draw (0,0) to (1,0);
        \draw[dashed] (1,0) to (3,0);}
        \GreenRobot{(3,0)}{g2}
        \RedRobot{(1,0)}{r2}
        \RedRobot{(0,0.3)}{r1}
        \GreenRobot{(0,0)}{g1}
         \draw[bend left, ->] (g1) to (r2);
         \node at (1.5,-.7){$\phantom c$};
    \end{tikzpicture}
        \quad
        \begin{tikzpicture}
     \fg{\node[below, yshift=-2mm] at (0,0) {$x$};
        \node[below, yshift=-2mm] at (1,0) {$\tilde x$};
        \node[below, yshift=-2mm] at (3,0) {$y$};
        \vertex{(0,0)}{}
        \vertex{(1,0)}{}
        \vertex{(3,0)}{}
        \draw (0,0) to (1,0);
        \draw[dashed] (1,0) to (3,0);}
        \GreenRobot{(3,0)}{g2}
        \RedRobot{(1,0.3)}{r2}
        \RedRobot{(0,0)}{r1}
        \GreenRobot{(1,0)}{g1}
         \node at (1.5,-.7){$c'$};
    \end{tikzpicture}
    \caption{Case 3 of the inductive step of \cref{cor:SwapColors}}
    \label{fig:SwapColorsInductive3}
\end{figure}

\textbf{Case 4: $\boldsymbol{i,j\notin c(\tilde x)}$}.

\quad\textbf{Case 4(a): $\boldsymbol{x}$ is not available}.
Using the base case, we can swap the $i$-colored robot on $x$ and a $k$-colored robot on $\tilde x$. Then, using the inductive hypothesis, we can swap the $i$-colored robot on $\tilde x$ with the $j$-colored robot on $y$. Using the base case again, we can swap the $k$-colored robot on $x$ with the $j$-colored robot on $\tilde x$ to obtain $c'$.  See \cref{fig:SwapColorsInductive4a}.

\begin{figure}[h!]
    \centering
    \begin{tikzpicture}
       \fg{\node[below, yshift=-2mm] at (0,0) {$x$};
        \node[below, yshift=-2mm] at (1,0) {$\tilde x$};
        \node[below, yshift=-2mm] at (3,0) {$y$};
        \vertex{(0,0)}{}
        \vertex{(1,0)}{}
        \vertex{(3,0)}{}
        \draw (0,0) to (1,0);
        \draw[dashed] (1,0) to (3,0);}
        \GreenRobot{(0,0)}{g1}
        \RedRobot{(3,0)}{r1}
        \BlueRobot{(1,0)}{y1}
        \draw[bend right, ->](g1) to (y1);
        \draw[bend right, ->](y1) to (g1);
        \node at (1.5,-.7){$c$};
    \end{tikzpicture}
    \quad
    \begin{tikzpicture}
     \fg{\node[below, yshift=-2mm] at (0,0) {$x$};
        \node[below, yshift=-2mm] at (1,0) {$\tilde x$};
        \node[below, yshift=-2mm] at (3,0) {$y$};
        \vertex{(0,0)}{}
        \vertex{(1,0)}{}
        \vertex{(3,0)}{}
        \draw (0,0) to (1,0);
        \draw[dashed] (1,0) to (3,0);}
        \GreenRobot{(1,0)}{g1}
        \RedRobot{(3,0)}{r1}
        \BlueRobot{(0,0)}{y1}
        \draw[bend right, ->] (g1) to (r1);
        \draw[bend right, ->](r1) to (g1); 
         \node at (1.5,-.7){$\phantom c$};
    \end{tikzpicture}
    \quad
        \begin{tikzpicture}
     \fg{\node[below, yshift=-2mm] at (0,0) {$x$};
        \node[below, yshift=-2mm] at (1,0) {$\tilde x$};
        \node[below, yshift=-2mm] at (3,0) {$y$};
        \vertex{(0,0)}{}
        \vertex{(1,0)}{}
        \vertex{(3,0)}{}
        \draw (0,0) to (1,0);
        \draw[dashed] (1,0) to (3,0);}
        \GreenRobot{(3,0)}{g1}
        \RedRobot{(1,0)}{r1}
        \BlueRobot{(0,0)}{y1}
        \draw[bend right, ->] (r1) to (y1);
        \draw[bend right, ->](y1) to (r1); 
         \node at (1.5,-.7){$\phantom c$};
    \end{tikzpicture}
        \quad
        \begin{tikzpicture}
     \fg{\node[below, yshift=-2mm] at (0,0) {$x$};
        \node[below, yshift=-2mm] at (1,0) {$\tilde x$};
        \node[below, yshift=-2mm] at (3,0) {$y$};
        \vertex{(0,0)}{}
        \vertex{(1,0)}{}
        \vertex{(3,0)}{}
        \draw (0,0) to (1,0);
        \draw[dashed] (1,0) to (3,0);}
        \GreenRobot{(3,0)}{g1}
        \RedRobot{(0,0)}{r1}
        \BlueRobot{(1,0)}{y1}
         \node at (1.5,-.7){$c'$};
    \end{tikzpicture}
    \caption{Case 4(a) of the inductive step of \cref{cor:SwapColors}}
    \label{fig:SwapColorsInductive4a}
\end{figure}

\quad\textbf{Case 4(b): $\boldsymbol{x}$ is available}. We can move the $i$-colored robot from $x$ onto $\tilde x$. Using the inductive hypothesis, we can swap the $i$-colored robot on $\tilde x$ with the $j$-colored robot on $y$. Then, we can move the $j$-colored robot from $\tilde x$ onto $x$ to obtain $c'$.  See \cref{fig:SwapColorsInductive4b}.

\begin{figure}[h!]
    \centering
    \begin{tikzpicture}
        \fg{\node[below, yshift=-2mm] at (0,0) {$x$};
        \node[below, yshift=-2mm] at (1,0) {$\tilde x$};
        \node[below, yshift=-2mm] at (3,0) {$y$};
        \vertex{(0,0)}{}
        \vertex{(1,0)}{}
        \vertex{(3,0)}{}
        \draw (0,0) to (1,0);
        \draw[dashed] (1,0) to (3,0);}
        \GreenRobot{(0,0)}{g1}
        \RedRobot{(3,0)}{r1}
        \BlueRobot{(1,0)}{y1}
        \BlueRobot{(0,0.3)}{y2}
        \draw[bend left, ->](g1) to ($(y1) +(0,0.3)$); 
        \node at (1.5,-.7){$c$};
    \end{tikzpicture}
    \quad
    \begin{tikzpicture}
     \fg{\node[below, yshift=-2mm] at (0,0) {$x$};
        \node[below, yshift=-2mm] at (1,0) {$\tilde x$};
        \node[below, yshift=-2mm] at (3,0) {$y$};
        \vertex{(0,0)}{}
        \vertex{(1,0)}{}
        \vertex{(3,0)}{}
        \draw (0,0) to (1,0);
        \draw[dashed] (1,0) to (3,0);}
        \GreenRobot{(1,0)}{g1}
        \RedRobot{(3,0)}{r1}
        \BlueRobot{(1,0.3)}{y1}
        \BlueRobot{(0,0)}{y2}
        \draw[bend right, ->] (g1) to (r1);
        \draw[bend right, ->](r1) to (g1); 
         \node at (1.5,-.7){$\phantom c$};
    \end{tikzpicture}
    \quad
        \begin{tikzpicture}
     \fg{\node[below, yshift=-2mm] at (0,0) {$x$};
        \node[below, yshift=-2mm] at (1,0) {$\tilde x$};
        \node[below, yshift=-2mm] at (3,0) {$y$};
        \vertex{(0,0)}{}
        \vertex{(1,0)}{}
        \vertex{(3,0)}{}
        \draw (0,0) to (1,0);
        \draw[dashed] (1,0) to (3,0);}
        \GreenRobot{(3,0)}{g1}
        \RedRobot{(1,0)}{r1}
        \BlueRobot{(1,0.3)}{y1}
        \BlueRobot{(0,0)}{y2}
        \draw[bend right, ->] (r1) to ($(y2) +(0,0.3)$); 
         \node at (1.5,-.7){$\phantom c$};
    \end{tikzpicture}
        \quad
        \begin{tikzpicture}
    \fg{\node[below, yshift=-2mm] at (0,0) {$x$};
        \node[below, yshift=-2mm] at (1,0) {$\tilde x$};
        \node[below, yshift=-2mm] at (3,0) {$y$};
        \vertex{(0,0)}{}
        \vertex{(1,0)}{}
        \vertex{(3,0)}{}
        \draw (0,0) to (1,0);
        \draw[dashed] (1,0) to (3,0);}
        \GreenRobot{(3,0)}{g1}
        \RedRobot{(0,0)}{r1}
        \BlueRobot{(1,0)}{y1}
        \BlueRobot{(0,0.3)}{y2}
         \node at (1.5,-.7){$c'$};
    \end{tikzpicture}
    \caption{Case 4(b) of the inductive step of \cref{cor:SwapColors}}
    \label{fig:SwapColorsInductive4b}
\end{figure}
\end{proof}

Next, we show that if $c$ and $c'$ are $0$-cells such that $|c(v)|=|c'(v)|$ for all $v\in V(G),$ there is a path in $\Sr(G)$ from $c$ to $c'. $ In this case, we say that $c$ and $c'$ have the same \textit{type}.

\begin{defn}\label[defn]{defn:type}
    Given  0-cells $c$ and $c'$ and a vertex $v\in V(G)$, let $d_v(c,c')=|c(v)|-|c'(v)|.$ We say $c$ and $c'$ have the same \textit{type} if $d_v(c,c')=0$ for all $v\in V(G)$.
\end{defn}

\begin{lemma}\label[cor]{cor:SameType}
    If $G$ is a connected graph, and $c$ and $c'$ are 0-cells of the same type in $\Sr(G),$ there is a path in $\Sr(G)$ from $c$ to $c'$.
\end{lemma}

\begin{proof}
Consider two 0-cells $c$ and $c'$ of the same type, as in \cref{fig:sameType2}. We say that a color $i$ is an \textit{extra color} on a vertex $v$ if $i\in c(v)$ but $i\notin c'(v).$ For example, in \cref{fig:sameType2}, color 2 (green/diamond) is an extra color on $v_1,$ since there is a green/diamond robot on $v_1$ in $c,$ but not in $c'.$ If $i$ is an extra color on $v$, there must be some other vertex $u$ such that $i\in c'(u)$ but $i\notin c(u);$ we say $i$ is a \textit{missing color} on $u$. In \cref{fig:sameType2}, green/diamond is missing on $v_2$ since there is a green/diamond robot on $v_2$ in $c',$ but not in $c.$ Since $|c(u)|=|c'(u)|,$ and $u$ has a missing color, $u$ must also have some extra color $j$. In \cref{fig:sameType2}, red/square is an extra color on $v_2$. We use the notation $(i,v)\to(j,u)$ to denote that $i$ is an extra color on $v$, $i$ is a missing color on $u,$ and $j$ is an extra color on $u.$ The relationship $(2,v_1)\to(1,v_2)$ is shown in \cref{fig:sameType2} as an arc from the 2-colored (green/diamond) robot on $v_1$ to the 1-colored (red/square) robot on $v_2.$

\begin{figure}[h!]
    \centering
    \begin{tikzpicture}
        \Xgraph
        \node[below,yshift=-1mm] at (A) {$v_1$};
        \node[above left] at (B) {$v_2$};
        \node[below,yshift=-1mm] at (C) {$v_3$};
        \node[below left] at (D) {$v_4$};
        \node[below,yshift=-1mm] at (E) {$v_5$};
        \GreenRobot{(A)}{r1}
        \YellowRobot{($(A)+(0.3,0)$)}{y2}
        \RedRobot{(B)}{g1}
        \BlueRobot{(C)}{b1}
        \YellowRobot{($(C)+(0.3,0)$)}{y1}
        \GreenRobot{(D)}{r2}
        \RedRobot{($(D)+(0,.3)$)}{g2}
        \GreenRobot{(E)}{r3}
        \BlueRobot{($(E)+(-.3,0)$)}{b2}
        \draw[->,bend right] (r1) to (g1);
        \draw[->,bend right] (g1) to (b1);
        \draw[->,bend right] (b1) to (g2);
        \draw[->,bend right] (g2) to (b2);
        \draw[->,bend right] (b2) to (r1);
        \node at (0,-2) {$c$};
    \end{tikzpicture}
    \quaad  
    \begin{tikzpicture}
        \Xgraph
        \node[below,yshift=-1mm] at (A) {$v_1$};
        \node[above left] at (B) {$v_2$};
        \node[below,yshift=-1mm] at (C) {$v_3$};
        \node[below left] at (D) {$v_4$};
        \node[below,yshift=-1mm] at (E) {$v_5$};
        \BlueRobot{(A)}{r1}
        \YellowRobot{($(A)+(0.3,0)$)}{y2}
        \GreenRobot{(B)}{g1}
        \RedRobot{(C)}{b1}
        \GreenRobot{($(C)+(0.3,0)$)}{y1}
        \YellowRobot{(D)}{r2}
        \BlueRobot{($(D)+(0,.3)$)}{g2}
        \GreenRobot{(E)}{r3}
        \RedRobot{($(E)+(-.3,0)$)}{b2}
        \node at (0,-2) {$c'$};
    \end{tikzpicture}
    \caption{Two cells $c,c'$ of the same type, and a cycle in $c$}
    \label{fig:sameType2}
\end{figure}
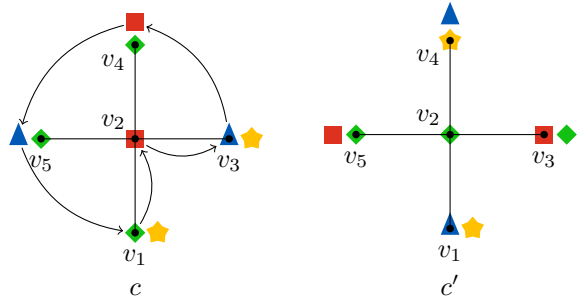

We can create a cycle of the form 
\begin{equation}\label{eqn:Cycle}
    (i_1,v_1)\to(i_2,v_2)\to\cdots\to(i_p,v_p)\to(i_{p+1},v_{p+1}),
\end{equation}
where the vertices $v_1,\dots,v_p$ are distinct, $i_{p+1}=i_1$ and $v_{p+1}=v_1,$ and $i_t$ is an extra color on $v_t$ and a missing color on $v_{t+1}$ for $t=1,\dots,p.$  \cref{fig:sameType2} shows the cycle $(2,v_1)\to (1,v_2)\to(3,v_3)\to(1,v_4)\to(3,v_5)\to(2,v_1).$

The cycle in (\ref{eqn:Cycle}) defines a permutation of the robots in which the extra $i_t$-colored robot on $v_t$  gets moved to the vertex $v_{t+1},$ on which $i_t$ is missing. We will show that we can use a sequence of color swaps to achieve this permutation, giving a path (via \cref{cor:SwapColors}) from $c$ to another 0-cell with fewer extra/missing colors.

Let $t\in\{2,3,\dots,p+1\}$ be the minimal value such that $i_1\in c(v_t)$. In other words, $v_t$ is the first vertex in the cycle after $v_1$ which contains an $i_1$-colored robot ($i_1=\text{green/diamond}$ and $t=4$ in \cref{fig:sameType2}). Notice that since $i_1$ is missing on $v_2,$ we must actually have $t\ge 3.$ Since $i_{t-1}$ is missing on $v_t,$ and there is no $i_1$-colored robot on $v_{t-1}$ (by the minimality of $t$), we can swap the $i_1$-colored robot on $v_t$ with the $i_{t-1}$-colored robot on $v_{t-1}.$ See the left side of \cref{fig:colorSwapSequence}. After this, we can similarly swap the $i_1$-colored robot on $v_{t-1}$ with the $i_{t-2}$-colored robot on $v_{t-2}$ (see the middle of \cref{fig:colorSwapSequence}), and continue this process until we have an $i_1$-colored robot on $v_2$, and $i_s$-colored robots on $v_{s+1}$ for $s=2,\dots,t-1.$

\begin{figure}[h!]
    \centering
    \begin{tikzpicture}
        \Xgraph
        \node[below,yshift=-1mm] at (A) {$v_1$};
        \node[above left] at (B) {$v_2$};
        \node[below,yshift=-1mm] at (C) {$v_3$};
        \node[below left] at (D) {$v_4$};
        \node[below,yshift=-1mm] at (E) {$v_5$};
        \GreenRobot{(A)}{r1}
        \YellowRobot{($(A)+(0.3,0)$)}{y2}
        \RedRobot{(B)}{g1}
        \BlueRobot{(C)}{b1}
        \YellowRobot{($(C)+(0.3,0)$)}{y1}
        \GreenRobot{(D)}{r2}
        \RedRobot{($(D)+(0,.3)$)}{g2}
        \GreenRobot{(E)}{r3}
        \BlueRobot{($(E)+(-.3,0)$)}{b2}
        \draw[<->,bend right] (b1) to  (r2);
        \node at (0,-2) {$c$};
    \end{tikzpicture}
    \quaad
    \begin{tikzpicture}
        \Xgraph
        \node[below,yshift=-1mm] at (A) {$v_1$};
        \node[above left] at (B) {$v_2$};
        \node[below,yshift=-1mm] at (C) {$v_3$};
        \node[below left] at (D) {$v_4$};
        \node[below,yshift=-1mm] at (E) {$v_5$};
        \GreenRobot{(A)}{r1}
        \YellowRobot{($(A)+(0.3,0)$)}{y2}
        \RedRobot{(B)}{g1}
        \GreenRobot{(C)}{r2}
        \YellowRobot{($(C)+(0.3,0)$)}{y1}
        \BlueRobot{(D)}{b1}
        \RedRobot{($(D)+(0,.3)$)}{g2}
        \GreenRobot{(E)}{r3}
        \BlueRobot{($(E)+(-.3,0)$)}{b2}
        \draw[<->,bend right] (g1) to  (r2);
        \node at (0,-2) {$\phantom c$};
    \end{tikzpicture}
    \quaad
    \begin{tikzpicture}
        \Xgraph
        \node[below,yshift=-1mm] at (A) {$v_1$};
        \node[above left] at (B) {$v_2$};
        \node[below,yshift=-1mm] at (C) {$v_3$};
        \node[below left] at (D) {$v_4$};
        \node[below,yshift=-1mm] at (E) {$v_5$};
        \GreenRobot{(A)}{r1}
        \YellowRobot{($(A)+(0.3,0)$)}{y2}
        \GreenRobot{(B)}{r2}
        \RedRobot{(C)}{g1}
        \YellowRobot{($(C)+(0.3,0)$)}{y1}
        \BlueRobot{(D)}{b1}
        \RedRobot{($(D)+(0,.3)$)}{g2}
        \GreenRobot{(E)}{r3}
        \BlueRobot{($(E)+(-.3,0)$)}{b2}
        \draw[->,bend right] (r1) to (g2);
        \draw[->,bend right] (g2) to (b2);
        \draw[->,bend right] (b2) to (r1);
        \node at (0,-2) {$\tilde{c}$};
    \end{tikzpicture}
    \quaad
    \caption{A sequence of color swaps (left and middle); a smaller cycle in $\tilde c$ (right)}
    \label{fig:colorSwapSequence}
\end{figure}
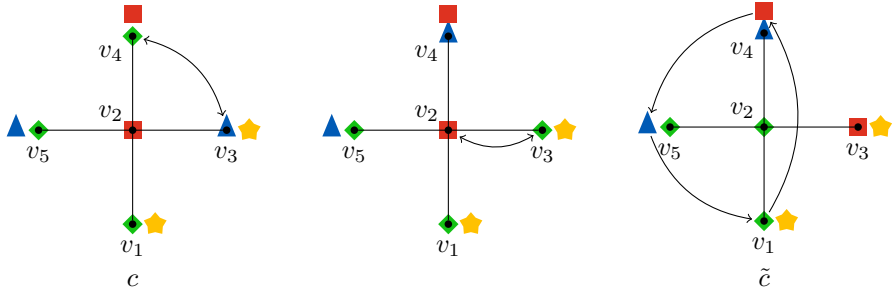

If $t=p+1$ (i.e. the only $i_1$-colored robot is on $v_{p+1}=v_1$), this sequence of color swaps achieves the permutation corresponding to (\ref{eqn:Cycle}). Otherwise, denote the resulting cell by $\tilde c$ and consider the cycle 
\begin{equation}\label{eqn:Cycle2}
(i_1,v_1)\to(i_{t'},v_{t'})\to(i_{t'+1},v_{t'+1})\to\cdots\to(i_p,v_p)\to(i_{p+1},v_{p+1}),
\end{equation}
where $t'=t$ if $i_t\ne i_1$ (so that the $i_t$-colored robot was not removed from $v_t$), and $t'=t+1$ if $i_t=i_1.$ See the right side of \cref{fig:colorSwapSequence}; in this case, we have $t'=t=4$, since $i_t=1$ (red/square), but $i_1=2$ (green/diamond). 

Using the cycle in (\ref{eqn:Cycle2}) to permute the robots on $\tilde c$ results in the same configuration as using the cycle in (\ref{eqn:Cycle}) to permute the robots on $c,$ so that we've essentially reduced the cycle in (\ref{eqn:Cycle}) to the smaller cycle in (\ref{eqn:Cycle2}). We can repeat this process to reduce each cycle to a smaller cycle using color swaps, until we've completed the original cycle in (\ref{eqn:Cycle}). This gives a path from $c$ to another 0-cell with fewer extra/missing vertices. We can continue using cycles in this fashion until there are no more extra/missing vertices. All together, this gives a path from $c$ to $c'$ using a sequence of color swaps. 
\end{proof}

\begin{thm}\label[thm]{thm:SrGConnected}
    If $G$ is a connected graph and $\vec r$ is a non-trivial color vector with at least three colors, then $\Sr(G)$ is path-connected.
\end{thm}

\begin{proof}
    Let $A$ and $B$ be any two configurations in $\Sr(G)$. If $A$ has any robots on any edges of $G$, we can move each such robot onto one of the endpoints of the edge it occupies to obtain a path $\gamma_{A}$ from $A$ to a $0$-cell $c.$ Likewise, we can obtain a path $\gamma_B$ from $B$ to a 0-cell $c'$. See \cref{fig:pathconnected}. If $c$ and $c'$ are of the same type, \cref{cor:SameType} gives a path $\gamma$ from $c$ to $c'.$ Otherwise, there are vertices $x$ and $y$ such that $d_x(c,c')>0$ (i.e. $|c(x)|>|c'(x)|$) and $d_y(c,c')<0$ (i.e. $|c(y)|<|c'(y)|$). 
    
    If $|c(x)|>|c(y)|,$ there must be some color $k$ with $k\in c(x)$ but $k\notin c(y).$ In this case, we can use \cref{lemma:leapfrog2}, to leapfrog a robot from $x$ to add a $k$-colored robot to $y$ in $c.$ Denote the resulting cell by $c_1$ and denote the path from $c$ to $c_1$ by $\gamma_1$. For book-keeping purposes, let $c'_1=c'$ and denote the constant path from $c'$ to $c_1'$ by $\gamma_1'.$ See \cref{fig:excess1}.

    \newsavebox{\cellc}
    \savebox{\cellc}{
    \begin{tikzpicture}
            \fg{\vertex{(0,0)}{A}
            \vertex{(1.5,0)}{B}
            \node[below,yshift =-1mm] at (A) {$x$};
            \node[below,yshift =-1mm] at (B) {$y$};
            \draw[dashed] (A) to (B);
            }
            \RedRobot{(A)}{}
            \RedRobot{(B)}{}
            \GreenRobot{($(A)+(0,0.3)$)}{g1}
            \GreenRobot{($(B)+(0,0.3)$)}{g2}
            \BlueRobot{($(A)+(0,0.6)$)}{b}
            \draw[bend left, ->] (b) to  ($(g2)+(0,0.3)$);
            \node at (0.75,-.65) {$c$};
        \end{tikzpicture}
    }
    \newsavebox{\cellcone}
    \savebox{\cellcone}{
    \begin{tikzpicture}
            \fg{\vertex{(0,0)}{A}
            \vertex{(1.5,0)}{B}
            \node[below,yshift =-1mm] at (A) {$x$};
            \node[below,yshift =-1mm] at (B) {$y$};
            \draw[dashed] (A) to (B);
            }
            \RedRobot{(A)}{}
            \RedRobot{(B)}{}
            \GreenRobot{($(A)+(0,0.3)$)}{}
            \GreenRobot{($(B)+(0,0.3)$)}{}
            \BlueRobot{($(B)+(0,0.6)$)}{}
            \node at (0.75,-.65) {$c_1$};
        \end{tikzpicture}
    }
    \newsavebox{\cellcp}
    \savebox{\cellcp}{
    \begin{tikzpicture}
            \fg{\vertex{(0,0)}{A}
            \vertex{(1.5,0)}{B}
            \node[below,yshift =-1mm] at (A) {$x$};
            \node[below,yshift =-1mm] at (B) {$y$};
            \draw[dashed] (A) to (B);
            }
            \YellowRobot{(A)}{y}
            \RedRobot{(B)}{r}
            \GreenRobot{($(B)+(0,0.3)$)}{}
            \BlueRobot{($(B)+(0,0.6)$)}{}
            \node at (0.75,-.65) {$c'$};
        \end{tikzpicture}
    }
    \newsavebox{\cellcpone}
    \savebox{\cellcpone}{
    \begin{tikzpicture}
            \fg{\vertex{(0,0)}{A}
            \vertex{(1.5,0)}{B}
            \node[below,yshift =-1mm] at (A) {$x$};
            \node[below,yshift =-1mm] at (B) {$y$};
            \draw[dashed] (A) to (B);
            }
            \YellowRobot{(A)}{}
            \RedRobot{(B)}{}
            \GreenRobot{($(B)+(0,0.3)$)}{}
            \BlueRobot{($(B)+(0,0.6)$)}{}
            \node at (0.75,-.65) {$c'_1$};
        \end{tikzpicture}
    }

    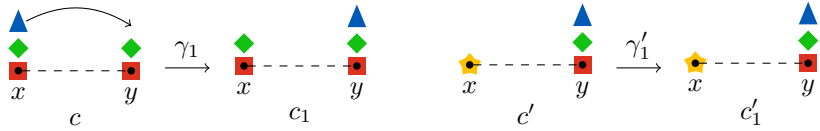
\begin{figure}[h!]
        \centering
        \begin{tikzpicture}
            \node (cellc) at (0,0) {\usebox{\cellc}};
            \node (cellcone) at (3,0) {\usebox{\cellcone}};
            \draw[->] (cellc) to node[above]{$\gamma_1$} (cellcone);
            \node (cellcp) at (6,0) {\usebox{\cellcp}};
            \node (cellcpone) at (9,0) {\usebox{\cellcpone}};
            \draw[->] (cellcp) to node[above]{$\gamma'_1$} (cellcpone);
        \end{tikzpicture}
        \caption{$d_x(c,c')=2;\ d_y(c,c')=-1;\ d_x(c_1,c'_1)=1;\ d_y(c_1,c'_1)=0$}
        \label{fig:excess1}
    \end{figure}

    If $|c(x)|\le |c(y)|,$ then since $|c'(x)|<|c(x)|$ and $|c(y)|<|c'(y)|,$ we must have $|c'(x)|<|c'(y)|,$ so there must be some color $k$ with $k\in c'(y)$ but $k\notin c'(x).$ In this case, we can leapfrog a robot from $y$ to add a $k$-colored robot to $x$ in $c'.$ Denote the resulting cell by $c_1',$ denote the path from $c'$ to $c'_1$ by $\gamma_1'$, let $c_1=c$ and denote the constant path from $c$ to $c_1$ by $\gamma_1.$ See \cref{fig:excess2}.

    \savebox{\cellc}{
    \begin{tikzpicture}
            \fg{\vertex{(0,0)}{A}
            \vertex{(1.5,0)}{B}
            \node[below,yshift =-1mm] at (A) {$x$};
            \node[below,yshift =-1mm] at (B) {$y$};
            \draw[dashed] (A) to (B);
            }
            \RedRobot{(A)}{}
            \RedRobot{(B)}{}
            \node at (0,1.0) {};
            \GreenRobot{($(A)+(0,0.3)$)}{}
            \GreenRobot{($(B)+(0,0.3)$)}{}
            \node at (0.75,-.75) {$c$};
        \end{tikzpicture}
    }
    \savebox{\cellcone}{
    \begin{tikzpicture}
            \fg{\vertex{(0,0)}{A}
            \vertex{(1.5,0)}{B}
            \node at (0,1.0) {};
            \node[below,yshift =-1mm] at (A) {$x$};
            \node[below,yshift =-1mm] at (B) {$y$};
            \draw[dashed] (A) to (B);
            }
            \RedRobot{(A)}{}
            \RedRobot{(B)}{}
            \GreenRobot{($(A)+(0,0.3)$)}{}
            \GreenRobot{($(B)+(0,0.3)$)}{}
            \node at (0.75,-.65) {$c_1$};
        \end{tikzpicture}
    }
    \savebox{\cellcp}{
    \begin{tikzpicture}
            \fg{\vertex{(0,0)}{A}
            \vertex{(1.5,0)}{B}
            \node at (0,1.0) {};
            \node[below,yshift =-1mm] at (A) {$x$};
            \node[below,yshift =-1mm] at (B) {$y$};
            \draw[dashed] (A) to (B);
            }
            \YellowRobot{(A)}{y}
            \YellowRobot{($(B)+(0,0.95)$)}{}
            \RedRobot{(B)}{}
            \GreenRobot{($(B)+(0,0.3)$)}{}
            \BlueRobot{($(B)+(0,0.6)$)}{b}
            \draw[bend right, ->] (b) to ($(y)+(0,0.3)$);
            \node at (0.75,-.65) {$c'$};
        \end{tikzpicture}
    }
    \savebox{\cellcpone}{
    \begin{tikzpicture}
            \fg{
            \node at (0,1.0) {};
            \vertex{(0,0)}{A}
            \vertex{(1.5,0)}{B}
            \node[below,yshift =-1mm] at (A) {$x$};
            \node[below,yshift =-1mm] at (B) {$y$};
            \draw[dashed] (A) to (B);
            }
            \YellowRobot{($(A)+(0,0.35)$)}{}
            \YellowRobot{($(B)+(0,0.65)$)}{}
            \BlueRobot{(A)}{}
            \GreenRobot{($(B)+(0,0.3)$)}{}
            \RedRobot{($(B)+(0,0)$)}{}
            \node at (0.75,-.65) {$c'_1$};
        \end{tikzpicture}
    }

    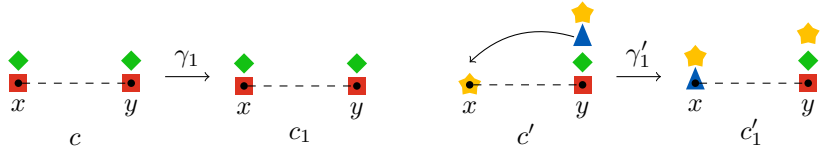
\begin{figure}[h!]
        \centering
        \begin{tikzpicture}
            \node (cellc) at (0,0) {\usebox{\cellc}};
            \node (cellcone) at (3,0) {\usebox{\cellcone}};
            \draw[->] (cellc) to node[above]{$\gamma_1$} (cellcone);
            \node (cellcp) at (6,0) {\usebox{\cellcp}};
            \node (cellcpone) at (9,0) {\usebox{\cellcpone}};
            \draw[->] (cellcp) to node[above]{$\gamma'_1$} (cellcpone);
        \end{tikzpicture}
        \caption{$d_x(c,c')=1;\ d_y(c,c')=-2;\ d_x(c_1,c'_1)=0;\ d_y(c_1,c'_1)=-1$}
        \label{fig:excess2}
    \end{figure}
    In either case, we have $0\le d_x(c_1,c_1')<d_x(c,c')$, $0\ge d_y(c_1,c_1')>d_y(c,c')$, and $d_v(c_1,c_1')=d_v(c,c')$ for all vertices $v\notin\{x,y\}$. We can continue this process to get sequences of 0-cells $c=c_0,c_1,\dots,c_p$ and $c'=c'_0,c'_1,\dots,c'_p$ and paths $\gamma_i$ from $c_{i-1}$ to $c_i$ and $\gamma'_i$ from $c'_{i-1}$ to $c'_i$, where $d_v(c_p,c'_p)=0$ for all vertices $v$. In particular, $c_p$ and $c_p'$ are of the same type, so by \cref{cor:SameType} there is a path $\tilde\gamma$ from $c_p$ to $c_p'.$ The sequence of paths $\gamma_1,\gamma_2,\dots,\gamma_p$, followed by $\tilde \gamma,$ followed by the sequence of paths $\gamma'_1,\dots,\gamma'_p$ in reverse gives a path $\gamma$ from $c$ to $c'.$ See \cref{fig:pathconnected}.

    \savebox{\cellc}{
    \begin{tikzpicture}
            \Xgraph
            \node[below,yshift=-1mm] at (A) {$x$};
            \node[below,yshift=-1mm] at (C) {$y$};
            \RedRobot{(A)}{}
            \RedRobot{(B)}{r}
            \RedRobot{(C)}{}
            \GreenRobot{($(A)+(0.3,0)$)}{}
            \GreenRobot{($(C)+(0,0.3)$)}{g}
            \BlueRobot{($(A)+(0.6,0)$)}{b1}
            \BlueRobot{($(B)+(0.25,0.25)$)}{b2}
            \YellowRobot{(D)}{}
            \YellowRobot{(E)}{}
            \draw[bend right, ->] (b1) to (b2);
            \draw[bend left, ->] (b2) to (g);
            \node at (0,-2) {$c=c_0$};
        \end{tikzpicture}
    }
    \savebox{\cellcone}{
    \begin{tikzpicture}
            \Xgraph
            \node[below,yshift=-1mm] at (A) {$x$};
            \node[below,yshift=-1mm] at (C) {$y$};
            \RedRobot{(A)}{}
            \RedRobot{(B)}{}
            \RedRobot{(C)}{}
            \GreenRobot{($(A)+(0.3,0)$)}{}
            \GreenRobot{($(C)+(0,0.3)$)}{}
            \BlueRobot{($(C)+(0,0.6)$)}{}
            \BlueRobot{($(B)+(0.25,0.25)$)}{}
            \YellowRobot{(D)}{}
            \YellowRobot{(E)}{}
            \node at (0,-2) {$c_1$};
        \end{tikzpicture}
    }
    \newsavebox{\cellctwo}
    \savebox{\cellctwo}{
    \begin{tikzpicture}
            \Xgraph
            \node[below,yshift=-1mm] at (A) {$x$};
            \node[below,yshift=-1mm] at (C) {$y$};
            \RedRobot{(A)}{}
            \RedRobot{(B)}{}
            \RedRobot{(C)}{}
            \GreenRobot{($(A)+(0.3,0)$)}{}
            \GreenRobot{($(C)+(0,0.3)$)}{}
            \BlueRobot{($(C)+(0,0.6)$)}{}
            \BlueRobot{($(B)+(0.25,0.25)$)}{}
            \YellowRobot{(D)}{}
            \YellowRobot{(E)}{}
            \node at (0,-2) {$c_2$};
        \end{tikzpicture}
    }
    \savebox{\cellcp}{
    \begin{tikzpicture}
            \Xgraph
            \node[below,yshift=-1mm] at (A) {$x$};
            \node[below,yshift=-1mm] at (C) {$y$};
            \RedRobot{(B)}{r1}
            \RedRobot{(C)}{r2}
            \RedRobot{(D)}{}
            \GreenRobot{($(C)+(0,0.3)$)}{}
            \GreenRobot{(E)}{}
            \BlueRobot{(A)}{b1}
            \BlueRobot{($(C)+(0,0.6)$)}{}
            \YellowRobot{($(B)+(0.25,0.25)$)}{}
            \YellowRobot{($(C)+(0,0.95)$)}{}
            \node at (0,-2) {$c'=c'_0$};
        \end{tikzpicture}
    }
    \savebox{\cellcpone}{
    \begin{tikzpicture}
            \Xgraph
            \node[below,yshift=-1mm] at (A) {$x$};
            \node[below,yshift=-1mm] at (C) {$y$};
            \RedRobot{(B)}{r1}
            \RedRobot{(C)}{r2}
            \RedRobot{(D)}{}
            \GreenRobot{($(C)+(0,0.3)$)}{}
            \GreenRobot{(E)}{}
            \BlueRobot{(A)}{b1}
            \BlueRobot{($(C)+(0,0.6)$)}{}
            \YellowRobot{($(B)+(0.25,0.25)$)}{}
            \YellowRobot{($(C)+(0,0.95)$)}{}
            \draw[->, bend left] (r2) to (r1);
            \draw[->, bend right] (r1) to (b1);
            \node at (0,-2) {$c_1'$};
        \end{tikzpicture}
    }
    \newsavebox{\cellcptwo}
    \savebox{\cellcptwo}{
    \begin{tikzpicture}
            \Xgraph
            \node[below,yshift=-1mm] at (A) {$x$};
            \node[below,yshift=-1mm] at (C) {$y$};
            \RedRobot{(B)}{}
            \RedRobot{(A)}{}
            \RedRobot{(D)}{}
            \GreenRobot{($(C)+(0,0)$)}{}
            \GreenRobot{(E)}{}
            \BlueRobot{($(A)+(0.3,0)$)}{}
            \BlueRobot{($(C)+(0,0.3)$)}{}
            \YellowRobot{($(B)+(0.25,0.25)$)}{}
            \YellowRobot{($(C)+(0,0.65)$)}{}
            \node at (0,-2) {$c'_2$};
        \end{tikzpicture}
    }
    \savebox{\ConfigA}{
    \begin{tikzpicture}
            \Xgraph
            \node[below,yshift=-1mm] at (A) {$x$};
            \node[below,yshift=-1mm] at (C) {$y$};
            \RedRobot{(A)}{}
            \RedRobot{(B)}{r}
            \RedRobot{(C)}{}
            \GreenRobot{($.65*(A)+.35*(B)$)}{}
            \GreenRobot{($(C)+(0,0.3)$)}{}
            \BlueRobot{($(A)+(0.3,0)$)}{}
            \BlueRobot{($0.5*(B)+0.5*(D)$)}{}
            \YellowRobot{(D)}{}
            \YellowRobot{(E)}{}
            \node at (0,-2) {$A$};
        \end{tikzpicture}
    }
    \savebox{\ConfigB}{
    \begin{tikzpicture}
            \Xgraph
            \node[below,yshift=-1mm] at (A) {$x$};
            \node[below,yshift=-1mm] at (C) {$y$};
            \RedRobot{(B)}{r1}
            \RedRobot{(C)}{r2}
            \RedRobot{(D)}{}
            \YellowRobot{($(C)+(0,0.3)$)}{}
            \GreenRobot{(E)}{}
            \BlueRobot{(A)}{b1}
            \BlueRobot{($0.65*(B)+0.35*(C)$)}{}
            \YellowRobot{($.5*(B)+.5*(D)$)}{}
            \GreenRobot{($.35*(B)+0.65*(C)$)}{}
            \node at (0,-2) {$B$};
        \end{tikzpicture}}
    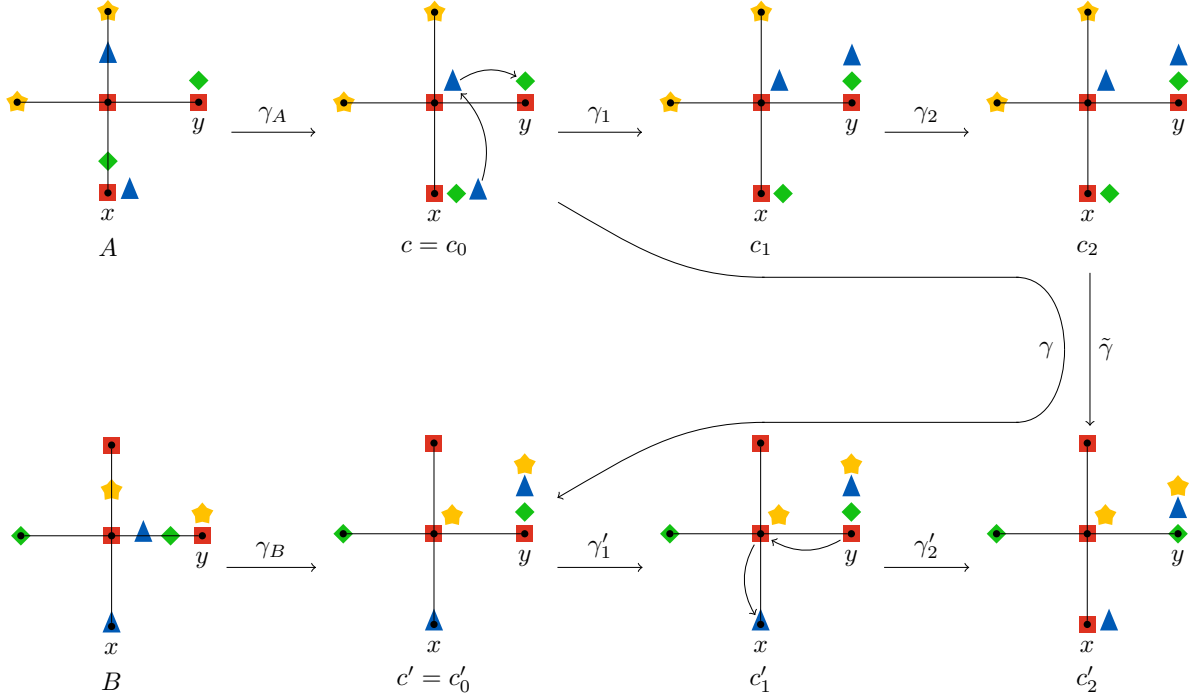
\begin{figure}[h!]
        \centering
        \scalebox{.965}{
        \begin{tikzpicture}
            \node (A) at (-4.5,0) {\usebox{\ConfigA}};
            \node (B) at (-4.5,-6) {\usebox{\ConfigB}};
            \node (c0) at (0,0) {\usebox{\cellc}};
            \node (c1) at (4.5,0) {\usebox{\cellcone}};
            \node (c2) at (9,0) {\usebox{\cellctwo}};
            \node (cp0) at (0,-6) {\usebox{\cellcp}};
            \node (cp1) at (4.5,-6) {\usebox{\cellcpone}};
            \node (cp2) at (9,-6) {\usebox{\cellcptwo}};
            \draw[->] (A) to node[above]{$\gamma_A$} (c0);
            \draw[->] (B) to node[above]{$\gamma_B$} (cp0);
            \draw[->] (c0) to node[above]{$\gamma_1$} (c1);
            \draw[->] (c1) to node[above]{$\gamma_2$} (c2);
            \draw[->] (cp0) to node[above]{$\gamma'_1$} (cp1);
            \draw[->] (cp1) to node[above]{$\gamma'_2$} (cp2);
            \draw[->] (c2) to node[right]{$\tilde \gamma$} (cp2);
            \node[inner sep = 0, outer sep = -1pt] (c1a) at (4.5,-2){};
            \node[inner sep = 0, outer sep = -1pt] (c2a) at (8,-2){};
            \node[inner sep = 0, outer sep = -1pt] (cp2a) at (8,-4){};
            \node[inner sep = 0, outer sep = -1pt] (cp1a) at (4.5,-4){};
            \draw [out = -30, in=180] (c0) to (c1a);
            \draw[out = 0, in=180] (c1a) to (c2a);
            \draw[out = 0, in=90] (c2a) to ($0.5*(c2a)+0.5*(cp2a)+(.65,0)$);
            \node at ($0.5*(c2a)+0.5*(cp2a)+(.4,0)$) {$\gamma$};
            \draw[out = -90, in=0] ($0.5*(c2a)+0.5*(cp2a)+(.65,0)$) to (cp2a);
            \draw[out = 0, in=180] (cp1a) to (cp2a);
            \draw [out = 30, in=180,<-] (cp0) to (cp1a);
        \end{tikzpicture}
        }
        \caption{A path from $A$ to $B$}
        \label{fig:pathconnected}
    \end{figure}

    In total, the path $\gamma_A$, followed by $\gamma$, followed by $\gamma_B$ in reverse gives a path from $A$ to $B$. Therefore, there is a path in $\Sr(G)$ between any two configurations $A$ and $B$, so $\Sr(G)$ is path-connected.
\end{proof}
\section{Cell Counts of $\Sr(G)$}\label{sec:CellCounts}
In this section, we determine the number of cells of $\Sr(G)$ for two classes of color vectors when $G$ is a simple graph (i.e. a graph with no loops or multiple edges). We first consider the case in which $\vec r=(2,1,\dots,1)\in \Z^n$. That is, we have a total of $n$ colors, with two 1-colored robots and one $k$-colored robot for $k\ge 2,$ so that we have a total of $n+1$ robots. 

\begin{thm}\label[thm]{thm:numCellsA}
    Let $G$ be a simple graph with $n\ge 2$ vertices and $m$ edges, and let $\vec{r}=(2,1,\dots,1)\in\Z^n$. In $\Sr(G)$ we have 
    \[
    \#(\text{i-cells)}=\begin{cases}
        \frac{n!(n^2+n-2)}{4}&\text{if }i=0\\
        \frac{m(n-1)!(n^2+n-4)}{2}&\text{if }i=1\\
        0&\text{if }i\ge 2\\
    \end{cases}.
    \]
\end{thm}

\begin{proof}
    In any 0-cell $c$ of $\Sr(G)$, there must be two distinct vertices, $x\ne y$, occupied by the two 1-colored robots; there are $\binom{n}{2}$ choices for these two vertices. Since there are a total of $n+1$ robots and $n$ vertices, there must be exactly one available vertex $v$, with $|c(v)|=2.$ So, for each of the choices of the vertices for the two 1-colored robots, we have the following two cases:
    
    First, suppose either $v=x$ or $v=y$. For each of these two possibilities, there must be some non-available vertex $u$ (either $y$ or $x$) with $1\in c(u),$ so the remaining $n-1$ robots can be permuted among the $n-1$ vertices in $V(G)\setminus\{u\}$. So, there are $2(n-1)!$ possibilities in this case.
    
    Second, if $v\ne x$ and $v\ne y,$ there must be two other colors $i\ne j$ with $i,j\in c(v).$ There are $\binom{n-1}{2}$ choices for these two colors, then $n-2$ choices for the vertex $v$, followed by $(n-3)!$ ways to permute the remaining robots among the remaining vertices in $V(G)\setminus\{x,y,v\}$. So, there are $\binom{n-1}{2}\cdot (n-2)\cdot(n-3)!=\binom{n-1}{2}\cdot (n-2)!$ possibilities in this case.

    Thus, in total, we have 
    \[
    \#(\text{0-cells)}=\binom{n}{2}\left(2(n-1)! +\binom{n-1}{2}(n-2)!\right)
    =\frac{n!(n^2+n-2)}{4}.
    \]

    For any $1$-cell $c$, there must be one edge $e$ with $|c(e)|=1,$ and $|c(v)|=1$ for each $v\in V(G).$ There are $m$ choices for the edge. If $c(e)=\{1\},$ there are $n-2$ choices for the location of the other 1-colored robot (since it can't fall on either endpoint of $e$), and $(n-1)!$ permutations of the remaining $n-1$ robots among the remaining $n-1$ vertices. If $c(e)=\{i\},$ with $i\ne 1,$ there are $n-1$ choices for $i$,  $\binom{n}{2}$ choices for the locations of the two $1$-colored robots, and $(n-2)!$ permutations of the remaining $n-2$ robots.
    
    Thus, in total, we have 
    \[
    \#(\text{1-cells})=m\left((n-2)(n-1)!+(n-1)\binom{n}{2}(n-2)!\right)=\frac{m(n-1)!(n^2+n-4)}{2}.
    \]

    The fact that there are no $i$-cells for $i\ge 2$ follows immediately from \cref{prop:MaxDimension}.
\end{proof}

Note that under the hypotheses of \cref{thm:numCellsA}, $\Sr(G)$ has only 0-cells and 1-cells, so that $\Sr(G)$ is itself a graph. For example, \cref{fig:SrP3} shows the configuration space $\Sr(P_3)$ for $\vec r = (2,1,1).$ 

\def\scalefactor{1}
    \savebox{\ConfigA}{
    \scalebox{\scalefactor}{\begin{tikzpicture}[scale=0.5]
        \vertex{(0,0)}{x}
        \vertex{(1,0)}{y}
        \vertex{(2,0)}{z}
        \RedRobot{(x)}{}
        \RedRobot{(z)}{}
        \GreenRobot{(y)}{}
        \BlueRobot{($(y)+(0,.5)$)}{}
        \smallvertex{(0,0)}{x}
        \smallvertex{(1,0)}{y}
        \smallvertex{(2,0)}{z}     
        \draw (0,0)--(2,0);
    \end{tikzpicture}
    }}
    \savebox{\ConfigB}{
    \scalebox{\scalefactor}{\begin{tikzpicture}[scale=0.5]
        \vertex{(0,0)}{x}
        \vertex{(1,0)}{y}
        \vertex{(2,0)}{z}
        \RedRobot{(x)}{}
        \RedRobot{(z)}{}
        \GreenRobot{(y)}{}
        \BlueRobot{($(x)+(0,.5)$)}{}
        \smallvertex{(0,0)}{x}
        \smallvertex{(1,0)}{y}
        \smallvertex{(2,0)}{z}
        \draw (0,0)--(2,0);
    \end{tikzpicture}
    }}
    \savebox{\ConfigC}{
    \scalebox{\scalefactor}{\begin{tikzpicture}[scale=0.5]
        \vertex{(0,0)}{x}
        \vertex{(1,0)}{y}
        \vertex{(2,0)}{z}
        \RedRobot{(x)}{}
        \RedRobot{(z)}{}
        \GreenRobot{(y)}{}
        \BlueRobot{($(z)+(0,.5)$)}{}
        \smallvertex{(0,0)}{x}
        \smallvertex{(1,0)}{y}
        \smallvertex{(2,0)}{z}
        \draw (0,0)--(2,0);
    \end{tikzpicture}
    }}
    \savebox{\ConfigD}{
    \scalebox{\scalefactor}{\begin{tikzpicture}[scale=0.5]
        \vertex{(0,0)}{x}
        \vertex{(1,0)}{y}
        \vertex{(2,0)}{z}
        \RedRobot{(x)}{}
        \RedRobot{(z)}{}
        \BlueRobot{(y)}{}
        \GreenRobot{($(x)+(0,.6)$)}{}
        \smallvertex{(0,0)}{x}
        \smallvertex{(1,0)}{y}
        \smallvertex{(2,0)}{z}
        \draw (0,0)--(2,0);
    \end{tikzpicture}
    }}
    \savebox{\ConfigE}{
    \scalebox{\scalefactor}{\begin{tikzpicture}[scale=0.5]
        \vertex{(0,0)}{x}
        \vertex{(1,0)}{y}
        \vertex{(2,0)}{z}
        \RedRobot{(x)}{}
        \RedRobot{(z)}{}
        \BlueRobot{(y)}{}
        \GreenRobot{($(z)+(0,.6)$)}{}
        \smallvertex{(0,0)}{x}
        \smallvertex{(1,0)}{y}
        \smallvertex{(2,0)}{z}
        \draw (0,0)--(2,0);
    \end{tikzpicture}
    }}
    \savebox{\ConfigF}{
    \scalebox{\scalefactor}{\begin{tikzpicture}[scale=0.5]
        \vertex{(0,0)}{x}
        \vertex{(1,0)}{y}
        \vertex{(2,0)}{z}
        \RedRobot{(y)}{}
        \RedRobot{(z)}{}
        \GreenRobot{($(y)+(0,0.65)$)}{}
        \BlueRobot{($(x)+(0,0)$)}{}
        \smallvertex{(0,0)}{x}
        \smallvertex{(1,0)}{y}
        \smallvertex{(2,0)}{z}
        \draw (0,0)--(2,0);
    \end{tikzpicture}
    }}
    \savebox{\ConfigG}{
    \scalebox{\scalefactor}{\begin{tikzpicture}[scale=0.5]
        \vertex{(0,0)}{x}
        \vertex{(1,0)}{y}
        \vertex{(2,0)}{z}
        \RedRobot{($(y)+(0,0)$)}{}
        \RedRobot{(z)}{}
        \GreenRobot{($(z)+(0,0.65)$)}{}
        \BlueRobot{($(x)+(0,0)$)}{}
        \smallvertex{(0,0)}{x}
        \smallvertex{(1,0)}{y}
        \smallvertex{(2,0)}{z}
        \draw (0,0)--(2,0);
    \end{tikzpicture}
    }}
    \savebox{\ConfigH}{
    \scalebox{\scalefactor}{\begin{tikzpicture}[scale=0.5]
        \vertex{(0,0)}{x}
        \vertex{(1,0)}{y}
        \vertex{(2,0)}{z}
        \RedRobot{($(y)+(0,0)$)}{}
        \RedRobot{(z)}{}
        \GreenRobot{($(x)+(0,0.75)$)}{}
        \BlueRobot{($(x)+(0,0)$)}{}
        \smallvertex{(0,0)}{x}
        \smallvertex{(1,0)}{y}
        \smallvertex{(2,0)}{z}
        \draw (0,0)--(2,0);
    \end{tikzpicture}
    }}
    \savebox{\ConfigI}{
    \scalebox{\scalefactor}{\begin{tikzpicture}[scale=0.5]
        \vertex{(0,0)}{x}
        \vertex{(1,0)}{y}
        \vertex{(2,0)}{z}
        \RedRobot{($(y)+(0,0)$)}{}
        \RedRobot{(z)}{}
        \GreenRobot{($(x)+(0,0)$)}{}
        \BlueRobot{($(y)+(0,0.5)$)}{}
        \smallvertex{(0,0)}{x}
        \smallvertex{(1,0)}{y}
        \smallvertex{(2,0)}{z}
        \draw (0,0)--(2,0);
    \end{tikzpicture}
    }}
    \savebox{\ConfigJ}{
    \scalebox{\scalefactor}{\begin{tikzpicture}[scale=0.5]
        \vertex{(0,0)}{x}
        \vertex{(1,0)}{y}
        \vertex{(2,0)}{z}
        \RedRobot{($(y)+(0,0)$)}{}
        \RedRobot{(z)}{}
        \GreenRobot{($(x)+(0,0)$)}{}
        \BlueRobot{($(z)+(0,0.5)$)}{}
        \smallvertex{(0,0)}{x}
        \smallvertex{(1,0)}{y}
        \smallvertex{(2,0)}{z}
        \draw (0,0)--(2,0);
    \end{tikzpicture}
    }}
    \savebox{\ConfigK}{
    \scalebox{\scalefactor}{\begin{tikzpicture}[scale=0.5]
        \vertex{(0,0)}{x}
        \vertex{(1,0)}{y}
        \vertex{(2,0)}{z}
        \RedRobot{($(y)+(0,0)$)}{}
        \RedRobot{(x)}{}
        \GreenRobot{($(y)+(0,0.6)$)}{}
        \BlueRobot{($(z)+(0,0)$)}{}
        \smallvertex{(0,0)}{x}
        \smallvertex{(1,0)}{y}
        \smallvertex{(2,0)}{z}
        \draw (0,0)--(2,0);
    \end{tikzpicture}
    }}
    \savebox{\ConfigL}{
    \scalebox{\scalefactor}{\begin{tikzpicture}[scale=0.5]
        \vertex{(0,0)}{x}
        \vertex{(1,0)}{y}
        \vertex{(2,0)}{z}
        \RedRobot{($(y)+(0,0)$)}{}
        \RedRobot{(x)}{}
        \GreenRobot{($(x)+(0,0.6)$)}{}
        \BlueRobot{($(z)+(0,0)$)}{}
        \smallvertex{(0,0)}{x}
        \smallvertex{(1,0)}{y}
        \smallvertex{(2,0)}{z}
        \draw (0,0)--(2,0);
    \end{tikzpicture}
    }}
    \savebox{\ConfigM}{
    \scalebox{\scalefactor}{\begin{tikzpicture}[scale=0.5]
        \vertex{(0,0)}{x}
        \vertex{(1,0)}{y}
        \vertex{(2,0)}{z}
        \RedRobot{($(y)+(0,0)$)}{}
        \RedRobot{(x)}{}
        \GreenRobot{($(z)+(0,0.75)$)}{}
        \BlueRobot{($(z)+(0,0)$)}{}
        \smallvertex{(0,0)}{x}
        \smallvertex{(1,0)}{y}
        \smallvertex{(2,0)}{z}
        \draw (0,0)--(2,0);
    \end{tikzpicture}
    }}
    \savebox{\ConfigN}{
    \scalebox{\scalefactor}{\begin{tikzpicture}[scale=0.5]
        \vertex{(0,0)}{x}
        \vertex{(1,0)}{y}
        \vertex{(2,0)}{z}
        \RedRobot{($(y)+(0,0)$)}{}
        \RedRobot{(x)}{}
        \GreenRobot{($(z)+(0,0)$)}{}
        \BlueRobot{($(y)+(0,0.5)$)}{}
        \smallvertex{(0,0)}{x}
        \smallvertex{(1,0)}{y}
        \smallvertex{(2,0)}{z}
        \draw (0,0)--(2,0);
    \end{tikzpicture}
    }}
    \savebox{\ConfigO}{
    \scalebox{\scalefactor}{\begin{tikzpicture}[scale=.5]
        \vertex{(0,0)}{x}
        \vertex{(1,0)}{y}
        \vertex{(2,0)}{z}
        \RedRobot{($(y)+(0,0)$)}{}
        \RedRobot{(x)}{}
        \GreenRobot{($(z)+(0,0)$)}{}
        \BlueRobot{($(x)+(0,0.5)$)}{}
        \smallvertex{(0,0)}{}
        \smallvertex{(1,0)}{}
        \smallvertex{(2,0)}{}
        \draw (0,0)--(2,0);
    \end{tikzpicture}
    }}

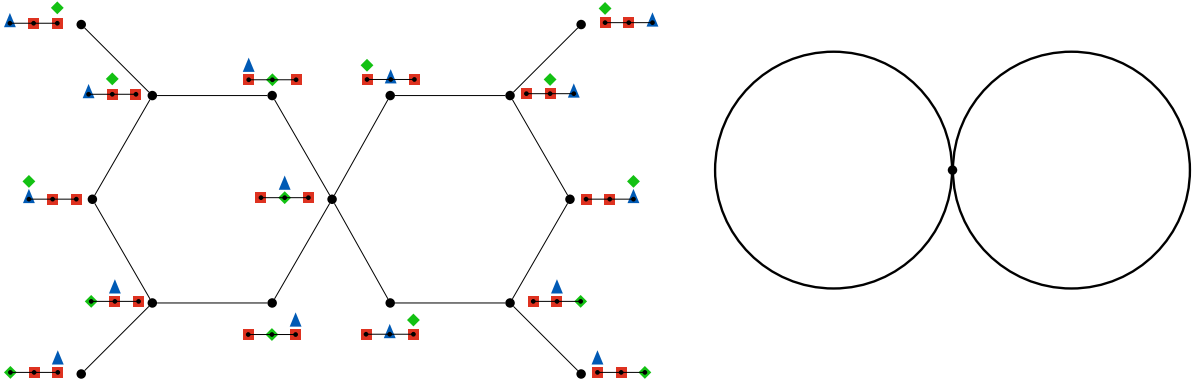
\begin{figure}[h!]
\centering
\scalebox{.63}{
\begin{tikzpicture}
        \def\vertsize{2}
        \node[regular polygon, regular polygon sides = 6, minimum size = 5cm, line width = 0.5mm] (A) at (-2.51,0) {};
        \node[regular polygon, regular polygon sides = 6, minimum size = 5cm, line width = 0.5mm] (B) at (2.51,0) {};
        \vertex[\vertsize]{(A.corner 6)}{a6};
        \node[yshift=2mm,xshift=-1cm] at (a6) {\usebox\ConfigA};
        \vertex[\vertsize]{(A.corner 1)}{a1};
        \node[yshift=.5cm] at (a1) {\usebox\ConfigB};
        \draw(a6)--(a1);
        \vertex[\vertsize]{(A.corner 5)}{a5};
        \node[yshift=-.5cm] at (a5) {\usebox\ConfigC};
        \draw(a6)--(a5);
        \vertex[\vertsize]{(B.corner 2)}{b2};
        \node[yshift=.5cm] at (b2) {\usebox\ConfigD};
        \draw(a6)--(b2);
        \vertex[\vertsize]{(B.corner 4)}{b4};
        \node[yshift=-.5cm] at (b4) {\usebox\ConfigE};
        \draw(a6)--(b4);
        \vertex[\vertsize]{(A.corner 2)}{a2};
        \node[yshift=2mm,xshift=-.85cm] at (a2) {\usebox\ConfigF};
        \draw(a1)--(a2);
        \vertex[\vertsize]{($(A.corner 2)+(-1.5,1.5)$)}{c1};
        \node[yshift=2mm,xshift=-1cm] at (c1) {\usebox\ConfigG}; 
        \draw(a2)--(c1);
        \vertex[\vertsize]{(A.corner 3)}{a3};
        \node[yshift=2mm,xshift=-.85cm] at (a3) {\usebox\ConfigH};
        \draw(a2)--(a3);
        \vertex[\vertsize]{(A.corner 4)}{a4};
        \node[yshift=2mm,xshift=-.8cm] at (a4) {\usebox\ConfigI};
        \draw(a3)--(a4);
        \vertex[\vertsize]{($(A.corner 4)+(-1.5,-1.5)$)}{c2};
        \node[yshift=2mm,xshift=-1cm] at (c2) {\usebox\ConfigJ}; 
        \draw(a4)--(c2);
        \draw (a4) -- (a5);
        \vertex[\vertsize]{(B.corner 1)}{b1};
        \node[yshift=2mm,xshift=.85cm] at (b1) {\usebox\ConfigK};
        \draw (b2)--(b1);
        \vertex[\vertsize]{($(B.corner 1)+(1.5,1.5)$)}{c3};
        \node[yshift=2mm,xshift=1cm] at (c3) {\usebox\ConfigL}; 
        \draw(b1)--(c3);
        \vertex[\vertsize]{(B.corner 6)}{b6}
        \node[yshift=2mm,xshift=.85cm] at (b6) {\usebox\ConfigM}; 
        \draw(b1)--(b6);
        \vertex[\vertsize]{(B.corner 5)}{b5}
        \node[yshift=2mm,xshift=1cm] at (b5) {\usebox\ConfigN}; 
        \draw(b5)--(b6);
        \draw (b4)--(b5);
        \vertex[\vertsize]{($(B.corner 5)+(1.5,-1.5)$)}{c4}
        \node[yshift=2mm,xshift=.85cm] at (c4) {\usebox\ConfigO};
        \draw (b5)--(c4);
    \end{tikzpicture}
    \quaad\raisebox{2cm}{
    \begin{tikzpicture}
        \node[circle, draw, minimum size = 5cm, line width = 0.5mm] (A) at (-2.51,0) {};
        \node[circle, draw, minimum size = 5cm, line width = 0.5mm] (B) at (2.51,0) {};
        \vertex[2]{(0,0)}{};
    \end{tikzpicture}}
    }
    \caption{$\Sr(P_3),$ where $\vec r = (2,1,1)$ (left); A wedge of two circles (right)}
    \label{fig:SrP3}
    \end{figure}

    The main feature of $\Sr(G)$ in this case is that it has two cycles (the two hexagons), so that $\Sr(G)$ is in some sense similar to the two circles joined at a point on the right side of \cref{fig:SrP3}; we refer to this as a \textit{wedge} of two circles. Topologically speaking, we say $\Sr(G)$ is \textit{homotopy equivalent} to a wedge of two circles. More information regarding homotopy equivalence can be found in standard topology textbooks, such as \cite{Munkres_2000}. One can show that any connected graph with $N$ vertices and $M$ edges is homotopy equivalent to a wedge of $M-N+1$ circles. For $\Sr(P_3)$ with $\vec r=(2,1,1),$ we have $N=15$ and $M=16$ (i.e. there are 15 0-cells and 16 1-cells of $\Sr(G)$), so that $\Sr(G)$ is homotopy equivalent to a wedge of $16-15+1=2$ circles.

    In general, under the hypotheses of \cref{thm:numCellsA}, we can classify $\Sr(G)$ up to homotopy.
    \begin{cor}\label[cor]{cor:HomotopyType}
        Let $G$ be a connected simple graph with $n\ge 2$ vertices and $m$ edges, and let $\vec{r}=(2,1,\dots,1)\in\Z^n$. The space $\Sr(G)$ is homotopy equivalent to a wedge of $L$ circles, where 
        \[
        L=\frac{m(n-1)!(n^2+n-4)}{2}-\frac{n!(n^2+n-2)}{4}+1.
        \]
    \end{cor}

    \begin{proof}
        For $n=2,$ the two 1-colored robots must occupy the two vertices of $G$, so that the one 2-colored robot can move freely throughout $G$. Therefore $\Sr(G)$ is homeomorphic to $G$ and thus $\Sr(G)$ is connected. For $n\ge 3,$ $\Sr(G)$ is connected by \cref{thm:SrGConnected}. The result then follows immediately from the counts of the 0- and 1-cells given in \cref{thm:numCellsA}.
    \end{proof}

As an example, \cref{tab:CellTableA} shows the number of 0- and 1-cells of $\Sr(P_4)$, $\Sr(T_4)$, $\Sr(C_4)$, $\Sr(K_4)$ (for $\vec r = (2,1,1,1)$) and $\Sr(K_5)$ (for $\vec r = (2,1,1,1,1)$), and the number, $L$, of circles in the wedge of circles to which each space is homotopy equivalent. Here, $P_n,$ $T_n,$ $C_n,$ and $K_n$ denote the path, star, cycle, and complete graphs with $n$ vertices.
 
\newsavebox{\GraphAa}
\savebox{\GraphAa}{
\begin{tikzpicture}[scale=0.35]
    \vertex{(0,0)}{a}
    \vertex{(1,0)}{b}
    \vertex{(2,0)}{c}
    \vertex{(3,0)}{d}
    \draw (a)--(d);
\end{tikzpicture}
}

\newsavebox{\GraphA}
\savebox{\GraphA}{
\begin{tikzpicture}[scale=0.5]
    \Ygraph;
\end{tikzpicture}
}

\newsavebox{\GraphB}
\savebox{\GraphB}{
\begin{tikzpicture}[scale=0.7]
    \vertex{(0,0)}{a}
    \vertex{(0,1)}{b}
    \vertex{(1,1)}{c}
    \vertex{(1,0)}{d}
    \draw (a)--(b)--(c)--(d)--(a);
\end{tikzpicture}
}

\newsavebox{\GraphC}
\savebox{\GraphC}{
\begin{tikzpicture}[scale=0.7]
    \vertex{(0,0)}{a}
    \vertex{(0,1)}{b}
    \vertex{(1,1)}{c}
    \vertex{(1,0)}{d}
    \draw (a)--(b)--(c)--(d)--(a)--(c);
    \draw (b)--(d);
\end{tikzpicture}
}

\newsavebox{\GraphD}
\savebox{\GraphD}{
\begin{tikzpicture}[scale=0.5]
    \node[regular polygon, regular polygon sides=5, minimum size = 1cm] (A) at (0,0) {};
    \vertex{(A.corner 1)}{a};
    \vertex{(A.corner 2)}{b};
    \vertex{(A.corner 3)}{c};
    \vertex{(A.corner 4)}{d};
    \vertex{(A.corner 5)}{e};
    \draw (a)--(b)--(c)--(d)--(e)--(a)--(c)--(e)--(b)--(d)--(a);
\end{tikzpicture}
}

\begin{table}[h!]
    \centering
    $\begin{array}{c|c|c|c|c|c}
         &\raisebox{4mm}{\usebox\GraphAa}&   \usebox\GraphA & \usebox\GraphB & \usebox\GraphC & \usebox\GraphD\\
         \hline
         \vec r&(2,1,1,1)&(2,1,1,1) & (2,1,1,1) & (2,1,1,1) & (2,1,1,1,1)\\
         \hline
         \text{0-cells}&108&108 & 108 & 108& 840\\
         \hline
         \text{1-cells} &144& 144 & 192 & 288 & 3120\\
         \hline{L}&37 & 37 & 85 & 181 & 2281
         \text{}
    \end{array}$
    \caption{Cell counts for $\Sr(P_4)$, $\Sr(T_4)$, $\Sr(C_4)$, $\Sr(K_4)$ and $\Sr(K_5)$}
    \label{tab:CellTableA}
\end{table}

For the second class of color vectors, consider the case in which $\vec r=(n-1,n-1,\dots,n-1)\in \Z^r$ for any $r\ge 2.$ That is, we have $r$ colors and $n-1$ robots of each color. In \cref{thm:numCellsB}, we will determine the number of $i$-cells of $\Sr(G)$ for this class of color vectors. We first make the following observation. 

Consider a cell $c=(c_1,c_2,\dots,c_r)$ of $\Sr(G),$ so that each $c_k$ is a collection of $n-1$ vertices and/or edges. If $c_k$ consists exclusively of vertices, then $c_k$ is an $(n-1)$-element subset of $V(G)$, so $c_k=V(G)\backslash\{v\}$ for some vertex $v.$ In this case, we say $k$ is a \textit{vertex-only color}, and $v$ is the \textit{$k$-missed vertex}. If there is some edge $e=(a,b)\in c_k$,  then since $|c_k|=n-1,$ condition (c) of \cref{defn:cell} forces $c_k=\{e\}\cup(V(G)\backslash\{a,b\})$. That is, $c_k$ consists of $e$ together with all of the vertices of $G$ other than the endpoints of $e.$ In this case, we say $k$ is an \textit{edge color,} and $e$ is \textit{the $k$-colored edge}.

Therefore, for each $i$-cell $c$ of $\Sr(G),$ we can construct an $r$-tuple 
\[
\tilde c=(x_1,x_2,\dots,x_r),
\]
where $x_k$ is the $k$-missed vertex if $k$ is a vertex-only color, and $x_k$ is the $k$-colored edge if $k$ is an edge color. Since there are $i$ edges in $c$, there are also $i$ edges in $\tilde c.$ As such, we refer to $\tilde c$ as an \textit{$i$-edge-$r$-tuple}.

For example, the left side of \cref{fig:2edge4tuple} shows a cell $c=(\{x,y,e\},\{u,x,g\},\{u,x,y\},\{w,x,y\})$ of $\Sr(Y')$, where $\vec r = (3,3,3,3).$ Colors 1 and 2 (red/square and green/diamond) are edge colors. The 1-colored edge is $e$ and the 2-colored edge is $g.$ Colors 3 and 4 (blue/triangle and yellow/star) are vertex-only colors. The 3-missed vertex is $w$ and the 4-missed vertex is $u.$ Therefore, $\tilde c=(e,g,w,u)$. Similarly, for the cell in the middle of \cref{fig:2edge4tuple}, we have $\tilde c = (e, h, w, w).$

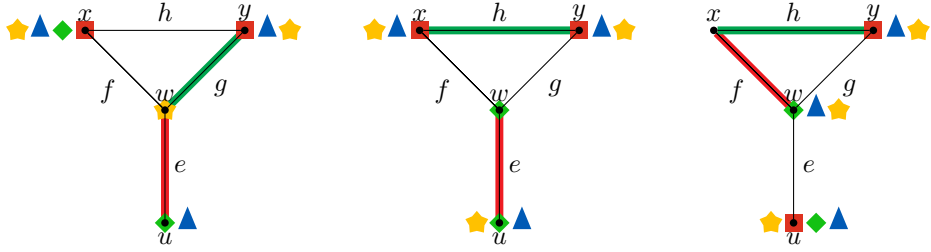
\begin{figure}[h!]
    \centering
    \begin{tikzpicture}[scale=1.5]
        \Ygraph;
        \fg{\draw (C)--(D);}
        \draw[line width=3, Red] (A) to (B);
        \draw[line width=3, Green] (B) to (D);
        \RedRobot{(D)}{}
        \RedRobot{(C)}{}
        \GreenRobot{(A)}{}
        \GreenRobot{($(C)+(-.2,0)$)}{}
        \BlueRobot{($(A)+(.2,0)$)}{}
        \BlueRobot{($(D)+(.2,0)$)}{}
        \BlueRobot{($(C)+(-.4,0)$)}{}
        \YellowRobot{(B)}{}
        \YellowRobot{($(D)+(.4,0)$)}{}
        \YellowRobot{($(C)+(-.6,0)$)}{}
        \node[below] at (A) {$u$};
        \node[above] at (B) {$w$};
        \node[above] at (C) {$x$};
        \node[above] at (D) {$y$};
        \node[above] at ($.5*(C)+.5*(D)$) {$h$};
        \node[right] at ($.5*(A)+.5*(B)$) {$e$};
        \node[below right] at ($.5*(D)+.5*(B)$) {$g$};
        \node[below left] at ($.5*(B)+.5*(C)$) {$f$};
    \end{tikzpicture}
    \quaad
    \begin{tikzpicture}[scale=1.5]
        \Ygraph
        \fg{\draw (C)--(D);}
        \draw[line width=3, Red] (A) to (B);
        \draw[line width=3, Green] (C) to (D);
        \RedRobot{(D)}{}
        \RedRobot{(C)}{}
        \GreenRobot{(A)}{}
        \GreenRobot{(B)}{}
        \BlueRobot{($(A)+(.2,0)$)}{}
        \BlueRobot{($(D)+(.2,0)$)}{}
        \BlueRobot{($(C)+(-.2,0)$)}{}
        \YellowRobot{($(A)+(-.2,0)$)}{}
        \YellowRobot{($(D)+(.4,0)$)}{}
        \YellowRobot{($(C)+(-.4,0)$)}{}
        \node[below] at (A) {$u$};
        \node[above] at (B) {$w$};
        \node[above] at (C) {$x$};
        \node[above] at (D) {$y$};
        \node[above] at ($.5*(C)+.5*(D)$) {$h$};
        \node[right] at ($.5*(A)+.5*(B)$) {$e$};
        \node[below right] at ($.5*(D)+.5*(B)$) {$g$};
        \node[below left] at ($.5*(B)+.5*(C)$) {$f$};
    \end{tikzpicture}
    \quaad
    \begin{tikzpicture}[scale=1.5]
        \Ygraph
        \fg{\draw (C)--(D);}
        \draw[line width=3, Red] (B) to (C);
        \draw[line width=3, Green] (C) to (D);
        \RedRobot{(D)}{}
        \RedRobot{(A)}{}
        \GreenRobot{($(A)+(.2,0)$)}{}
        \GreenRobot{(B)}{}
        \BlueRobot{($(A)+(.4,0)$)}{}
        \BlueRobot{($(D)+(.2,0)$)}{}
        \BlueRobot{($(B)+(.2,0)$)}{}
        \YellowRobot{($(A)+(-.2,0)$)}{}
        \YellowRobot{($(D)+(.4,0)$)}{}
        \YellowRobot{($(B)+(.4,0)$)}{}
        \node[below] at (A) {$u$};
        \node[above] at (B) {$w$};
        \node[above] at (C) {$x$};
        \node[above] at (D) {$y$};
        \node[above] at ($.5*(C)+.5*(D)$) {$h$};
        \node[right] at ($.5*(A)+.5*(B)$) {$e$};
        \node[below right] at ($.5*(D)+.5*(B)$) {$g$};
        \node[below left] at ($.5*(B)+.5*(C)$) {$f$};
    \end{tikzpicture}
    \caption{Valid cells of $\Sr(Y')$ (left and middle); Not a valid cell of $\Sr(Y')$ (right)}
    \label{fig:2edge4tuple}
\end{figure}

We can define a function $\Phi\from\{\text{$i$-cells of $\Sr(G)$}\}\to\{\text{$i$-edge-$r$-tuples}\}$ by $\Phi(c)=\tilde c.$ It is clear that $\Phi$ sends each $i$-cell $c$ to a unique $i$-edge-$r$-tuple, so that $\Phi$ is injective. Thus, if we define an $i$-edge-$r$-tuple to be \textit{valid} if it is in the image of $\Phi$ (and \textit{invalid} otherwise) we have a bijection 

\begin{equation}\label{eqn:bijection}
    \Phi\from\{i\text{-cells of }\Sr(G)\}\leftrightarrow \{\text{valid $i$-edge-$r$-tuples}\}.
\end{equation}

Any $i$-edge-$r$-tuple (valid or invalid) is determined by first choosing the $i$ components (i.e. colors) in the $r$-tuple which are edges; for each of these $i$ components, there are $m$ choices for the edge in that component. For the remaining $r-i$ components (which are vertices), there are $n$ choices for the vertex in that component. Thus, the total number of $i$-edge-$r$-tuples is 

\begin{equation}\label{eqn:total}
    \binom{r}{i} m^in^{r-i}.
\end{equation}

In the proof of \cref{thm:numCellsB}, we count the number of valid $i$-edge-$r$-tuples (i.e. number of $i$-cells of $\Sr(G)$) by subtracting the number of invalid $i$-edge-$r$-tuples from (\ref{eqn:total}).

\begin{ex}
As an example to motivate the proof, consider the invalid 2-edge-4-tuples for $\Sr(Y')$ (again with $\vec r = (3,3,3,3)$). Such a tuple $\tilde c$ must consist of 2 edges and 2 vertices such that: (a) each vertex is equal to some common vertex $v$, and (b), each edge must have $v$ as an endpoint. For example, the 2-edge-2-tuple $\tilde c = (f,h,x,x)$ would correspond to the ``cell" on the right of \cref{fig:2edge4tuple}, but this is not a valid cell since no robot is on the vertex $x.$ In this case, the common vertex of $\tilde c$ is $x$, and the two edges $f$ and $h$ have $x$ as an endpoint. Note both conditions (a) and (b) are necessary for $\tilde c$ to be invalid. For example, $\tilde c = (e,h,w,w)$ satisfies condition (a) but not (b), and as we saw above, $\tilde c$ is valid in this case.

So, to describe an invalid 2-edge-4-tuple $\tilde c$ for $\Sr(Y')$, we can first pick the common vertex $v$, then pick two of the four components of $\tilde c$ in which an edge appears, then for each of those $2$ components, pick one of the $\deg(v)$ edges which has $v$ as an endpoint. Each of the remaining $4-2$ components must equal $v.$ Thus, the total number of invalid 2-edge-2-tuples in this case is 
\[
\overbrace{\binom{4}{2}\deg(u)^2}^{v=u}+\overbrace{\binom{4}{2}\deg(w)^2}^{v=w}+\overbrace{\binom{4}{2}\deg(x)^2}^{v=x}+\overbrace{\binom{4}{2}\deg(y)^2}^{v=y}=6\cdot1+6\cdot9+6\cdot4+6\cdot4=108.
\]
Subtracting this from (\ref{eqn:total}), with $r=4,\ i=2, m=n=4$ gives the number of 2-cells of $\Sr(Y')$:
\[
\#(\text{2-cells})=\binom{4}{2}\cdot4^2\cdot4^2-108=1428.
\]
\end{ex}

\begin{thm}\label[thm]{thm:numCellsB}
    Let $G$ be a simple graph with $n$ vertices and $m$ edges, and let $\vec{r}=(n-1,n-1,\dots,n-1)\in\Z^r$ for some integer $r\ge 2$. In $\Sr(G)$ we have 
    \[
    \#(\text{i-cells)}=\begin{cases}\displaystyle
        \binom{r}{i}\left(m^in^{r-i}-\displaystyle\sum_{v\in V(G)}\deg(v)^i\right)&\text{if }0\le i<r\\
        m^r+m-\displaystyle\sum_{v\in V(G)}\deg(v)^r&\text{if }i=r\\
        0&\text{if }i> r\\
    \end{cases}.
    \]
\end{thm}

\begin{proof}
As noted above, to count the number of $i$-cells of $\Sr(G),$ it suffices to count the number of valid $i$-edge-$r$-tuples by subtracting the number of invalid $i$-edge-$r$-tuples from (\ref{eqn:total}).

If $\tilde c$ is an invalid $i$-edge-$r$-tuple, then any vertex in $\tilde c$ must be a $k$-missed vertex for \textit{every} vertex-only color $k$. This means that each vertex in $\tilde c$ is equal to some fixed vertex $v.$ Furthermore, every edge in $\tilde c$ must have $v$ as an endpoint, since if $k$ is an edge-color and $e$ is the $k$-colored edge, and $v$ is not an endpoint of $e$, there would be  a $k$-colored robot on $v$ (and hence $\tilde c$ would be valid).

Therefore, every invalid $i$-edge-$r$-tuple is of the form $\tilde c=(x_1,x_2,\dots,x_r),$ where there is some fixed vertex $v$ where $x_k=v$ when $x_k$ is a vertex, and $x_k$ has $v$ as an endpoint if $x_k$ is an edge. Thus for each $v\in V(G),$ we get an invalid $i$-edge-$r$-tuple $\tilde c$ by choosing $i$ of the $r$ components for the edges, then for each of these components, choosing one of the $\deg(v)$ edges which have $v$ as an endpoint. The remaining components of $\tilde c$ must be equal to $v.$ So, for each fixed vertex $v,$ we get $\binom{r}{i}\deg(v)^i$ invalid $i$-edge-$r$-tuples in which $v$ appears.

Furthermore, for $i<r,$ any $i$-edge-$r$-tuple must contain at least one vertex, so that in this case, distinct vertices give distinct invalid $i$-edge-$r$-tuples, so, the total number of invalid $i$-edge-$r$-tuples is 
\begin{equation}\label{eqn:invalid1}
    \sum_{v\in V(G)}\binom{r}{i}\deg(v)^i.
\end{equation}
Subtracting (\ref{eqn:invalid1}) from (\ref{eqn:total}) gives the result for $i<r.$

For $i=r$, any invalid $r$-edge-$r$-tuple $\tilde c$ must be of the form $\tilde{c}=(x_1,\dots,x_r),$ where each $x_k$ is an edge, with a common endpoint $v.$ So for each vertex $v,$ we must count the number of $r$-tuples of edges where each edge has $v$ as an endpoint. As in the previous case, there are $\deg(v)$ choices for each of the $r$ edges, so there are a total of $\deg(v)^r$ invalid $r$-edge-$r$-tuples in which all edges have $v$ as an endpoint. Consider the sum 
\begin{equation}\label{eqn:double}
    \sum_{v\in V(G)}\deg(v)^r.
\end{equation}

Since an invalid $r$-edge-$r$-tuple consists exclusively of edges, for any edge $e$ with endpoints $v_1$ and $v_2$, the invalid $r$-edge-$r$-tuple $\tilde c = (e,e,\dots,e)$ is double-counted in (\ref{eqn:double}): once in the term corresponding to $v_1$ and once in the term corresponding to $v_2.$  If $\tilde c$ consists of at least two different edges, then the edges in $\tilde c$ only have one common endpoint, so $\tilde c$ is only counted once. So, the only invalid $r$-edge-$r$-tuples which are double-counted are of the form $(e,e,\dots,e).$ Since there are $m$ edges in $G$, there are $m$ $r$-edge-$r$-tuples which are double-counted.

Therefore, the total number of invalid $r$-edge-$r$-tuples is 
\begin{equation}\label{eqn:Invalidrr}
    \left(\sum_{v\in V(G)}\deg(v)^r\right)-m.
\end{equation}

Subtracting (\ref{eqn:Invalidrr}) from (\ref{eqn:total}) (with $i=r$) gives the result.

The fact that there are no $i$-cells for $i>r$ follows from the fact that for each color $k$, there is at most one $k$-colored edge in any given cell $c$, so there are at most $r$ edges in total in $c$.
\end{proof}

As an example, \cref{tab:CellCountsB} gives the cell counts for $\Sr(T_4)$ and $\Sr(K_5)$ for various $\vec r.$

\begin{table}[h!]
\begin{subtable}{.5\linewidth}
    \centering
    \raisebox{0cm}{\scalebox{.9}{$\begin{array}{c|c|c|c|c}
         &  \multicolumn{4}{c}{\usebox\GraphA}\\
         \hline
         \vec r&(3,3) & (3,3,3) & (3,3,3,3) & (3,3,3,3,3)\\
         \hline
         \text{0-cells}  &12 & 60 & 252& 1020\\
         \hline
         \text{1-cells} & 12 & 126 & 744 & 3180\\
         \hline
         \text{2-cells} & 0 & 72 & 792 & 5640\\
         \hline
         \text{3-cells} & 0 &   0 & 312 & 4020\\
         \hline
         \text{4-cells} & 0 &  0  &  0   & 1200\\
         \hline
         \text{5-cells} & 0 & 0   &   0  & 0
    \end{array}$}}
\end{subtable}
\hspace*{-5mm}
\begin{subtable}{.5\linewidth}
    \centering
    \scalebox{.9}{$\begin{array}{c|c|c|c|c}
         &  \multicolumn{4}{c}{\usebox\GraphD}\\
         \hline
         \vec r&(4,4) & (4,4,4) & (4,4,4,4) & (4,4,4,4,4)\\
         \hline
         \text{0-cells} & 20 & 120 & 620& 3120\\
         \hline
         \text{1-cells} & 60 & 690 & 4920 & 31150\\
         \hline
         \text{2-cells} & 30 & 1260 & 14520 & 124200\\
         \hline
         \text{3-cells} &  0  & 690  & 18720 & 246800\\
         \hline
         \text{4-cells} &  0  &   0   & 8730  & 243600\\
         \hline  
         \text{5-cells} &  0  &   0   &    0   & 94890
    \end{array}$}
\end{subtable}
\caption{Cell counts of $\Sr(T_4)$ (left) and $\Sr(K_5)$ (right)}
\label{tab:CellCountsB}
\end{table}

In this case, the cell counts alone do not determine the homotopy type of $\Sr(G)$ when there are $i$-cells for $i\ge 2,$ so the cell counts do not give a complete description of the structure of $\Sr(G).$ One tool for determining more information about the structure of $\Sr(G)$ is the \textit{homology groups} of $\Sr(G)$. The homology groups of $\Sr(T_n)$ will be studied in a forthcoming manuscript by the second author.

\bibliographystyle{plain}
\bibliography{papers}

\end{document}